\documentclass[a4paper,english,reqno,10pt]{amsart}
\usepackage[utf8]{inputenc}
\usepackage[T1]{fontenc}
\usepackage{lmodern}
\usepackage{amssymb}
\usepackage{amsmath}
\usepackage{amsthm}
\usepackage{mathrsfs}
\usepackage{stmaryrd}
\usepackage{enumitem}

\usepackage[active]{srcltx}

\usepackage{graphics}
\usepackage{color}
\usepackage[all]{xy} 
\usepackage{dsfont}
\usepackage{dsfont}

\usepackage{calligra}

\numberwithin{equation}{section}
\DeclareMathAlphabet{\mathpzc}{OT1}{pzc}{m}{it}

\DeclareSymbolFont{EulerScripta}{U}{euf}{m}{n}
\SetSymbolFont{EulerScripta}{bold}{U}{euf}{b}{n} 
\DeclareSymbolFontAlphabet\matheufm{EulerScripta}

\DeclareSymbolFont{EulerScriptb}{U}{eur}{m}{n}
\SetSymbolFont{EulerScriptb}{bold}{U}{eur}{b}{n} 
\DeclareSymbolFontAlphabet\matheurm{EulerScriptb}

\DeclareSymbolFont{EulerScriptc}{U}{eus}{m}{n}
\SetSymbolFont{EulerScriptc}{bold}{U}{eus}{b}{n} 
\DeclareSymbolFontAlphabet\matheusm{EulerScriptc}
\newcommand\eusm{\matheusm}

\usepackage{xspace}

\usepackage{comment}

\usepackage{hyperref}

\swapnumbers
\theoremstyle{plain}
\newtheorem{thm}[subsubsection]{Theorem}
\newtheorem*{introthm}{Theorem}
\newtheorem*{introcorollary}{Corollary}
\newtheorem{proposition}[subsubsection]{Proposition}
\newtheorem{propositiondefinition}[subsubsection]{Proposition-Definition}

\newtheorem{lemma}[subsubsection]{Lemma}
\newtheorem{corollary}[subsubsection]{Corollary}

\theoremstyle{definition}
\newtheorem{definition}[subsubsection]{Definition}
\newtheorem{example}[subsubsection]{Example}
\theoremstyle{remark}
\newtheorem{remark}[subsubsection]{Remark}
\newtheorem{notation}[subsubsection]{Notation}

\newcommand{\name}[1]{#1}

\newdir{ >->}{{}*!/-10pt/@{>->}}
\newdir{ >}{{}*!/-10pt/@{>}}
\newdir{ (}{{}*!/-5pt/@^{(}}
\newdir{ )}{{}*!/-5pt/@_{(}}

\newcommand{\eq}[2]{\begin{equation}\label{#1}#2 \end{equation}}
\newcommand{\ml}[2]{\begin{multline}\label{#1}#2 \end{multline}}
\newcommand{\mlnl}[1]{\begin{multline*}#1 \end{multline*}}

\newcommand{\Div}{{\rm div}}

\newcommand{\xr}[1] {\xrightarrow{#1}}

\newcommand{\inj}{\hookrightarrow}
\newcommand{\surj}{\twoheadrightarrow}

\newcommand{\ul}{\underline}

\newcommand{\mc}[1]{\mathcal{#1}}
\newcommand{\id}{{\rm id}}
\newcommand{\Res}{{\rm Res}}

\newcommand{\Fil}{{\rm Fil}}
\newcommand{\CH}{{\rm CH}}

\newcommand{\Cor}{{\rm Cor}}
\newcommand{\Nis}{{\rm Nis}}

\newcommand{\HZar}{\mathrm{H}_{\mathrm{Zar}}}
\newcommand{\KMsheaf}{\mathcal{K}^{\mathrm{M}}}

\newcommand{\PST}{{\bf PST}}
\newcommand{\NST}{{\bf NST}}

\newcommand{\Nm}{{\rm Nm }}
\newcommand{\Image}{\mathop{\rm Image}}
\newcommand{\Coker}{\mathop{\rm Coker}}

\newcommand{\Ker}{\mathop{\mathrm{Ker}}}

\newcommand{\DMeffm}{\mathbf{DM}^{\mathrm{eff}}_{-}}

\newcommand{\DM}{\mathbf{DM}}
\newcommand{\RF}{\mathbf{RF}}
\newcommand{\MF}{\mathbf{MF}}
\newcommand{\PT}{\mathbf{PT}}
\newcommand{\NT}{\mathbf{NT}}
\newcommand{\LT}{\mathrm{LT}}
\newcommand{\T}{\mathrm{T}}

\newcommand{\MFsp}{\mathbf{MFsp}}
\newcommand{\LMFsp}{\mathbf{LMFsp}}
\newcommand{\tH}{\mathsf{H}}

\newcommand{\Reg}{\Regone}

\newcommand{\SmCor}{\rm SmCor}
\newcommand{\KM}{\mathrm{K}^\mathrm{M}}
\newcommand{\Mod}{\mathrm{Mod}}

\newcommand{\ra}{\rightarrow}

\newcommand{\Tr}{\mathrm{Tr}}
\newcommand{\pt}{\mathrm{pt}}

\newcommand{\Hom}{\mathrm{Hom}}
\newcommand{\colim}{\mathop{\mathrm{colim}}}
\newcommand{\Spec}{\mathop{\mathrm{Spec}}}
\newcommand{\op}{\mathrm{op}}

\newcommand{\otimesM}{\overset{\mathrm{M}}{\otimes}}

\newcommand{\length}{\mathrm{lg}}
\newcommand{\Bil}{\mathrm{Bil}}
\newcommand{\Lin}[1]{#1-\mathrm{Lin}}
\newcommand{\Tor}{\mathrm{Tor}}

\renewcommand{\mathcal}{\mathscr}

\newcommand{\sC}{{\mathcal C}}
\newcommand{\sD}{{\mathcal D}}

\newcommand{\sF}{{\mathcal F}}
\newcommand{\sG}{{\mathcal G}}
\newcommand{\sH}{{\mathcal H}}

\newcommand{\sK}{{\mathcal K}}
\newcommand{\sL}{{\mathcal L}}
\newcommand{\sM}{{\mathcal M}}
\newcommand{\sN}{{\mathcal N}}
\newcommand{\sO}{{\mathcal O}}
\newcommand{\sP}{{\mathcal P}}

\newcommand{\sR}{{\mathcal R}}
\newcommand{\sS}{{\mathcal S}}

\newcommand{\A}{{\mathbb A}}

\newcommand{\G}{{\mathbb G}}

\newcommand{\N}{{\mathbb N}}
\renewcommand{\P}{{\mathbb P}}
\newcommand{\Q}{{\mathbb Q}}

\newcommand{\Z}{\mathbb{Z}}

\newcommand{\fM}{\mathfrak{M}}

\newcommand{\fm}{\mathfrak{m}}
\newcommand{\fn}{\mathfrak{n}}

\newcommand{\HINis}{{\HI_\Nis}}
\newcommand{\HI}{\mathbf{HI}}

\newcommand{\otimesHtr}{\otimes_{\HINis}}

\newcommand{\Forget}{o}

\newcommand{\Regone}{\mathrm{Reg}^{\leqslant 1}}
\newcommand{\RegCon}{\mathrm{RegCon}^{\leqslant 1}}

\newcommand{\Sm}{\mathrm{Sm}}
\newcommand{\Kgeo}{\mathrm{K}^{\mathrm{geo}}}

\begin{document}

\title{K-Groups of Reciprocity Functors}
\author{Florian {\scshape{Ivorra}}}
\address{Florian Ivorra: Institut de recherche math\'ematique de Rennes\\ UMR 6625 du CNRS\\ Universit\'e de Rennes 1\\
Campus de Beaulieu\\
35042 Rennes cedex (France)}
\email{florian.ivorra@univ-rennes1.fr}
\author{Kay R\"ulling}
\address{Kay R\"ulling: Institut f\"ur Mathematik\\Freie Universit\"at Berlin \\
14195 Berlin  (Germany)}
\email{kay.ruelling@fu-berlin.de}

\thanks{The first author acknowledges support from the DAAD (Deutscher Akademischer Austausch Dienst) during 
the preparation of this work and thanks \name{M. Levine} for providing an excellent working environment and making 
his stay at the University Duisburg-Essen possible. The second author was supported by the SFB/TR 45 
``Periods, moduli spaces and arithmetic of algebraic varieties'' of the DFG, by the ERC Advanced Grant 226257 
and thanks the first author for an invitation to the University of Rennes in 2010}

\begin{abstract}

In this work we introduce reciprocity functors, construct the associated K-group functor of a family of reciprocity functors, which itself is a reciprocity functor, and compute it in several different cases. It may be seen as a first attempt to get close to the  notion of reciprocity sheaves imagined by B. \name{Kahn}.
Commutative algebraic groups, homotopy invariant Nisnevich sheaves with transfers, cycle modules or K\"ahler differentials are examples of reciprocity functors.
As commutative algebraic groups do, reciprocity functors are equipped with symbols and satisfy a reciprocity law for curves.
\end{abstract}
\maketitle

\setcounter{tocdepth}{2}
\tableofcontents

\section*{Introduction}
Let $F$ be a perfect field. In this work we introduce a notion of reciprocity functors. Early in the 90's, \name{B. Kahn} already suggested to use the local symbol of \name{M. Rosenlicht} and \name{J.-P. Serre} \cite{Rosenlicht,SerreGACC}  for smooth commutative algebraic groups in order to develop a theory which contains algebraic groups and homotopy invariant Nisnevich sheaves with transfers, 
see e.g. \cite{KahnRF}. 
Our approach is inspired by his 
idea, and in this work reciprocity functors are introduced as functors  defined over finitely generated field extensions of $F$ and regular curves over them. Besides the above examples 
we also show that the absolute K\"ahler differentials are reciprocity functors in our sense.\par 
Given reciprocity functors $\sM_1,\ldots,\sM_n$, our main construction is a reciprocity functor $\T(\sM_1,\ldots,\sM_n)$ that we call the K-group of $\sM_1,\ldots, \sM_n$, although it is much more than a group, it is a reciprocity functor. This construction is related to the K-group associated by M. Somekawa in \cite{Somekawa} with a family of semi-Abelian varieties and its variants introduced in \cite{RS,Akhtar, Hiranouchi}. But notice that even in the case the Somekawa type K-group of reciprocity functors $\sM_1,\ldots, \sM_n$ is defined there is no direct comparison with the reciprocity functor $\T(\sM_1,\ldots, \sM_n)$ we construct. 
Nevertheless, we compute this K-group of reciprocity functors in several cases and comparing these results
with the computation of the corresponding Somekawa type K-groups in \cite{RS,Akhtar,KY}
we conclude that these groups agree in these cases. As a new example we show, as expected by \name{B. Kahn}, that 
$$\T(\G_a,\underbrace{\G_m,\ldots, \G_m}_{n\text{ copies}})$$ evaluated at a characteristic zero field $k$ computes
the $n$-th group of absolute K{\"a}hler differentials of $k$, $\Omega^n_{k/\Z}$. 
In a recent preprint T. Hiranouchi \cite{Hiranouchi} obtains a similar result using a Somekawa type K-group.

\section*{\itshape{Detailed description of this work}}
\par
 Let $F$ be a perfect field and $S=\Spec F$. A point over $S$ or - for short - an $S$-point is the spectrum of a finitely generated field extension and
we denote by $\Regone$ the category with objects the regular $S$-schemes of dimension $\leqslant 1$, which are separated and of finite type over some $S$-point. As in \cite{VoDM} we can define the category $\Regone\Cor$, which has the same objects as $\Regone$ but
finite correspondences as morphisms.\par
Inspired by \name{B. Kahn}'s idea to use Rosenlicht-Serre's treatment of local symbols for smooth connected and commutative algebraic groups over an algebraically closed field 
in order to develop a theory that contains algebraic groups and homotopy invariant Nisnevich sheaves with transfers, we define a {\em reciprocity functor} to be a presheaf $\sM$ of Abelian groups on $\Regone\Cor$, which satisfies the following conditions:
\begin{enumerate}
\item[(Nis)] $\sM$ is a sheaf in the Nisnevich topology on $\Regone$.
\item[(FP)] For all connected $X\in \Reg$ with generic point $\eta$, the group $\sM(\eta)$ is the stalk in the generic point of $\sM$ viewed as a Zariski sheaf on $X$.
\item[(Inj)] For all connected $X\in \Reg$ and all non-empty open subsets $U\subset X$, the restriction map $\sM(X)\inj \sM(U)$ is injective.
\item[(MC)] For all regular projective and connected curves $C\in \Reg$,  all non-empty open subsets $U\subset C$ and sections $a\in \sM(U)$, there exists an effective divisor $\fm$ with support equal to $C\setminus U$
         and such that
         \[\sum_{P\in U} v_P(f) \Tr_{P/x_C} s_P(a)=0,\]
         where $f$ is any non-zero element in the function field of $C$, which is congruent to 1 modulo $\fm$ 
      (i.e. $\Div(f-1)\geqslant \fm$ on $C\setminus U$), $v_P$ is the discrete valuation associated with the closed point $P\in C$,
         $s_P: \sM(U)\to \sM(P)$ is the specialization map, which is simply given
         by the pullback along the natural inclusion $P\inj U$, $x_C=\Spec \, H^0(C,\sO_C)$, and $\Tr_{P/x_C}:\sM(P)\to \sM(x_C)$ is the pushforward along the finite map $P\to x_C$.
\end{enumerate}
 The condition (MC) is nothing but the modulus condition from \cite{SerreGACC}.
As suggested by \name{B. Kahn}, examples of reciprocity functors are (see \S2):
\begin{itemize}
\item Smooth commutative algebraic groups over $S$. 
                 (This essentially follows from a theorem of M. Rosenlicht \cite{Rosenlicht}.)
\item Homotopy invariant Nisnevich sheaves with transfers in the sense of \cite{VoDM}.
\item The Nisnevich sheafification of $X\mapsto H^1_{\rm \'et}(X, \Q/\Z)$. 
\item Rost's cycle modules. (This can be seen as a consequence of the example above and \cite{Deglise},
         but also follows directly from the axioms of cycle modules.)     
\end{itemize}
As a particular case, we also obtain that 
\begin{itemize}
 \item for any smooth projective variety $X$ over $S$ of pure dimension $d$ and any $n\geqslant 0$
       the functor on $S$-points $x\mapsto\CH^{d+n}(X_x,n)$ defines a reciprocity functor, where $\CH^{d+n}(X_x,n)$ denotes Bloch's higher Chow groups. 
\end{itemize}
Futhermore for all $n\geqslant 0$,
\begin{itemize}
\item absolute K\"ahler differentials, $x\mapsto \Omega^n_{x/\Z}$, define a reciprocity functor.
\end{itemize}

As already observed by \name{B. Kahn}, one can use the same computations as in \cite{SerreGACC} to show that a reciprocity functor $\sM$ has local symbols. More precisely, if $C$ is a regular projective and connected curve over some $S$-point $x$ with generic point $\eta$, then for all closed points $P\in C$
there exists a bilinear map 
\[(-,-)_P: \sM(\eta)\times \G_m(\eta)\to \sM(x),\]
which is continuous, when $\sM(\eta)$ and $\sM(x)$ are equipped with the discrete topology and $\G_m(\eta)$ with the $\fm_P$-adic toplogy ($\fm_P$ the maximal ideal in the local ring of $C$ at $P$)
and satisfies $(a,f)_P= v_P(f)\Tr_{P/x}(s_P(a))$, for all $a\in \sM_{C,P}$, and the reciprocity law $\sum_{P\in C} (a,f)_P=0$.\par
These local symbols provide an increasing and exhaustive filtration $\Fil^\bullet_P \sM(\eta)$, where $\Fil^0_P\sM(\eta)=\sM_{C,P}$ and $\Fil^n_P\sM(\eta)$ is the subgroup consisting
of the elements $a\in \sM(\eta)$ such that $(a, 1+\fm_P^n)_P=0$. 
\section*{\itshape{Main results}}
Now let $\sM_1,\ldots,\sM_n$ and $\sN$ be reciprocity functors. Then a $n$-linear map of reciprocity functors $\Phi: \sM_1\times \cdots\times \sM_n\to \sN$
is a $n$-linear map of sheaves, which is compatible with pullback, satisfies a projection formula, and the following condition
\[(\mathrm{L3})\quad \Phi(\Fil^r_P\sM_1(\eta)\times \cdots\times \Fil^r_P\sM_n(\eta))\subset
 \Fil_P^r\sN(\eta),\]
for all regular projective curves $C$ with generic point $\eta$ and $P\in C$ a closed point and all positive integers $r\geqslant 1$.
We denote by $\Lin{n}(\sM_1,\ldots, \sM_n;\sN)$ the group of $n$-linear maps as above.
Then the main theorem of this article is the following:
\begin{introthm}[see Theorem \ref{thm-tensorRF}]
The functor $\RF\to (\text{Abelian groups})$, $\sN\mapsto \Lin{n}(\sM_1,\ldots, \sM_n;\sN)$ is representable by a reciprocity functor
\[\T(\sM_1,\ldots,\sM_n).\]
\end{introthm}
We call $\T(\sM_1,\ldots,\sM_n)$ the K-group of $\sM_1,\ldots, \sM_n$, although it is much more than a group, it is a reciprocity functor.
We would like to call this a tensor product, unfortunately it is not clear whether  associativity is satisfied (one reason is the condition (L3)).
But we have commutativity, compatibility with direct sums and the constant reciprocity functor $\Z$ is a unit object.

In \cite{KY},  \name{B. Kahn} and \name{T. Yamazaki} prove the following isomorphism
\begin{equation}\label{KahnYamIso}
\mathrm{K}(F,\sF_1,\ldots,\sF_n)\simeq\Hom_{\DMeffm}(\Z,\sF_1\otimes\cdots\otimes\sF_n )\tag{$\ast$}
\end{equation}
where the $\sF_i$'s are homotopy invariant Nisnevich sheaves with transfers, the left hand side is a Somekawa type K-group and the right hand side  is the group of morphisms in Voevodsky's category of effective motivic complexes. The analogy of Somekawa's K-groups with our K-groups of reciprocity functors suggested the following theorem:


\begin{introthm}[see Theorem \ref{CompHIA}]
Let $\sF_1,\ldots,\sF_n\in\HINis$ be homotopy invariant Nisnevich sheaves with transfers. There exists a canonical and functorial isomorphism of reciprocity functors
\begin{equation*}
\T(\sF_1,\ldots,\sF_n)\xrightarrow{\sim} (\sF_1\otimesHtr\cdots\otimesHtr \sF_n).
\end{equation*}
\end{introthm}
Let us emphasize that the definition of the K-group of reciprocity functors is different from the Somekawa type one. In particular the above theorem does not follow from the isomorphism (\ref{KahnYamIso}). Let us also point out that it is a priori not clear how the K-group of reciprocity functors associated with semi-Abelian varieties compare with the Somekawa K-groups. It follows from Kahn-Yamazaki's isomorphism (\ref{KahnYamIso}) and the above Theorem that $\T(G_1,\ldots, G_n)(S)$ with $G_i$ semi-Abelian varieties equals Somekawa's K-group $\mathrm{K}(F,G_1,\ldots,G_n)$.

In the same way as in \cite{KY} we obtain as a corollary:

\begin{introcorollary}[see Corollary \ref{cor-tensor-chow}]
Let $X_1,\ldots,X_n$ be smooth projective schemes over $S$  of pure dimension $d_1,\ldots,d_n$ and $r\geqslant 0$ an integer.
Then for all $S$-points $x$ we have an isomorphism
\[\T(\CH_0(X_1),\ldots, \CH_0(X_n),\G_m^{\times r})(x)\cong \CH^{d+r}(X_{1,x}\times_x\ldots\times_x X_{n,x}, r)\]
where $d=d_1+\cdots+d_n$ and $\G_m^{\times r}$ denotes the $r$-tuple $(\G_m,\ldots,\G_m)$
\end{introcorollary}

Using a variant of Somekawa's K-groups on the left hand side the above isomorphism was proven by \name{W. Raskind} and \name{M. Spie\ss} in \cite{RS} (case $r=0$) and by \name{R. Akhtar} in \cite{Akhtar} (case $r\geqslant 0$).
We can use this as in \cite{Somekawa} or \cite{RS} to determine the kernel of the Albanese map of a product of smooth projective and connected curves over $S=\Spec F$, where $F$ is algebraically closed.

Further we can also calculate by hand:
 
\begin{introthm}[see Theorem \ref{thm-Milnor=tensor}]
For all $n\geqslant  1$ and all $S$-points $x=\Spec \, k$, there is a canonical isomorphism
\[\T(\G_m^{\times n})(x) \xr{\simeq} \KM_n(k), \]
where $\KM_n(k)$ denotes the $n$-th Milnor $K$-group of $k$.
\end{introthm}

Observe that combining the corollary above in the case $n=1$ and $X_1=S$ with this theorem, we get the \name{Nesterenko}-\name{Suslin} theorem $\CH^n(k,n)\cong \KM_n(k)$.
But actually to prove the above theorem we use more or less the same methods as Yu. \name{Nesterenko} and \name{A. Suslin}, so it is not really a new proof of their theorem.

\begin{introthm}[see Theorem \ref{thm-tensorRF=KD}]
Assume $F$ has characteristic zero. Then there is an isomorphism for all $S$-points $x$ 
\[\theta: \Omega^n_{x/\Z}\xr{\simeq} \T(\G_a,\G_m^{\times n})(x). \]
\end{introthm}

Combining the above theorem with the theorem of \name{S. Bloch} and \name{H. Esnault} in \cite{BE} which identifies $\Omega^{n-1}_{x/\Z}$ with the additive higher Chow groups of $x$ of level $n$ (and with modulus 2) , we obtain
\[\T(\G_a,\G_m^{\times n-1})(x)\cong {\rm TCH}^n(x,n,2).\]

This result does not hold in positive characteristic. The reason is essentially that in positive characteristic the algebraic group (or the reciprocity functor) $\G_a$ has more endomorphisms than
in characteristic zero, namely the absolute Frobenius comes into the game. This forces  the following:

\begin{introcorollary}[see Corollary \ref{cor-Kaehler-posChar}]
Assume $F$ has characteristic $p>0$. Then for all $S$-points $x$ we have a surjective morphism 
\[\Omega^n_{x/\Z}/B_\infty\surj \T(\G_a,\G_m^{\times n})(x)\]
and the following commutative diagram 
\[\xymatrix{\Omega^n_{x/\Z}/B_\infty\ar[r]^{C^{-1}}\ar@{->>}[d] & \Omega^n_{x/\Z}/B_\infty\ar@{->>}[d]\\
            \T(\G_a,\G_m^{\times n})(x)\ar[r]^{F\otimes\id} & \T(\G_a,\G_m^{\times n})(x),  }\]
where $F:\G_a\to \G_a$ is the absolute Frobenius, $C^{-1}:\Omega^n_{x/\Z}\to \Omega^n_{x/\Z}/d\Omega^{n-1}_{x/\Z}$ is the inverse Cartier operator and 
$B_\infty$ is the union over $B_n$, where $B_1=d\Omega^{n-1}_{x/\Z}$ and $B_n$ is the preimage of $C^{-1}(B_{n-1})$ in $\Omega^n_{x/\Z}$ for $n\geqslant 2$.
\end{introcorollary}

We don't know whether the above map is an isomorphism. 
The above theorem and the corollary should be compared to \cite{Hiranouchi}, where \name{T. Hiranouchi} defines a variant of Somekawa's K-groups for perfect fields in which one can also plug in $\G_a$ (and also Witt group schemes) 
and in \cite[Theorem 3.6]{Hiranouchi} he proves $\Omega^n_{F/\Z}\cong \mathrm{K}(F; \G_a, \G_m^{n})$. (Notice that in positive characteristic this is just a vanishing result for the K-group, which for our group also follows from the above corollary.)

The following theorem was suggested to us by \name{B. Kahn}.

\begin{introthm}[see Theorem \ref{thm-van-unipotent-gps}]
Assume ${\rm char}(F)\neq 2$. Let $\sM_1,\ldots, \sM_n$ be reciprocity functors. 
Then 
\[\T(\G_a,\G_a,\sM_1,\ldots,\sM_n)=0.\]
\end{introthm}

\section*{\itshape{Acknowledgment}} 
The preprint \cite{KahnRF} was not available to the authors at the time this paper was written. However they had knowledge of this work through handwritten notes taken by the first author on September the 28th 2006 when B. Kahn sketched him the ideas of \cite{KahnRF} and proposed him to have a look at  problems in that direction.  The authors thank him heartily for discussing his ideas with them and for pointing out a mistake in an earlier version of this article. We also thank T. Hiranouchi for a careful reading of our article.
We are grateful to the anonymous referee for his very useful comments and questions.

\section*{\itshape{Conventions and notations}}\label{not-general}
We fix the following notation and conventions.
\begin{enumerate}
\item[(1)]  $F$ is a perfect field and $S=\Spec F$ is our base point. A point $x$ over $S$ is a morphism $\Spec k\to S$, where $k$ is a finitely generated field extension over $F$.
            By abuse of notation we will also say that $x=\Spec k$ is a point over $S$, meaning that it comes with a fixed morphism to $S$ and (by even more abuse) we will simply say {\em $x$ is an $S$-point}.
            If not said differently cartesian products of schemes are over $S$.
 \item[(2)] $R$ is a commutative ring with $1$ and we denote by $(R-\rm{mod})$ the category of $R$-modules.
\item[(3)]  By a curve over an $S$-point $x$ we mean a pure 1-dimensional scheme, which is separated and of finite type over $x$. Points on a curve $C$ over an $S$-point $x$, which are denoted by capital letters like 
$P,Q, P', \ldots\in C$ are always meant to be closed points; if $C$ is irreducible its generic point will be denoted by $\eta_C$ or just by $\eta$ if no confusion can arise.
\item[(4)] If $X$ is a scheme and $x\in X$ is a point, we denote by $\kappa(x)$ its residue field, if $X$ is integral and $\eta$ is its generic point then we also write
            $\kappa(X)=\kappa(\eta)$ for its function field.
\item[(5)] If $f: X\to Y$ is a morphism of $S$-schemes we denote by $\Gamma_f\subset X\times Y$ its graph and by $\Gamma_f^t$ its transpose. 
\item[(6)] If $C$ is a regular connected curve over a field and $P\in C$ is a closed point, then we set for a positive integer $n$
\[U_P^{(n)}:= 1+\fm_P^n,\]
      where $\fm_P\subset\sO_{C,P}$ denotes the maximal ideal. Further we denote by $v_P: \kappa(C)\to \Z\cup\{\infty\}$ the normalized discrete valuation associated with $P$.

          If $D\to C$ is a finite and surjective morphism between regular curves over a field, sending a closed point $Q\in D$ to $P\in C$, we will denote by $e(Q/P)$ and
           $f(Q/P)=[\kappa(Q):\kappa(P)]$ the ramification index and the inertia degree, respectively.
\item[(7)]  If $X$ is an integral scheme we denote by $\tilde{X}$ the normalization of $X$ and by $\nu_X: \tilde{X}\ra X$ the corresponding map.
\end{enumerate}

\section{Reciprocity Functors}

\subsubsection{}\label{some-cats}
We introduce the following categories.
\begin{enumerate}\label{defn-pts}
 \item Let $(\pt/S)$ be the category with objects the $S$-points $x$ and morphisms the $S$-morphisms $x\to y$ of $S$-points.
 \item Let $(\pt/S)_*$ be the category with objects the $S$-points and morphisms the {\em finite} $S$-morphisms of $S$-points.
 \item Let $(\mc{C}/S)$ be the category with objects the regular and connected curves $C$, which are
         projective over some $S$-point, morphisms are dominant morphisms of $S$-schemes.
 \item Let $\Reg_S=\Reg$ be the category with objects the regular $S$-schemes of dimension $\leqslant 1$, which are separated and of finite type over some $S$-point, morphisms are morphisms of $S$-schemes.
 \item Let $\RegCon$ (resp. $\RegCon_*$) be the category with objects the connected regular $S$-schemes of dimension $\leqslant 1$ which are separated and of finite type over some $S$-point, morphisms are 
         morphisms of $S$-schemes (resp. finite flat morphism of $S$-schemes).
\end{enumerate}
Notice that $(\pt/S)$ is a subcategory of $\Reg$ and that we have a faithful functor $(\mc{C}/S)\to \Reg$. Further
if $U$ is in $\Reg$, then any point $y\in U$  naturally defines an $S$-point. If $U\in \Reg$ is connected,  1-dimensional
and of finite type over some $S$-point $x$, then it admits a unique (up to isomorphism) open immersion over $x$ 
into a curve $C\in (\sC/S)$.
For a curve $C\in (\sC/S)$, we denote 
\eq{some-cats1}{x_C:= \Spec H^0(C,\sO_C)\in (\pt/S),}
Notice that $C\to x_C$ is projective and geometrically connected.

\subsection{Correspondences in dimension at most 1}

\begin{definition}\label{defn-corr}
Let $X$ and $Y$ be in $\Reg$. Then we denote by $\Cor_R(X,Y)=\Cor(X,Y)$ the free $R$-module generated by symbols $[V]$ with $V$ an integral closed subscheme of $X\times Y$,
which is finite and surjective over a connected component of $X$.
Elements in $\Cor(X,Y)$ are called correspondences from $X$ to $Y$ and correspondences of the form $[V]$ with $V$ as above are called  elementary correspondences.
\end{definition}

\begin{lemma}\label{lem-vanishing-of-Tor}
Let $X$, $Y$ and $Z$ be in $\Reg$ and let $[V]\in \Cor(X,Y)$, $[W]\in \Cor(Y,Z)$ be elementary correspondences. 
Let $(V\times Z)\cap (X\times W)$ be the intersection in $X\times Y\times Z$ and let $T$ be an irreducible component with generic point $\eta$.
Then $T$ as well as its image in $X\times Z$ are finite and surjective over a connected component of $X$ and 
\[\Tor_i^{\sO_{X\times Y\times Z,\eta}}(\sO_{V\times Z,\eta},\sO_{X\times W,\eta})=0\quad \text{for all } i\geqslant 1.\] 
\end{lemma}

\begin{proof}
This is well known. We include a proof for the reader's convenience. For the first statement see e.g. \cite[1.]{MVW}. 
  To prove the vanishing of the higher $\Tor$'s we can assume that the schemes $V,W,X,Y,Z$ are all integral of dimension $\leqslant 1$ and
hence are CM. We can spread them out to integral schemes $\bar{V}, \bar{W}, \bar{X}, \bar{Y}, \bar{Z}$, which are separated and of finite type over $S$
such that $\bar{X}$, $\bar{Y}$, $\bar{Z}$ are smooth over $S$ (this is possible since $S$ is perfect) and $\bar{V}\subset \bar{X}\times\bar{Y}$ and $\bar{W}\subset \bar{Y}\times \bar{X}$
are closed subschemes, which are finite and surjective over the first factor and which we can assume to be CM (by \cite[Cor. (6.4.2)]{EGAIV2} applied to $V\to \bar{V}$, $W\to \bar{W}$).
Furthermore we can assume that $T$ spreads out to an integral closed subscheme $\bar{T}\subset (\bar{V}\times\bar{Z})\cap (\bar{X}\times \bar{W})$, which is finite and surjective over $\bar{X}$.
Since it clearly suffices to prove the vanishing of the over-lined-Tor's we can assume in the following that $V,W,X,Y,Z$ are integral, separated and of finite type over $S$ (of some positive dimension)
with $X,Y,Z$ smooth over $S$ and $V,W$ CM.
We set $A:=\sO_{X\times Y\times Z,\eta}$, $A_V:=\sO_{V\times Z,\eta}$, $A_W:=\sO_{X\times W,\eta}$.
Since $A$ is regular the kernel of 
\[B:=A\otimes_F A\to A,\quad a\otimes b\mapsto ab\]
is generated by a regular sequence $\ul{t}=t_1,\ldots, t_n$, with $n:=\dim A$. Hence the Koszul complex $K^B_\bullet(\ul{t})$ is a free resolution of the $B$-module $A$.
By \cite[Cor. (6.7.3)]{EGAIV2} the $B$-algebra $A_V\otimes_F A_W$ is CM of dimension 
\[\dim (A_V\otimes_F A_W)=\dim A_V+\dim A_W=\dim Y+ \dim Z=\dim A=n\]
and 
\[(A_V\otimes_F A_W)/(\ul{t})= (A_V\otimes_F A_W)\otimes_B A= A_V\otimes_A A_W=\sO_{V\times_Y W,\eta}\]
has finite length. Thus $\ul{t}$ is a system of parameters for the CM $B$-module $A_V\otimes_F A_W$ and we obtain for $i\geqslant 1$
\begin{align}
\Tor_i^A(A_V,A_W) &\cong \Tor^B_i(A_V\otimes_F A_W, A), &\text{see e.g. \cite[V,(2)]{SerreAL}},\notag\\
                 &=H_i(K^B_\bullet(\ul{t})\otimes_B (A_V\otimes_F A_W))\notag\\
                 &=0, &\text{by \cite[IV, Thm 3, iv)]{SerreAL}.} \notag
\end{align}
Hence the statement. 
\end{proof}

\begin{definition}\label{defn-catCor}
\begin{enumerate}
 \item Let $X$, $Y$ and $Z$ be in $\Reg$ and let $[V]\in \Cor(X,Y)$ and $[W]\in \Cor(Y,Z)$ be elementary correspondences.
       Then we define
          \[[W]\circ [V]:=\sum_{T\subset V\times_Y W} \length(\sO_{V\times_Y W,\eta_T})\cdot  p_{XZ*}[T],\]
        where $T$ runs over the irreducible components of $V\times_Y W$, $\eta_T$ is the generic point of $T$ and $p_{XZ}: V\times_Y W\to X\times Z$ is the natural map induced by projection.
         (Notice that $p_{XZ}$ is proper.)
      It follows from Lemma \ref{lem-vanishing-of-Tor} that $[W]\circ [V]$ is an element in $\Cor(X,Z)$ and coincides with the pushforward to $X\times Z$  of the intersection cycle 
      $[V\times Z]\cdot [X\times W]$ (in the sense of \cite[V]{SerreAL}).
        Hence $\circ$ extends to a $R$-bilinear pairing
       \[\Cor(X,Y)\times \Cor(Y,Z)\to \Cor(X,Z), \quad (\alpha,\beta)\mapsto \beta\circ\alpha,\]
       which is associative and with the diagonal $\Delta\subset Y\times Y$ acting from the right (resp.  left) as identity on $\Cor(X,Y)$ (resp. $\Cor(Y,X)$).
\item We define the category $\Reg\Cor$ to be the category with objects the objects of $\Reg$ and morphisms given by correspondences. We denote by
      $\pt\Cor$ the full subcategory with objects the finite disjoint unions of $S$-points.
\end{enumerate}
\end{definition}

Notice that $\Reg\Cor$ is a $R$-linear additive category and that the graph gives in the usual way a faithful and essentially surjective functor $\Reg\to \Reg\Cor$.

\begin{lemma}\label{lem-basic-corr-formulas}
\begin{enumerate}
 \item Let $X$ and $Y$ be in $\Reg$ and let $V$ be an integral scheme with a map $\pi:V\to X\times Y$ such that $V$ is finite and surjective over $X$.
      Let $\tilde{V}$ be the normalization of $V$ and $q_V:\tilde{V}\to X$ and $p_V: \tilde{V}\to Y$ be induced by the projection maps.
      Then $\tilde{V}\in \Reg$ and 
       \[\pi_*[V]=[\Gamma_{p_V}]\circ[\Gamma^t_{q_V}]\in \Cor(X,Y).\]
\item Let $f:X\to Y$ be a finite and surjective map in $\RegCon$. Then 
        \[[\Gamma_f]\circ [\Gamma_f^t]=\deg(X/Y)\cdot [\Delta_Y],\]
        with $\Delta_Y\subset Y\times Y $ the diagonal.
\item  For any cartesian square
$$\xymatrix{{Y'}\ar[r]\ar[d]\ar@{}[rd]|{\square} & {Y}\ar[d]^{f}\\
{X'}\ar[r]^{g} & {X}} $$ where $Y,X,X'$ are in $\RegCon$, $g:X'\ra X$ is a morphism and $f:Y\ra X$ is a finite flat morphism, we have 
 \[[\Gamma^t_f]\circ[\Gamma_g]=\sum_T \length(\sO_{Y',\eta_T})\cdot [\Gamma_{g_T}]\circ [\Gamma_{f_T}^t]\quad \text{in } \Cor(X',Y),\]
where the sum is taken over the irreducible components $T$ of $Y'$, and for such an irreducible component $g_T:\tilde{T}\ra Y$ and $f_T:\tilde{T}\ra X'$ are the canonical morphisms from the normalization $\tilde{T}$ of $T$.
\end{enumerate}
\end{lemma}
\begin{proof}
(1) follows from $\Gamma^t_{q_V}\times_{\tilde{V}} \Gamma_{p_V}\cong \tilde{V}$, (2) from $\Gamma^t_f\times_X \Gamma_f\cong \Delta_X$ and (3) follows from 
$\Gamma_g\times_X\Gamma^t_f\cong X'\times_X Y\cong Y'$, the definition of the composition and (1).
\end{proof}

From this, we obtain:

\begin{lemma}\label{RemaLength}
Let $X,Y,Z$ be in $\Regone$ and $[V]\in\Cor(X,Y)$, $[W]\in\Cor(Y,Z)$ be elementary correspondences. Let $T$ be an irreducible component of $V\times_YW$. Then 
\begin{equation}\label{EgRemaLength}
\length(\sO_{V\times_Y W,\eta_{T}})=\sum_{C/T}\length(\sO_{\tilde{V}\times_Y\tilde{W},\eta_C})\deg(C/T).
\end{equation}
where the sum is taken over the irreducible components $C$ of $\tilde{V}\times_Y\tilde{W}$ that dominate $T$.
\end{lemma}

\subsection{Presheaves with transfers}

\begin{definition}\label{defn-PT}
A {\em presheaf 
with transfers on $\Reg$} 
 is a $R$-linear contravariant functor
       \[\sM: \Reg\Cor^\op\to (R-{\rm mod}).\]
        We denote by $\PT$ the category of presheaves with transfers on $\Reg$ with morphisms the natural transformations.
       Further we denote by $\NT$ the full subcategory with objects the {\em Nisnevich sheaves with transfers on $\Reg$}, i.e. those
       presheaves with transfers on $\Reg$ whose underlying presheaf on $\Reg$ is a sheaf in the Nisnevich topology.
\end{definition}

\begin{lemma}\label{CarPSTRegone}
A presheaf with transfers on $\Regone$ is equivalent to the following data 
\begin{itemize}
\item{a functor $\sM_*:\RegCon_*\ra (R-\Mod)$;}
\item{a functor $\sM^*:{\RegCon}^\op\ra(R-\Mod)$;}
\end{itemize}
which satisfies the following conditions (where we write $f_*:=\sM_*(f)$ and $f^*:=\sM^*(f)$ for a map $f$ in $\RegCon_*$ and $\RegCon$, respectively)
\begin{enumerate}
\item{
$\sM_*(X)=\sM^*(X)$ for any $X\in\RegCon$;}
\item{$g_*g^*=\deg(Y/X)\cdot\id_{\sM(X)}$ for any finite flat morphism $g:Y\ra X$ between connected schemes in $\Regone$;}
\item{ for any cartesian square
$$\xymatrix{{Y'}\ar[r]\ar[d]\ar@{}[rd]|{\square} & {Y}\ar[d]^{f}\\
{X'}\ar[r]^{g} & {X}} $$ where $Y,X,X'$ are in $\RegCon$, $g:X'\ra X$ is a morphism and $f:Y\ra X$ is a finite flat morphism, we have 
$$g^*\circ f_*=\sum_T \length(\sO_{Y',\eta_{T}})\cdot f_{T*}\circ g_T^*,
 $$
where the sum is taken over the irreducible components $T$ of $Y'$, and for such an irreducible component $g_T:\tilde{T}\ra Y$ and $f_T:\tilde{T}\ra X'$ are the canonical morphisms.
}
\end{enumerate}
\end{lemma}

\begin{proof}
That a presheaf with transfers on $\RegCon$ defines such data is an immediate consequence of Lemma \ref{lem-basic-corr-formulas}. Conversely let $\sM_*$ and $\sM^*$ be such a data satisfying (1)--(3).
Set $\sM(X):=\sM_*(X)=\sM^*(X)$ for any $X\in\RegCon$ and $\sM(X):=\oplus_i \sM(X_i)$ for $X \in \Regone$ with connected components $X_i$. For an elementary correspondence $[V]\in \Cor(X,Y)$
we define a morphism
$$\sM([V]):=q_{V*}p_V^*:\sM(Y)\ra \sM(X),$$
where $q_V: \tilde{V}\to X$ and $p_V:\tilde{V}\to Y$ are the canonical maps from the normalization $\tilde{V}$ of $V$. We extend this linearly to a map $\Cor(X,Y)\to \Hom_R(\sM(Y),\sM(X))$. Using (2), (3) and Lemma \ref{RemaLength} it is straightforward to see that this construction is functorial. Hence the statement.
\end{proof}

\subsection{Mackey functors with specialization maps}

\begin{definition}\label{defn-PST-MF}
 An {\em $R$-Mackey functor} (or just Mackey functor) is a $R$-linear contravariant functor 
        \[M:\pt\Cor^\op\to (R-{\rm mod}).\]
      We denote by $\MF_R=\MF$ the category of Mackey functors with morphisms the natural transformations.
\end{definition}

\begin{notation}\label{not-pb-pf-sp}
Let $\sM$ be a presheaf with transfers on $\Reg$ (or a Mackey functor).
Let $f:X\to Y$ be a flat map in $\Reg$ (equivalently each connected component of $X$ dominates a connected component of $Y$), then we define the {\em pullback} attached to $f$ as
\[f^*:=\sM([\Gamma_f]): \sM(Y)\to \sM(X).\]
Let $f:X\to Y$ be finite and flat (i.e. each connected component of $X$ maps finite and surjective to a connected component of $Y$), then we define the {\em pushforward or trace} attached to $f$ as 
\[f_*=\Tr_{f}=\Tr_{X/Y}:=\sM([\Gamma_f^t]): \sM(X)\to \sM(Y).\]
Let $i:P\inj U$ be the inclusion of a closed point $P$ in a 1-dimensional scheme in $\Reg$, then we define the {\em specialization map} attached to $i$ as
\[s_P:=i^*:=\sM([\Gamma_i]): \sM(U)\to \sM(P).\]
Let $C\in (\sC/S)$ be a curve, then we denote by $\sM_C$
the presheaf on the Zariski site of $C$ induced by $\sM$. For $x\in C$ a point we denote  by $\sM_{C,x}$ the Zariski stalk of $\sM_C$ at $x$ and by
 $\sM_{C,x}^h$ the Nisnevich stalk, i.e. the inductive limit of $\sM(U)$ over all Nisnevich neighborhoods $U$ of $x\in C$.
Notice that $\sM_{C,\eta_C}=\sM_{C,\eta_C}^h$.
Also notice that for $C\in (\sC/S)$ and $P\in C$ the specialization map also induces maps (which by abuse of notation we still denote by $s_P$)
\[s_P:\sM_{C,P}\to \sM(P),\quad s_P:\sM_{C,P}^h\to \sM(P).\]
\end{notation}

\begin{remark}\label{rmk-alt-defn-MF}
It follows from the proof of Lemma \ref{CarPSTRegone} that a Mackey functor is the same as
giving two functors
 \[M^*: ({\rm pt}/S)^\op\to (R-{\rm mod}),\quad M_*:(\rm{pt}/S)_*\to (R-\rm{mod}),\]
which satisfy the following relations.
\begin{enumerate}
 \item[(MF0)] For all $x\in (\pt/S)$ we have $M_*(x)=M^*(x)=:M(x)$.
 \item[(MF1)] Let $\varphi: y\to x$ be a finite morphism of $S$-points and $\psi: z\to x$ any morphism of $S$-points.
             For $s\in y\times_x z$ denote by $\psi_s: s\to y$ and $\varphi_s: s\to z$ the natural maps induced by the projections and set $l_s:={\rm length}(\sO_{y\times_x z,s})$.
             Then 
                 \[M^*(\psi)\circ M_*(\varphi)= \sum_{s\in y\times_x z} l_s\cdot M_*(\varphi_s)\circ M^*(\psi_s): M(y)\to M(z).\]
 \item[(MF2)] Let $\varphi: y\to x$ be a finite morphism of $S$-points, then $M_*(\varphi)\circ M^*(\varphi)$ is multiplication with $\deg(y/x)$.
\end{enumerate}
\end{remark}

\begin{remark}
Notice that there are also different definitions of Mackey functors in the literature (e.g. functoriality is only required for separable field extensions or (MF2) is not required to hold).
\end{remark}

\begin{definition}\label{defn-MFsp}
A {\em $R$-Mackey functor with specialization maps} is a Nisnevich sheaf with transfers on $\Reg$, 
which additionally satisfies the following two conditions. 
      \begin{enumerate}
      \item[(Inj)] For all open immersions $j:U\inj X$ between non-empty connected schemes in $\Reg$ the restriction map 
       \[j^*: \sM(X)\inj \sM(U)\] is injective.
      \item[(FP)] For all $C\in (\sC/S)$ with generic point $\eta$ the natural map \[\varinjlim_{U\subset C} \sM(U)\xr{\simeq} \sM(\eta)\]
 is an isomorphism, where the limit is over all non-empty open subsets $U\subset C$.
      \end{enumerate}
We denote by $\MFsp_R=\MFsp$ the full subcategory of $\NT$ with objects the Mackey functors with specialization maps.
\end{definition}

\begin{definition}\label{defn-projectivization}
A {\em compactification} of an integral and 1-dimensional scheme $U\in \Regone$ is by definition
an integral curve $C\in (\sC/S)$ which has $U$ as an open subscheme.
Notice that for each $S$-point $x$ over which $U$ is of finite type there exists a unique (up to isomorphism)
compactification $C$ over $x$.
\end{definition}

The name ``Mackey functor with specialization maps'' is justified by the following proposition.

\begin{proposition}\label{prop-alt-MFsp}
A Mackey functor with specialization maps is the same as giving a triple $(M, \sM, s)$, where 
\begin{enumerate}
 \item $M$ is a Mackey functor,
 \item $\sM$ is a collection of Zariski sheaves of $R$-modules  $\sM_C$ on $C$, $C\in (\sC/S)$, together with a collection of pullback morphisms $\pi^*: \sM_C(U)\to \sM_D(V)$
        for each morphism $\pi:V\to U$ between 1-dimensional connected schemes in $\Regone$ with compactifications
        $D$ and $C$, respectively, and a collection of pushforward morphisms $\pi_*: \sM_D(V)\to \sM_C(U)$, for each finite morphism $\pi$ as above, which both are functorial.
 \item $s$ is a collection of $R$-linear homomorphisms $s_P: \sM_{C,P}\to M(P)$ for $C\in (\sC/S)$ and $P\in C$ (where $\sM_{C,P}$ denotes the Zariski stalk).
\end{enumerate}
These data satisfy the following conditions:
\begin{enumerate}
 \item[(RS0)] For all $C\in (\sC/S)$ the functor on the small Nisnevich site  $C_\Nis$, which sends an \'etale $C$-scheme $U$ to $\oplus_i \sM_{C_i}(U_i)$, where the $U_i$'s are the connected components of $U$ with unique compactification
                $C_i$ over $C$, and which on morphisms is defined using the pullbacks from (2), is a sheaf.
 \item[(RS1)] For any $C\in (\sC/S)$ with generic point $\eta_C$, the Zariski-sheaf $\sM_C$ is a subsheaf of the constant sheaf $M(\eta_C)$,
              inducing an isomorphism at the generic stalk $\sM_{C,\eta_C}=M(\eta_C)$,
              which is compatible with the pullback and pushforward morphisms of $\sM$ and $M$.
 \item[(RS2)] For any 1-dimensional and integral $U\in \Reg$, any compactification $C\in (\sC/S)$ 
                   and any map $\varphi: U\to x$ to an $S$-point $x$, 
              the pullback $\varphi^*: M(x)\to M(\eta_C)$ has its image contained in $\sM_C(U)$.
 \item[(S1)] For any morphism $\pi: D\to C$ in $(\sC/S)$, $Q\in D$,  $P:=\pi(Q)$ and $a\in \sM_{C,P}$ 
                                 \[\pi_Q^*(s_P(a))=s_Q(\pi^*a)\quad \text{in } M(Q),\]
            where $\pi_Q:Q\to P$ denotes the map induced by $\pi$.
 \item[(S2)] For any finite morphism $\pi: D\to C$ in $(\sC/S)$, any $P\in C$ and $a\in \pi_*(\sM_D)_P$ 
                                \[s_P(\pi_*(a))=\sum_{Q\in \pi^{-1}(P)} e(Q/P)\cdot \pi_{Q*}(s_Q(a))\quad \text{in } M(P),\]
                                 where $\pi_{Q}:Q\to P$ is the finite morphism induced by $\pi$ and $s_Q$ on the right-hand side is the composition $\pi_*(\sM_D)_P\to \sM_{D,Q}\xr{s_Q} M(Q)$.
 \item[(S3)] For any $C\in(\sC/S)$ and any closed point $P\in C$ and  $a\in M(x_C)$ 
                                 \[s_P(\rho_C^*(a))=  \bar{\rho}_C^*(a),\]
                                  where $\bar{\rho}_C: P\to x_C$ is induced by the structure morphism $\rho_C: C\to x_C$.
                                  Notice that this makes sense because of (RS2).
\end{enumerate}
\end{proposition}

\begin{proof}
It follows from Lemma \ref{lem-basic-corr-formulas}, Notation \ref{not-pb-pf-sp} and Remark \ref{rmk-alt-defn-MF} that a Mackey functor with specialization maps determines a triple as in the proposition.
Conversely given such a triple $(M,\sM, s)$ we can construct a Mackey functor with specialization map $\tilde{\sM}$ as follows. For an $S$-point $x$ set $\tilde{\sM}(x):=M(x)$ and 
for $U\in \Reg$ integral and 1-dimensional set $\tilde{\sM}(U):= \sM_C(U)$, where $C\in (\sC/S)$ is some
compactification of $U$. Notice that by the functoriality in (2), we have $\sM_C(U)=\sM_{C'}(U)$,
whenever $C$ and $C'$ are compactifications of $U$. Thus $\tilde{\sM}(U)$ is well-defined.
We extend this additively to all objects in
$\Reg$. If $[Z]\in \Cor(X,Y)$ is an elementary correspondence between integral schemes in $\Reg$ and $\tilde{Z}$ is the normalization of $Z$ and $q_Z: \tilde{Z}\to X$ and 
$p_Z: \tilde{Z}\to Y$ are the natural maps induced by the projections, then we define
\[\tilde{\sM}([Z]):= \begin{cases}  q_{Z*}\circ p_Z^*,& \text{if }p_Z(\tilde{Z}) \text{ contains the generic point of } Y,\\
                                    q_{Z*}\circ \varphi^*\circ s_P, &\text{if }p_Z \text{ factors as } \tilde{Z}\xr{\varphi} P \text{ with } \{P\}\varsubsetneq Y \text{ closed},
                     \end{cases}
\]
where the pullback and pushforward on the right hand side are induced by the Mackey functor structure and the structure underlying $\sM$ respectively;
e.g. if $\tilde{Z}$ is 1-dimensional with compactification $C$ the map $\varphi^*:M(P)\to \sM_C(\tilde{Z})$ above is the map induced by $\varphi^*:M(P)\to M(\eta_{\tilde{Z}})$ via (RS2).
It is straightforward to check that this defines a Mackey functor with specialization maps 
(cf. the proof of Lemma \ref{CarPSTRegone}).
\end{proof}

The main reason we work with the Nisnevich topology is the following Lemma due to V. Voevodsky.
\begin{lemma}\label{lem-Nis-keeps-transfer}
Let $\sM$ be a presheaf with transfers on $\Reg$ and denote by $\sM_\Nis$ the associated sheaf of $R$-modules in the Nisnevich topology on $\Reg$.
Then $\sM_\Nis$ uniquely extends to a Nisnevich sheaf with transfers on $\Reg$ and the canonical map $\sM\to \sM_{\Nis}$ is a morphism of presheaves with transfers.
Moreover the underlying Mackey functors of $\sM$ and $\sM_\Nis$ are equal and if $\sM$ satisfies the conditions (Inj) and (FP) from Definition \ref{defn-MFsp}, 
then $\sM_{\Nis}$ is a Mackey functor with specialization map.
\end{lemma}
\begin{proof}
The first two statements are exactly \cite[Lemma 3.1.6]{VoDM} (there it is proven for presheaves with transfers on $\SmCor_S$ but the proof goes through in our situation).
That the underlying Mackey functors of $\sM$ and $\sM_\Nis$ are equal follows from the fact that a Nisnevich covering of a point always admits the identity as a refinement.
If $\sM$ satisfies (FP), then the natural map $\varinjlim_{U\subset C} \sM_\Nis(U)\to \sM_\Nis(\eta)=\sM(\eta)$ is surjective, where the limit is over all non-empty open subsets $U\subset C$.
Hence it suffices to show that if $\sM$ additionally satisfies (Inj), then $\sM_\Nis(U)\to \sM_\Nis(\eta)$, for $U$ as above, is injective.
For this take $a\in \sM_\Nis(U)$, which maps to zero in $\sM_\Nis(\eta)$ and denote by $\sigma:\sM\to \sM_{\Nis}$ the canonical map. Then we find a covering $\pi: V\to U$ in the Nisnevich topology and an element  $b\in \sM(V)$
such that $\pi^*a=\sigma(b)$. It follows that $b$ maps to zero in $\oplus_i\sM(\eta_i)$, where 
the $\eta_i$ are the generic points of $V$. Therefore by (Inj) and (FP) $b=0$, hence also $\pi^*a=0$.
By the Nisnevich property one of the $\eta_i$ is actually equal to $\eta$
and thus $a=0$. This yields the statement.
\end{proof}

\subsection{The modulus condition}

\begin{definition}\label{defn-modulus}
Let $\sM$ be a Mackey functor with specialization, $C\in (\sC/S)$, $a\in \sM(\eta_C)$ and let $\fm=\sum_{P\in C} n_P P$ be an effective divisor on $C$.
Then we say that $\fm$ is a {\em modulus for $a$} if and only if the following condition is satisfied:
\begin{enumerate}
 \item[(MC)] $a\in \sM(C\setminus|\fm|)$ (this makes sense by (FP)) and for $f\in \G_m(\eta_C)$ with $f\equiv 1$ mod $\fm$ we have
               \[\sum_{P\in C\setminus |\fm|} v_P(f)\Tr_{P/x_C}s_P(a)=0,\]
\end{enumerate}
where $x_C$ is defined in \eqref{some-cats1}. (Here we allow $\fm=0$, in which case $f\equiv 1$ mod $\fm$ is an empty condition.)
Clearly the set of all $a\in \sM(\eta)$, which have $\fm$ as a modulus define a $R$-submodule, denoted by
\[\sM(C,\fm)= \{a\in \sM(\eta)\,|\, \fm \text{ is a modulus for } a \}.\]
\end{definition}

We want to give a reformulation of the modulus condition in terms of correspondences.

\begin{definition}\label{defn-corr-with-mod}
For two pairs $(C,\fm)$ and $(D, \fn)$ with $C, D\in (\sC/S)$ curves and $\fm$ and $\fn$ effective divisors on them, define
\[\Cor((C,\fm), (D,\fn))\subset \Cor(C\setminus|\fm|, D\setminus|\fn|)\]
to be the free $R$-module generated by symbols $[V]$, where $V\subset (C\setminus|\fm|)\times (D\setminus|\fn|)$ is an integral closed subscheme which is finite and surjective over
$C\setminus|\fm|$, such that $\nu^*(\fm\times D)\geqslant \nu^*(C\times \fn)$ with $\nu: \tilde{V}\to \bar{V}\subset C\times D$ the normalization of the closure of $V$.

Notice that for $\fn=0_D$ the zero divisor on $D$ we have $\Cor((C,\fm), (D,0_D))=\Cor(C, D)$.
\end{definition}

\begin{lemma}\label{lem-alt-defn-MC}
Let $C\in (\sC/S)$ be a curve, $\fm$ an effective divisor on it and $\sM$ a Mackey functor with specialization maps.
Then an element $a\in\sM(\eta_C)$ has modulus $\fm$ if and only if $a\in \sM(C\setminus|\fm|)$ and for all $\gamma\in \Cor((\P^1_{x_C},\{1\}), (C,\fm))$  we have
\eq{lem-alt-defn-MC1}{i^*_0\sM(\gamma)(a)=i^*_\infty\sM(\gamma)(a),}
where $i_\epsilon: \{\epsilon\}\inj \P^1$, $\epsilon\in\{0,\infty\}$, are the closed immersions.
\end{lemma}
\begin{proof}
Clearly \eqref{lem-alt-defn-MC1} implies (MC): Indeed for $f\in \G_m(\eta_C)$ with $f\equiv 1$ mod $\fm$, 
take $[\Gamma]^t\in \Cor(\P^1, C)$ to be the transpose of the graph of the finite and surjective map 
$C\to \P^1$ determined by $f$. 
Then $[\Gamma]^t$ restricts to an element $\gamma$ in $\Cor((\P^1_{x_C},\{1\}), (C,\fm))$.
In this case \eqref{lem-alt-defn-MC1} is just a reformulation of (MC) (use (S2)).

Now assume that $a$ has a modulus $\fm$ and let $\gamma=[V]\in\Cor((\P^1_{x_C},\{1\}), (C,\fm))$ 
be an elementary correspondence. 
Denote by $\tilde{V}$ the normalization of the closure of $V$ in $\P^1_{x_C}\times C$ and by 
$\pi_{\P^1}: \tilde{V}\to \P^1_{x_C}$ and $\pi_C: \tilde{V}\to C$  the natural maps induced by projection.
Using Lemma \ref{lem-basic-corr-formulas}, formulas  (S1), (S2) and the notations introduced in \ref{not-pb-pf-sp} 
we obtain
\mlnl{i_0^*\sM(\gamma)(a)-i_\infty^*\sM(\gamma)(a)= 
\sum_{Q\in \pi_{\P^1}^{-1}(0)} e(Q/0)\cdot \pi_{\P^1,Q*}\pi_{C,Q}^*(s_{\pi_C(Q)}(a))\\ 
                  -  \sum_{Q\in \pi_{\P^1}^{-1}(\infty)} e(Q/\infty)\cdot \pi_{\P^1,Q*}\pi_{C,Q}^*(s_{\pi_C(Q)}(a)), }
where $\pi_{C,Q}: Q\to \pi_C(Q)$ and $\pi_{\P^1,Q}: Q\to \pi_{\P^1}(Q)=x_C$ are the natural maps.
By (MF2) we have $\pi_{\P^1,Q*}\pi_{C,Q}^*= f(Q/P)\cdot\Tr_{P/x_C}$, with $P=\pi_{C}(Q)$. 
Let $f\in \G_m(\eta_{\tilde{V}})$ be the function corresponding to $\pi_{\P^1}$, then
$f\equiv 1$ mod $\pi^*\fm$ (by definition of $\Cor$ for pairs) and $e(Q/0)=v_Q(f)$, $e(Q/\infty)=-v_Q(f)$.
Thus
\begin{align}
i_0^*\sM(\gamma)(a)-i_\infty^*\sM(\gamma)(a) & = 
           \sum_{P\in C\setminus|\fm|}  \bigg(\sum_{Q\in \pi^{-1}_C(P)} f(Q/P)v_Q(f)\bigg)\cdot s_P(a)\notag\\
                                    & = \sum_{P\in C\setminus|\fm|} v_P(\Nm_{\eta_{\tilde{V}}/\eta_C}(f))\cdot s_P(a).\notag
\end{align}
But this last sum is zero since $a$ has modulus $\fm$ and $\Nm_{\eta_{\tilde{V}}/\eta_C}(f)\equiv 1$ mod $\fm$ by the following Lemma; hence the statement.
\end{proof}

\begin{lemma}\label{lem-mod-norm}
Let $L/K$ be a finite field extension and $v$ a discrete valuation on $K$  with valuation ring $A$. 
Assume that the normalization $B$ of $A$ in $K$ is a finite $A$-module and let $w_1,\ldots, w_r$ be all the extensions of $v$ to $L$.
Let $K_s\subset L$ be the separable closure of $K$ in $L$ and $w_{i,s}:={w_i}_{|K_s}$, $i=1,\ldots, r$.
Then for all $m\geqslant 1$ and  $n_i\geqslant \frac{m\cdot e(w_{i,s}/v)}{f(w_i/w_{i,s})}$  we have 
\[\Nm_{L/K}(U^{(n_1)}_{w_1}\cap\ldots \cap U^{(n_r)}_{w_r})\subset U^{(m)}_v.\]
\end{lemma}

\begin{proof}
It suffices to consider separately the two cases in which $L/K$ is either purely inseparable or separable.
In the purely inseparable case we have $\Nm_{L/K}(g)= g^{[L:K]}$ and $[L:K]=e(w/v) f(w/v)$.
Thus for $g=1+ \tau^n a$ with $a\in \sO_w$, $n\geqslant \frac{m}{f(w/v)}$ and $\tau$ a local parameter at $w$ we have 
$\Nm_{L/K}(g)= 1+ t^m b$ for some $b\in A$ and $t$ a local parameter at $v$.
In the separable case  take $g\in U^{(n_1)}_{w_1}\cap\ldots \cap U^{(n_r)}_{w_r}$, $n_i\geqslant m e(w_i/v)$, and choose an element $\tau\in B$ which is a local parameter at $w_1,\ldots,w_r$.
Then there exist $a_i\in \sO_{w_i}$ such that $g=1+ \tau^{n_i} a_i$. We can rewrite this as $g= 1+ t^m a'_i$ with $a'_i\in\sO_{w_i}$. It follows that the $a_i'$ are all equal to an element, say, $a'\in B$.
Thus
 \[\Nm_{L/K}(g)=\prod_{\sigma\in \Hom_K(L, \bar{K})} (1+ t^n\sigma(a'))\in U^{(n)}_v.\]
Hence the statement.
\end{proof}


\begin{proposition}\label{prop-modulus}
Let $\sM$ be a Mackey functor with specialization maps, $C\in(\sC/S)$ and $\fm$ an effective divisor on $C$.
\begin{enumerate}
 \item If $\fm\leqslant\fn$ on $C$, then $\sM(C,\fm)\subset \sM(C, \fn)$.
 \item Let $\pi: D\to C$ be a finite morphism in $(\sC/S)$ and $\fm$ and $\fn$ effective divisors on $C$ and $D$, respectively, with $\pi^*\fm\geqslant\fn$. 
       Then the pushforward $\pi_*: \sM(\eta_D)\to \sM(\eta_C)$ restricts to 
         \[\pi_*:\sM(D,\fn)\to \sM(C,\fm).\]
 \item Let $\pi:D\to C$ be a morphism in $(\sC/S)$ such that the induced map $x_D\xr{\simeq} x_C$ is an isomorphism and let $\fm$ and $\fn$ be effective divisors on $C$ and $D$,
       respectively, with $\fn\geqslant \pi^*\fm /[D:C]_{\rm insep}$.
       Then $\pi$ is finite and the pullback $\pi^*: \sM(\eta_C)\to \sM(\eta_D)$ restricts to
         \[\pi^*: \sM(C,\fm)\to \sM(D,\fn).\]
\end{enumerate}
\end{proposition}

\begin{proof}
 (1) is easy. For (2) take $b\in \sM(D,\fn)$ and $f\in \G_m(\eta_C)$ with $f\equiv 1$ mod $\fm$.
  Then $\pi^*f$ is congruent to 1 modulo $\pi^*\fm$  a fortiori modulo $\fn$ and thus (2) follows from
   \begin{align}\label{prop-modulus1}
      \sum_{Q\in\pi^{-1}(P)}\Tr_{x_D/x_C}&(v_Q(\pi^*f)\Tr_{Q/x_D}(s_Q(b)))\notag\\
 &=v_P(f)\Tr_{P/x_C}\bigg(\sum_{Q\in\pi^{-1}(P)} e(Q/P) \Tr_{Q/P}(s_Q(b))\bigg) \notag\\
  &= v_P(f)\Tr_{P/x_C}(s_P(\pi_*b)), 
  \end{align}
where the last equality holds by (S2). Finally, for (3) take $a\in\sM(C,\fm)$ and $g\in\G_m(\eta_D)$ with $g\equiv 1$ mod $\fn$. Then $\pi_*(g)\equiv 1$ mod $\fm$ (by Lemma \ref{lem-mod-norm})
and (3) follows from
\begin{align}\label{prop-modulus2}
\sum_{Q\in \pi^{-1}(P)} v_Q(g) \Tr_{Q/x_D}s_Q(\pi^*a)
      &= \sum_{Q\in \pi^{-1}(P)} f(Q/P) v_Q(g) \Tr_{P/x_C}(s_P(a))\notag \\
      &= v_P(\pi_*g)\Tr_{P/x_C}(s_P(a)), 
      \end{align}
where the first equality holds by (S1) and (MF2) and the second equality follows from $\Div(\pi_*g)=\pi_*\Div(g)$.
 \end{proof}

\subsubsection*{Associated symbols}

\begin{proposition}[cf. {\cite[III, \S 1, Prop. 1]{SerreGACC}}]\label{prop-modulus=symbol}
Let $M_*: ({\rm pt}/S)_*\to (R-\text{mod})$ be a functor and let $C\in (\sC/S)$ be fixed. Assume for all $P\in C$ we are
given submodules $M_P(\eta)\subset M(\eta)$ and $R$-linear maps $s_P: M_P(\eta)\to M(P)$. For $A\subset C$ we set $M_A(\eta):=\cap_{P\in A} M_P(\eta)$ and 
for $\varphi:x\to y$ in $({\rm pt}/S)_*$ we set $\Tr_{x/y}:=M_*(\varphi)$. Then for any $a\in M(\eta)$ the following 
two statements are equivalent:
\begin{enumerate}
 \item There exists a unique family of continuous group homomorphisms $\{\rho_P: \G_m(\eta)\to M(x_C)\}_{P\in C}$ 
        such that for some non-empty open subset $V\subset C$ the following conditions are satisfied
         \begin{enumerate}
           \item $a\in M_V(\eta)$.
           \item $\rho_P(f)=v_P(f)\Tr_{P/x_C}(s_P(a))$ for all $f\in \G_m(\eta)$ and $P\in V$.
           \item $\sum_{P\in C} \rho_P(f)=0$ for all $f\in \G_m(\eta)$.
         \end{enumerate}
 \item There exists an effective divisor $\fm=\sum_{Q\in C} n_Q Q$ on $C$ such that $a\in M_{C\setminus |\fm|}(\eta)$ and for all $f\in \G_m(\eta)$ with $f\equiv 1$ mod $\fm$, we have
        \[\sum_{P\in C\setminus |\fm|} v_P(f) \Tr_{P/x_C}(s_P(a))=0.\] 
\end{enumerate}
\end{proposition}

\begin{proof}The proof is along the lines of the proof of \cite[III, \S 1, Prop. 1]{SerreGACC}. \end{proof}

\subsubsection*{The modulus condition is local in the Nisnevich topology}

\begin{thm}\label{thm-modulus-Nis-local}
Let $\sM\in \MFsp$ be a Mackey functor with specialization maps. 
Let  $C\in (\sC/S)$ be a curve, $U\subset C$ a non-empty open subset and $a\in \sM(U)$ a section.
Assume there exists a covering $\pi:\coprod_i\ V_i\to U$ in the Nisnevich topology satisfying the following properties:
\begin{enumerate}
 \item The $V_i's$ are connected. We denote by $\pi_i: D_i\to C$ in $(\sC/S)$ the 
         compactification of $\pi_{|V_i}:V_i\to U$ (over $x_C$).
  \item For all $i$ there exists an effective divisors $\fn_i$ on $D_i$ such that  $|\fn_i|=D_i\setminus V_i$ and  
         $\pi_i^*a\in \sM(D_i,\fn_i)$.
\end{enumerate}
Then there exists an effective divisor $\fm$ on $C$ such that $|\fm|=C\setminus U$ and $a\in\sM(C,\fm)$.
Furthermore, if the divisors $\fn_i$ are bounded by a natural number $n\geqslant 1$ for all $i$, then there exists 
an effective divisor $\fm$ as above which is  bounded by the same $n$.
(Here we say that an effective divisor $\sum_{P} n_P [P]$ is {\em bounded by $n$} iff $n_P\leqslant n$ for all $P$.)
\end{thm}

Before we start with the proof of the theorem we need some preparations.

\begin{lemma}\label{lem-modulus-Zar-local}
Let $\sM\in \MFsp$ be a Mackey functor with specialization maps. 
Let  $C\in (\sC/S)$ be a curve, $U\subset C$ a non-empty open subset and $a\in \sM(U)$ a section.
Assume there exist a Zariski open covering $U=\cup_i U_i$ and effective divisors $\fm_i$ on $C$ with $|\fm_i|=C\setminus U_i$ and $a\in \sM(C,\fm_i)$, for all $i$.
Then there exists an effective divisor $\fm$ on $C$ with $|\fm|=C\setminus U$ and $a\in \sM(C,\fm)$.
Further, if the $\fm_i$ are bounded by $n$ for all $i$, then we can achieve that $\fm$ is also bounded by $n$.
\end{lemma}

\begin{proof}
It suffices to consider finite coverings and by induction over the number of $U_i's$ it suffices to consider coverings by two open subsets $U=U_1\cup U_2$.
By Proposition \ref{prop-modulus}, (1) we can further assume that 
\[\fm_1=\sum_{P\in C\setminus U_1} n[P],\quad \fm_2=\sum_{P\in C\setminus U_2} n[P],\]
for some large enough positive integer $n$. Set 
\[\fm:=\min\{\fm_1,\fm_2\}.\]
Then $|\fm|= |\fm_1|\cap |\fm_2|= (C\setminus U_1)\cap (C\setminus U_2)=C\setminus U$
and $\fm$ is bounded by $n$. We claim 
\eq{lem-modulus-Zar-local1}{a\in \sM(C,\fm).}
For this set $\fm_{12}:=\max\{\fm_1,\fm_2\}$. Then $C\setminus|\fm_{12}|=U_1\cap U_2$ and $\fm_{1,2}\geqslant \fm_i$, $i=1,2$.
Hence
\eq{lem-modulus-Zar-local2}{a\in \sM(C,\fm_1),\, \sM(C,\fm_2),\,\sM(C,\fm_{12}).}
Therefore the pairs $(a,\fm_1)$, $(a,\fm_2)$, $(a,\fm_{12})$ satisfy condition (2) of Proposition \ref{prop-modulus=symbol} (with $M_*:=\sM_{|(\pt/S)}$ and $M_P(\eta_C)=\sM_{C,P}$).
Thus we get families of continuous group homomorphisms $\{\rho_{P,\fm_*}:\G_m(\eta_C)\to \sM(x_C)\}$ satisfying the properties (a)-(c) from Proposition \ref{prop-modulus=symbol}, (1)
(with $V=C\setminus|\fm_*|$), for $*\in \{1,2,12\}$. But the uniqueness statement of this proposition yields
\eq{lem-modulus-Zar-local3}{\rho_{P,\fm_1}=\rho_{P,\fm_{12}}=\rho_{P,\fm_2}=:\rho_P.}
Now fix $f\in \G_m(\eta_C)$ with $f\equiv 1$ mod $\fm$.
Take $N\geqslant n$ with $\rho_P(U_P^{(N)})=0$ for all $P\in \fm_{12}$ (exists by continuity) and choose $h\in \G_m(\eta_C)$ so that
\[\frac{h}{f}\in U^{(N)}_P\text{ for all }P\in |\fm| \quad\text{and}\quad h\in U^{(N)}_P \text{ for all }P\in |\fm_{12}|\setminus |\fm|.\]
Then 
\begin{enumerate}
 \item $h\equiv 1$ mod $\fm_*$, for all $*\in \{1,2, 12\}$.
 \item $\sum_{P\in |\fm_*|} \rho_P(h)=0$, for all $*\in \{1,2, 12\}$.
 \item $\rho_P(h)=\rho_P(f)$ for all $P\in |\fm|$.
 \item $\rho_P(f)=v_P(f) \Tr_{P/x_C}(s_P(a))$, for all $P\in U$.
\end{enumerate}
Here (1) holds by definition, (3) follows from the choice of $N$ and $h$ and 
(4) from \eqref{lem-modulus-Zar-local3} and $U=U_1\cup U_2$. For 
(2) notice that by \eqref{lem-modulus-Zar-local2}, \eqref{lem-modulus-Zar-local3} and reciprocity 
$\sum_{P\in |\fm_*|} \rho_P(h)=-\sum_{P\in C\setminus|\fm_*|}v_P(h)\Tr_{P/x_C}(s_P(a))$ which vanishes
since $\fm_*$ is a modulus for $a$.
Hence
\begin{align}
0 & = \sum_{P\in |\fm_{12}|} \rho_P(h)- \sum_{P\in |\fm_1|} \rho_P(h)- 
                         \sum_{P\in |\fm_2|} \rho_P(h), &\text{by (2)}\notag\\
  & = -\sum_{P\in |\fm|}\rho_P(h) = -\sum_{P\in |\fm|}\rho_P(f), &\text{by (3)}\notag\\
  & = \sum_{P\in U} \rho_P(f),& \text{by reciprocity}\notag\\
  &= \sum_{P\in U} v_P(f) \Tr_{P/x_C}(s_P(a)), & \text{by (4).}\notag
\end{align}
This proves the Claim \eqref{lem-modulus-Zar-local1} and hence the lemma.
\end{proof}

\begin{notation}\label{not-modulus-functions}
For $C\in (\sC/S)$ with function field $K=\kappa(C)$ and $\fm$ an effective divisor on $C$ we denote
\[K_{\fm,1}:=\{f\in K^\times\,|\, f\equiv 1 \text{ mod }\fm\}.\]
Notice that $K_{\fm,1}\subset K^\times$ is a subgroup and that $K_{\fm',1}\subset K_{\fm,1}$ for all effective divisors $\fm\leqslant \fm'$ and $K_{\{0\}, 1}=K^\times$.
\end{notation}

\begin{lemma}\label{lem-Norm-on-Ln-is-surj}
Let $\pi: D\to C$ be a finite flat morphism in $(\sC/S)$ and denote by $L=\kappa(D)$ and $K=\kappa(C)$
the function fields. Let $\fm'$ be an effective divisor on $C$ and $P\in C$ a closed point such that
\[\fm'=\fm+n[P],\]
where $\fm$ is an effective divisor on $C$ whose support is disjoint from $P$ and $n\geqslant 1$.
Assume there exists a closed point $Q_0\in \pi^{-1}(P)$ such that $\pi$ is \'etale in 
a neighborhood of $Q_0$ and induces an isomorphism $\kappa(P)\xr{\simeq}\kappa(Q_0)$.
Then the norm map $\Nm: L^\times \to K^\times$ induces a surjection
\[\Nm: L_{\pi^*\fm'-n[Q_0],1}\surj \frac{K_{\fm,1}}{K_{\fm',1}}.\]
\end{lemma}

\begin{proof}
Take $f\in K_{\fm,1}$. Choose a natural number $N$ greater than the coefficients in $\pi^*\fm'$ and a function
 $g\in L^\times$ such that
\[g/f\in U^{(N)}_{Q_0},\quad  g\in U^{(N)}_R, \text{ for all }R\in |\pi^*m'|\setminus\{Q_0\}.\]
Since $Q_0$ is \'etale over $P$ the divisor $\pi^*\fm'-n[Q_0]$ has its support disjoint from 
$Q_0$. Thus $g\in L_{\pi^*\fm'-n[Q_0],1}\subset L_{\pi^*\fm,1}$ and by Lemma \ref{lem-mod-norm}
also $\Nm(g)\in K_{\fm,1}$. It remains to show  $\Nm(g)/f\in U^{(n)}_P$.
Denote by $K_P$ the completion of $K$ with respect to $v_P$, by $\hat{U}^{(n)}_P$ the groups of 
higher 1-units in $K_P$
and by $L_Q$ the completion of $L$ with respect to $v_Q$, for $Q\in \pi^{-1}(P)$ . 
By the compatibility of the norm map with completion we obtain in $K_P$
\[\frac{\Nm(g)}{f}= \frac{\Nm_{L_{Q_0}/K_P}(g)}{f}\cdot 
                                               \prod_{Q\in \pi^{-1}(P)\setminus\{Q_0\}} \Nm_{L_Q/K_P}(g).\]
By the choice of $N$ and Lemma \ref{lem-mod-norm} we have 
$\Nm_{L_Q/K_P}(g)\in \hat{U}_P^{(n)}$ for all $Q\neq Q_0$. By assumption $L_{Q_0}=K_P$.
Hence $\Nm_{L_{Q_0}/K_P}(g)/f=g/f\in \hat{U}^{(n)}_{Q_0}=\hat{U}^{(n)}_P$.
All together we get $\Nm(g)/f\in K^\times\cap \hat{U}^{(n)}_P=U^{(n)}_P$. Hence the lemma. 
\end{proof}

\begin{proof}[Proof of Theorem \ref{thm-modulus-Nis-local}]
Let $\pi:\coprod_i\ V_i\to U$ be a covering in the Nisnevich topology as in the theorem. 
Then one of the $V_i$ is a Nisnevich neighborhood
of the generic point of $U$ and hence is a dense open subset of $U$. Therefore 
the theorem follows from Lemma \ref{lem-modulus-Zar-local} and the following special case:

Assume there exists a finite flat map $\pi: D\to C$ in $(\sC/S)$ and effective divisors $\fm'$ and $\fn$ on $C$ and $D$,
 respectively, which are bounded by a natural number $n\geqslant 1$ such that 
\begin{itemize}
 \item $\pi$ restricted to $V:=D\setminus|\fn|$ is \'etale over $U$.
 \item $U=U'\cup \pi(V)$, where $U':=C\setminus |\fm'|$.
 \item $U\setminus U'=\{P_1,\ldots, P_r\}$ and there exist closed points $Q_i\in \pi^{-1}(P_i)\cap V$ 
          such that $\pi$ induces  an isomorphism  $\kappa(P_i)\xr{\simeq} \kappa(Q_i)$ for all $i$.
 \item $a\in \sM(C,\fm')$ and $\pi^*a\in \sM(D,\fn)$.
\end{itemize}
{\em Claim:} There exists an effective divisor $\fm$ on $C$ bounded by $n$ 
                    such that $U=C\setminus|\fm|$ and $a\in \sM(C,\fm)$.

{\em Proof:} 
We proof the claim by induction on $r$.

{\em Case $r=1$.}
By Lemma \ref{lem-modulus-Zar-local} and Proposition \ref{prop-modulus}, 
(1) we can make $U$ smaller around $P_1$ and $\fm'$ and $\fn$ bigger.
Thus we can assume
\begin{enumerate}
 \item $V=\pi^{-1}(U')\cup\{Q_1\}$.
 \item $\fm'=\fm + n[P_1]$, where $\fm$ is an effective divisor on $C$ bounded by $n$ with $P_1\not\in |\fm|$, 
             i.e. $|\fm|=C\setminus U$.
 \item $\pi^*\fm'\geqslant \fn$.
\end{enumerate}
Since $|\fn|=D\setminus V= \pi^{-1}(C\setminus U')\setminus\{Q_1\}= \pi^{-1}(|\fm'|)\setminus\{Q_1\}$ 
and since the multiplicity of $Q_1$ in $\pi^*\fm'$ is $n$ ($Q_1$ being unramified over $P_1$)
we have $\pi^*\fm'-n[Q_1]\geqslant \fn$ and thus we can in fact assume
\begin{enumerate}
 \item[(4)] $\fn=\pi^*\fm'-n[Q_1]$.
\end{enumerate}
(Notice that now $\fn$ does not need to be bounded by $n$ anymore, but $\fm$ still is.)
Now the claim holds with $\fm$ as above.
Indeed, take $f\in \G_m(\eta_C)$ with $f\equiv 1$ mod $\fm$. 
By Lemma \ref{lem-Norm-on-Ln-is-surj} 
there exists a $g\in \G_m(\eta_D)$ with $g\equiv 1$ mod $\fn$
and $h\in \G_m(\eta_C)$ with $h\equiv 1$ mod $\fm'$ such that 
\[\pi_* g=fh\quad \text{in }\G_m(\eta_C).\]
We get
\begin{align}
\sum_{P\in U} v_P(f) \Tr_{P/x_C}(s_P(a))& = 
             \sum_{P\in U} v_P(f) \Tr_{P/x_C}(s_P(a)) + \sum_{P\in U'} v_P(h) \Tr_{P/x_C}(s_P(a)) \label{MNL1}\\
                                        & = \sum_{P\in U} v_P(fh) \Tr_{P/x_C}(s_P(a)) \label{MNL2}\\
                                        & = \sum_{P\in U} v_P(\pi_*g) \Tr_{P/x_C}(s_P(a))\notag\\
                                        & = \sum_{P\in U} \sum_{Q\in \pi^{-1}(P)} [Q:P]v_Q(g) \Tr_{P/x_C}(s_P(a))\label{MNL3}\\
                                        & = \sum_{Q\in\pi^{-1}(U) } v_Q(g) \Tr_{Q/x_C}(s_Q(\pi^*a))\label{MNL4} \\
                                        & = \sum_{Q\in V} v_Q(g) \Tr_{Q/x_C}(s_Q(\pi^*a)) \label{MNL5} \\
                                        & = 0,\label{MNL6}
\end{align}
where \eqref{MNL1} holds since $a\in\sM(C,\fm')$, \eqref{MNL2} by $v_{P_1}(h)=0$, \eqref{MNL3} by
$\Div(\pi_*g)=\pi_*\Div (g)$, \eqref{MNL4} by (MF2), (S1), \eqref{MNL5} since  $v_Q(g)=0$  
for $Q\in\pi^{-1}(P_1)\setminus\{Q_1\}$
and \eqref{MNL6} holds since $\pi^*a\in \sM(D,\fn)$. This proves $a\in \sM(C,\fm)$ and hence the claim 
in the case $r=1$.

{\em Case $r\ge 1$.}  
Set $U_1:=U'\cup \{P_1\}$, $V_1:=\pi^{-1}(U_1)\cap V$, $\fn_1:=\fn+\sum_{Q\in V\setminus V_1}[Q]$.
Then $\fn_1$ is bounded by $n$,  $|\fn_1|=D\setminus V_1$, $\pi^*a\in\sM(D,\fn_1)$, 
 $\pi$ restricted to $V_1$ 
is \'etale, $Q_1\in V_1$, $U_1\setminus U'=\{P_1\}$ and $U_1=U'\cup \pi(V_1)$.
Thus we can apply the case $r=1$ to get an effective divisor $\fm_1$ on $C$ bounded by $n$ with 
$|\fm_1|=C\setminus U_1$ and $a\in \sM(C,\fm_1)$. Therefore we can replace $(U',\fm')$
by $(U_1,\fm_1)$ and since $U\setminus U_1=\{P_2,\ldots, P_r\}$ we conclude by induction.
This proves the claim and hence the theorem.
\end{proof}

\subsection{Reciprocity functors}

\begin{definition}\label{defn-RF}
An {\em $R$-reciprocity functor} (or just reciprocity functor) is a $R$-Mackey functor with specialization maps $\sM$ 
such that
for any $C\in (\sC/S)$, any non-empty open subset $U\subset C$ and any section $a\in \sM(U)$ there exists an 
effective divisor $\fm$ on $C$ which has support equal to $C\setminus U$ and is a modulus for $a$, 
i.e. for all $\emptyset\neq U\subset C$ we have
\eq{defn-RF1}{\sM(U)=\bigcup_{|\fm|=C\setminus U} \sM(C,\fm),}
where the union is over all effective divisors on $C$ with  support equal to $C\setminus U$.
(In particular we have $\sM(C)=\sM(C, 0)$ and $\sM(\eta_C)=\bigcup_{\fm} \sM(C,\fm)$, where the union is over all effective divisors on $C$.)

We denote by $\RF_R=\RF$ the full subcategory of $\MFsp$ whose objects are reciprocity functors.
\end{definition}

\begin{remark}\label{rmk-defn-RF}
Let $\sM$ be a Mackey functor with specialization maps.
By Theorem \ref{thm-modulus-Nis-local} the condition \eqref{defn-RF1} is satisfied for all $C\in (\sC/S)$ and all non-empty open subsets $U\subset C$ if and only if for all $C\in (\sC/S)$ and all points $x\in C$
(closed or not) we have
\eq{defn-RF2}{\sM^h_{C,x}= \varinjlim_{(D,\fn)} \sM(D,\fn),} 
where the limit is over all pairs $(D,\fn)$  with $D\to C$ finite flat in $(\sC/S)$ and $\fn$ is an effective divisor on $D$ such that $D\setminus |\fn|$ is a Nisnevich neighborhood of $x$
(i.e. it is \'etale over $C$ and there is a point $y\in D\setminus |\fn|$, which maps isomorphically to $x$).
\end{remark}

\begin{propositiondefinition}\label{prop-defn-localsymbol}
Let $\sM$ be a reciprocity functor.
Then for any $C\in (\sC/S)$  and $P\in C$ there exists a biadditive pairing
\[(-,-)_P: \sM(\eta_C)\times \G_m(\eta_C)\to \sM(x_C),\]
which is $R$-linear in the first argument and has the following properties:
\begin{enumerate}
 \item $(-,-)_P$ is continuous when $\sM(\eta_C)$ and $\sM(x_C)$ are endowed with the discrete topology and $\G_m(\eta_C)$ with the topology for which $\{U_P^{(n)}\,|\, n\geqslant 1\}$ is a fundamental system of open 
       neighborhoods of $1$. 
 \item For all $a\in \sM_{C,P}$ (notation as in \ref{not-pb-pf-sp}) and $f\in \G_m(\eta_C)$ we have
            \[(a,f)_P= v_P(f)\Tr_{P/x_C}(s_P(a)). \]
 \item For all $a\in \sM(\eta_C)$ and $f\in \G_m(\eta_C)$ we have 
           \[\sum_{P\in C} (a, f)_P=0.\] 
\end{enumerate}
Furthermore $(-,-)_P$ is uniquely determined by the properties above. We call $(-,-)_P$ the {\em local symbol at $P$ attached to $\sM$}.
\end{propositiondefinition}

\begin{proof}
The Zariski-stalk $\sM_{C,P}$ is a submodule of $\sM(\eta_C)$ and since $\sM_C$ is a Zariski sheaf on $C$, we have $\sM_C(U)=\cap_{P\in U}\sM_{C,P}$ for all open subsets $U\subset C$.
Thus we are in the situation of Proposition \ref{prop-modulus=symbol}. For $a\in \sM(\eta_C)$ choose a modulus $\fm$, then condition  (2) in Proposition \ref{prop-modulus=symbol} is satisfied.
Let $\{\rho_{a,P}:\G_m(\eta_C)\to \sM(x_C)\}_{P\in C}$ be the family of continuous homomorphisms from Proposition \ref{prop-modulus=symbol}, (1) constructed for $(a,\fm)$.
Notice that this family does not depend on the choice of the modulus $\fm$ for $a$: Indeed if $\fm'$ is another modulus for $a$ then so is $\fm''=\fm+\fm'$. But the family $\{\rho_{a,P}''\}$ constructed
via $\fm''$ satisfies in particular, the same properties as $\{\rho_{P,a}\}$ and so by uniqueness they have to be the same.
Then for $a\in \sM(\eta_C)$ and $f\in \G_m(\eta_C)$ we set
\[(a,f)_P:= \rho_{a,P}(f).\]
It follows from the uniqueness of the $\rho$'s that this defines a biadditive pairing, which is $R$-linear in the first argument. The properties (1)-(3) now follow immediately from the corresponding property of the $\rho$'s.
\end{proof}

\begin{corollary}[cf. {\cite[III,\S1, Prop. 2]{SerreGACC}}]\label{cor-LS-functoriality}
Let $\Phi: \sM\to \sN$ be a morphism between reciprocity functors. Then for all $C\in (\sC/S)$, $P\in C$, $a\in \sM(\eta_C)$ and $f\in \G_m(\eta_C)$ we have
       \[\Phi((a,f)^\sM_P)=(\Phi(a),f)^\sN_P.\]
\end{corollary}
\begin{proof}
Follows from the uniqueness statement in Proposition \ref{prop-modulus=symbol}.
\end{proof}

\begin{proposition}[cf. {\cite[III, Prop. 3 and 4]{SerreGACC}}]\label{prop-pfpb-localsymbol}
Let $\sM$ be a reciprocity functor and $\pi: D\to C$ a finite morphism in $(\sC/S)$.
\begin{enumerate}
 \item For all $b\in \sM(\eta_D)$, $f\in \G_m(\eta_C)$ and $P\in C$ we have in $\sM(x_C)$
        \[(\pi_*b, f)_P=\sum_{Q\in \pi^{-1}(P)}\Tr_{x_D/x_C}(b,\pi^*f)_Q.\]
 \item For all $a\in \sM(\eta_C)$, $g\in \G_m(\eta_D)$ and $P\in C$ we have in $\sM(x_C)$
       \[(a,\pi_*g)_P=\sum_{Q\in\pi^{-1}(P)}\Tr_{x_D/x_C}(\pi^*a,g)_Q.\]
\end{enumerate}
\end{proposition}

\begin{proof}The proof is along the lines of the proof of \cite[III, Prop. 3 and 4]{SerreGACC}. \end{proof}

\begin{definition}\label{defn-filtration}
Let $\sM$ be a reciprocity functor, $C\in (\sC/S)$ a curve and $P\in C$ a closed point.
Then we define 
\[\Fil^0_P \sM(\eta_C):=\sM_{C,P}\]
and for $n\geqslant 1$
\[\Fil^n_P \sM(\eta_C):=\{a\in \sM(\eta_C)\,|\, (a,u)_P=0\text{ for all } u\in U^{(n)}_P\}.\]
Clearly $\Fil^\bullet_P \sM(\eta_C)$ forms an increasing and exhaustive filtration of sub-$R$-modules of $\sM(\eta_C)$.
\end{definition}

\begin{lemma}\label{lem-alt-descr-modulus-space}
Let $\sM$ be a reciprocity functor and $\fm=\sum_{P\in C} n_P P$  an effective divisor on a curve  $C\in (\sC/S)$. Then
\[\sM(C,\fm)= \bigcap_{P\in C} \Fil^{n_P}_P \sM(\eta_C)= \bigcap_{P\in |\fm|} \Fil^{n_P}_P \sM(\eta_C)\cap \sM(C\setminus|\fm|).\]
\end{lemma}

\begin{proof}
The second equality is clear and the inclusion $\supset$ for the first equality follows immediately from reciprocity.
For the other inclusion take $a\in \sM(C,\fm)$, $P\in |\fm|$ and $u\in U^{(n_P)}_P$.  Take $N_Q\geqslant n_Q$ with $(a, U^{(N_Q)}_Q)_Q=0$ for all $Q\in |\fm|$ and choose
$u_P\in \G_m(\eta)$ with $u_P/u\in U^{(N_P)}_P$ and $u_P\in U^{(N_Q)}_Q$ for $Q\in |\fm|\setminus\{P\}$. Then $u_P\equiv 1$ mod $\fm$ and
\[(a,u)_P=(a,u_P)_P=\sum_{Q\in |\fm|} (a,u_P)_Q=0.\]
Thus $a\in {\rm Fil}^{n_P}_P \sM(\eta)$, which gives the statement.
\end{proof}

\begin{definition}\label{defn-RFn}
For $n\geqslant 0$ we define $\RF_n$ to be the full subcategory of $\RF$ formed by the reciprocity functors $\sM$, which have the property that for any curve $C\in(\sC/S)$ and any closed point $P$ in $C$
$$\Fil^n_P M(\eta_C)=M(\eta_C).$$
\end{definition}

\begin{example}\label{ex-RF0RF1}
Let $\sM$ be a reciprocity functor. Then we have 
\begin{enumerate}
 \item $\sM\in \RF_0$ iff $\sM(\eta_C)=\sM(C)$ for all $C\in (\sC/S)$.
 \item $\sM\in \RF_1$ iff  $s^\sM_0=s^{\sM}_1:\sM_{\P^1_x}(\A^1_x)\to M(x)$.
\end{enumerate}
Indeed (1) follows immediately from the definition. For (2) first notice that $\sM\in \RF_1$ iff
for all $C\in (\sC/S)$,  all non-empty open subsets $U\subset C$ we have 
$\sM(U)=\sM(C,\sum_{P\in C\setminus U}[P])$, as follows from Lemma \ref{lem-alt-descr-modulus-space}.
Now the statement follows from Lemma \ref{lem-alt-defn-MC} up to applying the isomorphism 
$\A^1_x\cong \P^1_x\setminus\{1\}$, which sends $0,1$ to $0,\infty$, respectively.
(For this $\Rightarrow$ direction apply Lemma \ref{lem-alt-defn-MC} in the case
 $\gamma=$ graph of identity $\in \Cor((\P^1_x,\{1\}),(\P^1_x,\{1\}))$.)
\end{example}

\section{Examples}
The examples in \S\ref{EXA} and \S\ref{EXB} were suggested by \name{B. Kahn}.
\subsection{Constant reciprocity functors}\label{ex-ModisRF}
Let $M$ be an $R$-module. Then $M$ defines a constant Nisnevich sheaf with transfers on $\Reg$, defined by $M(X):=M^{\oplus r}$, for $X\in \Regone$ with $r$ connected components, and 
an elementary correspondence $Z\in \Cor(X,Y)$ acts via multiplication with its degree $\deg(Z/X)$.
One easily checks that in this way $M$ defines a reciprocity functor and that we obtain
a fully faithful functor
\[(R-{\rm mod})\to \RF_0\subset \RF.\]

\subsection{Algebraic groups}\label{EXA}
In the following by an {\em algebraic group} we mean a smooth connected and commutative group scheme over $S$.
\subsubsection{}\label{trace-for-alg-gps}
{\em Trace for algebraic groups.} Let $G$ be an algebraic group and $\pi: Y\to X$ a finite and flat morphism of degree $d$ between noetherian $S$-schemes. 
Then (see e.g.  \cite[Exp. XVII, (6.3.4.2)]{SGA4III}) there exists a canonical $X$-morphism
\eq{trace-for-alg-gps1}{X\to {\rm Sym}_X^d Y.}
We quickly recall how this is constructed locally: Assume $\pi$ corresponds to a ring map $A\to B$ which makes $B$ a free $A$-module of rank $d$.
For a not necessarily commutative $A$-algebra $D$ denote by $\mathrm{TS}^d_A(D)=(D^{\otimes_A d})^{\Sigma_d}$ the $A$-algebra of symmetric tensors.
Then (see e.g. \cite[Exp. XVII, (6.3.1.4)]{SGA4III}) the determinant ${\rm End}_A(B)\to A$ induces an homomorphism of $A$-algebras
$\mathrm{TS}^d_A({\rm End}_A(B))\to A$, whose composition with ${\rm End}_A(B)\to \mathrm{TS}^d_A({\rm End}_A(B))$, $x\mapsto x\otimes\ldots\otimes x$, gives back
the determinant. Then \eqref{trace-for-alg-gps1} is locally given by taking $\Spec$ of the composition
\[\mathrm{TS}^d_A(B)\to\mathrm{TS}^d_A({\rm End}_A(B))\to A,\]
where the first map is induced by the natural map $B\to {\rm End}_A(B)$.

Now in \cite[Exp. XVII, (6.3.13.2)]{SGA4III} the trace morphism 
\[\pi_*= \Tr_{Y/X}: G(Y)\to G(X)\]
is defined as follows: Let $u: Y\to G$ be a $S$-morphism, then $\Tr_{Y/X}(u): X\to G$ is given by the composition
\[X\xr{\eqref{trace-for-alg-gps1}} {\rm Sym}_X^d Y\xr{\sum_{1=1}^d u} G. \] 

By \cite[Exp. XVII, Ex. 6.3.18]{SGA4III} the trace thus defined equals the usual trace $\Gamma(Y,\sO_Y)\xr{\Tr_{Y/X}}\Gamma(X,\sO_X)$ (resp. norm $\Gamma(Y,\sO_Y^\times)\xr{\Nm_{Y/X}} \Gamma(X,\sO_X^\times)$)
for $G=\G_a$ (resp. $G=\G_m$).

\begin{proposition}[cf. {\cite[III]{SerreGACC}}]\label{prop-alg-gp-RF}
Let $G$ be an algebraic group. Then the Nisnevich sheaf  $\Regone\ni U\mapsto G(U)$ extends to a functor on $\Regone\Cor$ (via the trace recalled above) and
as such is a $\Z$-reciprocity functor.
Furthermore a morphism of algebraic groups induces a morphism of the corresponding reciprocity functors.
\end{proposition}

\begin{proof}
It is well-known that $U\to G(U)$ is a Nisnevich sheaf on $\Regone$. Further, if $U\in \RegCon$ is 1-dimensional with generic point $\eta$
the natural map $G(U)\to G(\eta)$ is injective. (This follows from the valuative criterion for properness, since the neutral section $e_G: S\to G$ is a proper map.)
Thus the sheaf $G$ satisfies the conditions (Inj) and (FP) from Definition \ref{defn-MFsp}. Furthermore by \cite[Exp. XVII, Prop. 6.3.15]{SGA4III} the trace recalled above yields a functor
$G: \RegCon_*\to (\Z-{\rm mod})$, such that for a finite and surjective morphism $\pi: X\to Y$ in $\RegCon$ the composition $\pi_*\pi^*$ is multiplication with $\deg \pi$ on $G(Y)$.
Thus $G$ is a Mackey functor with specialization maps if it satisfies condition (3) of Lemma \ref{CarPSTRegone}. So let $Y, X, X'$ and $f, g$ be as in Lemma \ref{CarPSTRegone}, (3).
Since $G$ satisfies (Inj) we may assume that $X'$ actually is an $S$-point (Then $g: X'\to X$ is either dominant or the inclusion of a closed point.)
By \cite[Exp. XVII, Prop. 6.3.15]{SGA4III} the trace is compatible with base change and decomposes as a sum on disjoint schemes.
Hence it suffices to show, that if $k$ is a field and $A$ is a finite local $k$-algebra, then $\Tr_{A/k}: G(A)\to G(k)$ equals the composition
\[G(A)\to G(A_{\rm red})\xr{\length(A)\cdot \Tr_{A_{\rm red}/k}} G(k).\]
This follows immediately from \cite[Exp. XVII, Prop. 6.3.5]{SGA4III}, which implies that if $d_0=[A_{\rm red}:k]$
and $d=[A:k]=d_0\cdot\length(A)$, then the map $\Spec k\to {\rm Sym}_k^d (\Spec A)$ factors via 
\[\Spec k \to ({\rm Sym}^{d_0}_k (\Spec A_{\rm red}))^{\times \length(A)}\xr{\rm can.} {\rm Sym}_k^d (\Spec A),\]
where the arrow labeld ``can.'' is the canonical morphism induced by the composition
\mlnl{\mathrm{TS}^d_k(A)\to \mathrm{TS}^d_k(A_{\rm red})\xr{\mathrm{TS}({\rm diagonal})}  \mathrm{TS}^d_k( A_{\rm red}^{\oplus \length(A)}))\\
      \cong \bigoplus_{\sum n_i=d} \mathrm{TS}^{n_1}_k(A_{\rm red})\otimes_k\ldots \otimes_k \mathrm{TS}^{n_{\length(A)}}_k(A_{\rm red})
        \xr{\rm proj. } \mathrm{TS}^{d_0}_k(A_{\rm red})^{\otimes \length(A)}. }
Thus $G$ is a Mackey functor with specialization maps. It remains to check that for $C\in (\sC/S)$ and any non-empty open subset $U\subset C$ any section $a\in G(U)$ admits a modulus (see Definition \ref{defn-modulus}).
In case $\kappa(x_C)$ is algebraically closed this is a theorem due to M. Rosenlicht (see \cite[III, \S 1, Thm. 1]{SerreGACC}). The general case follows from this as follows:
Let $\sigma:\bar{x}\to x:=x_C$ be a geometric point (i.e. $\kappa(\bar{x})$ is algebraically closed), denote by $\nu : D\to (C\times_x \bar{x})_{\rm red}$ the normalization map
(notice that $(C\times_x \bar{x})_{\rm red}$ is integral since $C\to x$ is geometrically connected) and by $\pi$ the composition
\[\pi: D\xr{\nu} (C\times_x\bar{x})_{\rm red}\subset C\times_x\bar{x}\xr{\id\times\sigma} C.\]
Take $a\in G(U)$. Then by the theorem of M. Rosenlicht $\pi^*a\in G(\pi^{-1}(U))$ has a modulus and thus we find an effective divisor $\fm$ on $C$ with support equal to $C\setminus U$ and such that $\pi^*a$ has modulus $\pi^*\fm$.
We claim that $a$ has modulus $\fm$. For this take $f\in \G_m(\eta_C)$ with $f\equiv 1$ mod $\fm$ (in particular $\pi^*f\equiv 1$ mod $\pi^*\fm$ ). 
Since $\sigma^*:G(x)\to G(\bar{x})$ is injective ($G$ being an fppf-sheaf) it suffices to show 
\[0=\sum_{P\in C\setminus|\fm|} v_P(f)\sigma^*\Tr_{P/x}(s_P(a))= \sum_{P\in C\setminus|\fm|}\sum_{P'\in P\times_x \bar{x}}v_P(f) \cdot l_{P'} \cdot\sigma_{P'}^*s_P(a), \]
where $\sigma_{P'}: P'\to P$ is induced by $\sigma$ and $l_{P'}=\length(\sO_{P\times_x\bar{x}, P'})$; notice that we can drop the $\Tr_{P'/\bar{x}}$ on the right since $P'\cong\bar{x}$.
For this, let $t\in \sO_{C,P}$ be a local parameter, then by \cite[Lem. 1.7.2, Ex. 1.2.3]{F} we have 
\[l_{P'}\cdot [P'] = [\Div(t_{|C\times_x \bar{x}})]_{|P'}= l_{\eta_D}\cdot\sum_{Q\in \nu^{-1}(P')} e(Q/P) )\cdot[P'],\]
where $l_{\eta_D}:=\length(\sO_{C\times_x\bar{x},\eta_D})$ (notice that $[Q:P']=1$).  Thus we obtain
\[ \sum_{P\in C\setminus|\fm|}\sum_{P'\in P\times_x \bar{x}} v_P(f) \cdot l_{P'} \cdot \sigma_{P'}^*s_P(a) = l_{\eta_D} \sum_{Q\in D\setminus|\pi^*\fm|}  v_Q(\pi^*f) s_Q(\pi^*a),\]
which is zero by our choice of $\fm$. This finishes the proof.
\end{proof}

\begin{remark}\label{rmk-gps-Fil}
It follows from the proof of \cite[III, \S1, Theorem 1]{SerreGACC}, the above proof and Example \ref{ex-RF0RF1} that 
\[G\in \begin{cases}
        \RF_0, &\text{if $G$ is an Abelian variety}\\
        \RF_1, &\text{if $G$ is a semi-Abelian variety}\\
        \RF\setminus\cup_n \RF_{n\geqslant 0}, & \text{if $G$ is unipotent.}
       \end{cases}
\]
\end{remark}

\subsection{Homotopy invariant Nisnevich sheaves with transfers}\label{EXB}

\subsubsection{Recollections on PST}\label{PST}

Let $\SmCor_S$ be the category defined in \cite[2.1]{VoDM} whose objects are smooth $S$-schemes and the morphisms are given by finite correspondences, i.e.
$\Hom_{\SmCor_S}(X,Y)=\Cor_S(X,Y)$.
A {\em presheaf with transfers over $S$} is a contravariant and additive functor from $\SmCor_S$ to the category of Abelian groups, $\sF: (\SmCor_S)^\op\to ({\rm Ab})$.
The category of presheaves with transfers is denoted by $\PST$; it is an Abelian category.
Inside $\PST$ we have the full subcategories $\NST$ of Nisnevich sheaves with transfers, $\HI$  of homotopy invariant presheaves with transfers and $\HINis=\HI\cap\NST$ of homotopy invariant Nisnevish sheaves with transfers. 
The inclusion functor $\HI\ra\PST$ admits a left adjoint 
$$h_0:\PST\ra\HI $$
where $h_0(\sF)$ is the presheaves with transfers defined for any smooth $S$-scheme $X$ by
$$h_0(\sF)(X):=\Coker\left[\sF(\A^1_X)\xrightarrow{s_0^*-s_1^*} \sF(X)\right].$$
By \cite[Theorem 3.1.4]{VoDM} the sheafification functor $\sF\mapsto \sF_\Nis$ restricts to a functor $\PST\ra\NST$ which is the left adjoint to the inclusion functor $\NST\ra\PST$.

\begin{definition}\label{defn-model}
A {\em model} of $X\in\Reg$ is a separated and smooth $S$-scheme $U$ together with an affine morphism $u:X\to U$ such that for any open affine $\Spec A\subset U$ the ring $\Gamma(\Spec A, u_*\sO_X)$ is a localization
of $A$ by a multiplicative subset.  A morphism between two models is a morphism of schemes under $X$. We denote by $\fM_X$ the category of all models of $X$.
\end{definition}

\begin{lemma}\label{defn-models-cofil}
The category $\fM_X$ of models of $X\in \Regone$ is cofiltered and in particular non-empty.
\end{lemma}

\begin{proof}
 This is straightforward.
\end{proof}

\begin{proposition}\label{prop-PSTtoPT}
Let $\sF\in\PST$ be a presheaf with transfers.
Then 
$$\Regone\ni X\mapsto \hat{\sF}(X):=\colim_{U\in\fM_X^\op}\sF(U)$$
naturally defines a presheaf with transfers on $\Regone$ (in the sense of Definition \ref{defn-PT}, with $R=\Z$)
and we obtain an exact functor
\[\PST\to \PT,\quad \sF \mapsto \hat{\sF},\]
which restricts to $\NST\to \NT$.

\end{proposition}

\begin{proof}
For $X,Y\in \Regone$ and $[V]\in \Cor(X,Y)$ an elementary correspondence, we find models $X\to U$ and $Y\to U'$ and an elementary correspondence $[W]\in \Cor(U,U')$ which pulls back to $[V]$ under $X\times Y\to U\times U'$.
We get a map $\sF([W]): \sF(U')\to \sF(U)$ and it is straightforward to see that it induces a map $\hat{\sF}(Y)\to \hat{\sF}(X)$ which is independent of the choices of $U, U', W$ and is compatible with composition.
This proves the first statement and the other statements are immediate.
\end{proof}


The following statement is a collection of results of V. Voevodsky.
 
\begin{proposition}\label{HI=RF0}
The above functor $\PST\to \PT$ restricts to a conservative functor
\[\HI_{\Nis}\to \RF_1, \quad \sF\mapsto \hat{\sF},\]
where $\RF_1$ is defined in \ref{defn-RFn}.
Furthermore, restricting $\hat{\sF}$ to $\pt\Cor^\op$ we obtain a functor $\HI_\Nis\to \MF$, which still is conservative.
\end{proposition}

\begin{proof}
Take $\sF\in \HI_\Nis$. By Proposition \ref{prop-PSTtoPT} $\hat{\sF}$ is a Nisnevich sheaf with transfers on $\Regone$ and it automatically satisfies condition (FP) from Definition \ref{defn-MFsp}.
Further, by \cite[Corollary 4.19]{VoPST} and \cite[Proposition 3.1.11]{VoDM}, the restriction map $\sF(X)\inj \sF(U)$, for $U\inj X$ a dense open immersion in $\Sm_S$, is injective. 
This implies (Inj) from Definition \ref{defn-MFsp} and hence $\hat{\sF}\in \MFsp$. Furthermore, since $\sF$ is homotopy invariant we have $i_0^*=i_\infty^*$ on $\hat{\sF}(\P^1_x\setminus\{1\})$ with $x$ an $S$-point.
Thus $\hat{\sF}\in\RF_1$ by Lemma \ref{lem-alt-defn-MC}. The conservativity statement follows immediately from
\cite[Proposition 4.20]{VoPST} and \cite[Proposition 3.1.11]{VoDM}.
\end{proof}

\subsection{Connection with classical class field theory in the function field case}

{\em In this section we assume that $F$ is a perfect field of positive characteristic $p$.}

We denote by $W_n$ the Witt-scheme of length $n$ over $S$. 
It is in particular a unipotent algebraic group in the sense of \ref{EXA} and hence defines a reciprocity functor.
For an $F$-algebra $A$, the $A$-rational points of $W_n$ form the ring of Witt vectors $W_n(A)$. 
There is a Frobenius morphism
${\rm F}: W_n\to W_n$ which sends an $A$-rational point $(a_0,\ldots, a_{n-1})\in W_n(A)$ to $(a_0^p,\ldots, a_{n-1}^p)$.
We also have the morphism of algebraic groups 
\eq{p-map}{p:W_n\to W_{n+1},\quad (a_0,\ldots, a_{n-1})\mapsto p\cdot (a_0,\ldots, a_{n-1},0),}
which commutes with the Frobenius ${\rm F}$.

\begin{proposition}\label{prop-H1=RF}
Let $\sH^1$ be the Nisnevich sheafification of the presheaf 
\[\Reg\ni X\mapsto H^1_{\rm \'et}(X,\Q/\Z).\]
Then $\sH^1$ is a reciprocity functor and it decomposes in $\RF$ as 
\eq{prop-H1=RF1}{\sH^1= (\sH^1)'\oplus \varinjlim_n \frac{W_n}{({\rm F}-1)W_n},}
where $(\sH^1)'$ is induced by a homotopy invariant Nisnevich sheaf with transfers (in the sense of Proposition \ref{HI=RF0}),
$W_n/({\rm F}-1)W_n$ is the quotient in the category of Nisnevich sheaves  and the inductive limit is over the morphisms
\[p: \frac{W_n}{(F-1)W_n}\to \frac{W_{n+1}}{(F-1)W_{n+1}}\]
induced by \eqref{p-map}.
\end{proposition}

\begin{proof}
Set $(\Q/\Z)':=\oplus_{\ell\neq p}\Q_\ell/\Z_\ell$, where the sum is over all prime numbers $\ell$ unequal to $p$. We obtain
$\Q/\Z= (\Q/\Z)'\oplus \varinjlim_n \Z/p^n\Z$, where the transition maps $\Z/p^n\Z\to \Z/p^{n+1}\Z$ are induced by 
multiplication with $p$. Using Artin-Schreier-Witt theory we see that for an affine $F$-scheme $X=\Spec A$ 
we have the decomposition
\[H^1_{\rm \'et}(X, \Q/\Z)=H^1_{\rm \'et}(X, (\Q/\Z)')\oplus \varinjlim_n \frac{W_n(A)}{({\rm F}-1)W_n(A)}. \]
Denote by $(\sH^1)'$ the Nisnevich sheafification of $\Reg\ni X\to H^1_{\rm \'et}(X, (\Q/\Z)')$. This
gives the decomposition \eqref{prop-H1=RF1} in the category of Nisnevich sheaves. 
It follows from \cite[Lemma 6.21]{MVW}, \cite[Corollary 5.29]{VoPST}, \cite[Theorem 3.1.12]{VoDM} and 
Proposition \ref{HI=RF0}, that $(\sH^1)'$ is a reciprocity functor induced by a sheaf in $\HI_{\Nis}$.
Further $W_n/({\rm F}-1)W_n$ clearly is a Nisnevich sheaves with transfers, which satisfies (FP) and (MC).
For (Inj) it suffices to check that if $A$ is a henselian DVR with fraction field $K$ and we are
given a Witt vector $(a_0,\ldots, a_{n-1})\in W_n(A)$ such that there exist elements $b_i\in K$ with
\[(a_0,\ldots, a_{n-1})= (b_0^p,\ldots, b_{n-1}^p)- (b_0,\ldots, b_{n-1}), \]
then the $b_i$'s are in $A$. Suppose we already showed that for $i\ge 0$ the elements $b_0,\ldots, b_{i-1}$ are in $A$.
Then we get
\[a_i= b_i^p-b_i+ P(b_0,\ldots, b_{i-1})\]
for some polynomial $P\in \Z[x_0,\ldots, x_{i-1}]$. Thus $b_i(b_i^{p-1}-1)\in A$, i.e. $b_i\in A$. 
Thus $W_n/({\rm F}-1)W_n$ is a reciprocity functor and hence so is $\varinjlim_n W_n/({\rm F}-1)W_n$. This 
finishes the proof.\end{proof}

\subsubsection{}
Now assume that $F$ is a finite field and $C$ is a smooth projective geometrically connected
curve over $S=\Spec F$ with function field $K$
and generic point $\eta$.
For a closed point $P\in C$ denote by $K_P$ the completion of $K$ along $P$.
Further denote by $G^{\rm ab}(K)$, $G^{\rm ab}(K_P)$ and $G^{\rm ab}(F)$ the maximal abelian quotient of 
the absolute Galois groups of the fields $K, K_P, F$, respectively. We have
\eq{H1(S)}{\sH^1(S)=\Hom_{\rm conts}(G^{\rm ab}(F), \Q/\Z)=\Hom_{\rm conts}(\hat{\Z},\Q/\Z)=\Q/\Z.}
Thus the {\em local reciprocity  map}   (see e.g. \cite[XIII, \S 4]{SerreLF})
\[\rho_P: K_P^\times\to G^{\rm ab}(K_P)\cong \Hom(H^1_{\rm \'et}(K_P,\Q/\Z),\Q/\Z) \]
induces a morphism
\eq{CFT-symbol}{\sH^1(\eta)\times \G_m(\eta)\to \Q/\Z\cong \sH^1(S), \quad (\chi,f)\mapsto \rho_P(f)(\chi).}
Notice, if we identify $\chi$ with a character $\chi: G^{\rm ab}(K)\to \Q/\Z$, then $\rho_P(f)(\chi)= \chi(\rho_P(f))$.

\begin{proposition}\label{prop-CFT-symbol}
With the notation above, 
the symbol at $P$ attached to $\sH^1$ (see Proposition-Definition \ref{prop-defn-localsymbol})
is given by \eqref{CFT-symbol}, i.e.
\[(\chi, f)_P=\rho_P(f)(\chi),\quad f\in \G_m(\eta), \chi\in \sH^1(\eta).\]
Furthermore the filtration of $\sH^1(\eta)$ at $P$ (see Definition \ref{defn-filtration}) is given by
\[{\rm Fil}_P^n\sH^1(\eta)= \Hom_{\rm conts}(\frac{G^{\rm ab}(K)}{G^{{\rm ab},n}(K_P)},\Q/\Z),\]
where $G^{{\rm ab},n}(K_P)$ is the $n$-th ramification subgroup of $G^{\rm ab}(K_P)$ in the upper numbering
(see \cite[IV, \S3]{SerreLF}).
\end{proposition}

\begin{proof}
For the first statement we have to check that the symbol \eqref{CFT-symbol} satisfies the properties (1)-(3)
of Proposition-Definition \ref{prop-defn-localsymbol}. By \cite[XV, \S 2, Remark after Theorem 2]{SerreLF}
we have 
\eq{prop-CFT-symbol1}{\rho_P(\hat{U}^{(n)}_P)=G^{\rm ab, n}(K_P),}
where $\hat{U}^{(n)}_P$ denotes the completion of $U_P^{(n)}$. Now since any $\chi\in \sH^1(\eta)$
is in fact induced by a character $\chi': {\rm Gal}(L/K)\to \Q/\Z$ for some finite abelian Galois extension $L/K$
and since the images of $G^{\rm ab, n}(K_P)\to {\rm Gal}(L/K)$ are zero for $n>>0$ we see that
the continuity condition (1) is satisfied. For (2) notice that under the isomorphism \eqref{H1(S)} the 
element $v_P(f)\Tr_{P/S}(s_P(\chi))$, $\chi\in \sH^1_{C,P}, f\in \G_m(\eta)$, corresponds to
\[v_P(f)\chi(F_S^{[P:S]})= v_P(f)\chi(F_P)=\chi(\rho_P(f)),\]
where $F_S\in G^{\rm ab}(F)$ 
(resp. $F_P\in G^{\rm ab}(\kappa(P))\cong\pi_1^{\rm ab}(\Spec \sO_{C,P}^h)$) denotes the Frobenius element
and the second equality holds by \cite[XIII, \S 4, Proposition 13]{SerreLF}. Finally (3) is the Artin reciprocity law,
see e.g. \cite[VII, \S 3]{AT}. The second statement now follows directly from the definition of 
${\rm Fil}^n_P\sH^1(\eta)$ and \eqref{prop-CFT-symbol1}.
\end{proof}

\subsection{Cycle modules}

\subsubsection{}
Notice that the fact that cycle modules define reciprocity functors, may be seen as a consequence of the fact that homotopy invariant sheaves with transfers define reciprocity functors and the work of \name{F. D\'eglise} \cite{Deglise}. However we do not need the equivalence between the category of cycle modules and the category of homotopy modules of \cite{Deglise} to see this, as it follows from the axioms of cycle modules. 
In the following we will freely use the numbering of the data and rules a cycle module should satisfy as introduced by \name{M. Rost} in \cite{Rost}. We will check that a cycle module defines a reciprocity functor in the sense of Definition \ref{defn-RF}. For this we will use Proposition \ref{prop-alt-MFsp}, which provides a description of Mackey functor with specialization which is closer to the definition of cycle modules. 

\subsubsection{Cycle modules are graded reciprocity functors}

Let $M_*$ be a cycle module as defined in \cite[(2.1) Definition]{Rost}. Let $n\in \Z$ be an integer. First note that $M_n$ is a Mackey functor by {\bf{R1a-c}} and {\bf{R2d}}, see Remark \ref{rmk-alt-defn-MF}.
As we shall see, the Mackey functor $M_n$ is canonically the underlying Mackey functor of a Nisnevich sheaf with transfers on $\Regone$ which satisfies the modulus condition (MC). 
Let $C\in\sC/S$ be a curve. To a closed point $P$  in $C$ corresponds a discrete valuation $v_P$ on the function field $\kappa(\eta_C)$ of the curve, and thus a residue homomorphism {\bf{D4}}
$$\partial_P:=\partial_{v_P}:M_n(\eta_C)\ra M_{n-1}(P).$$
Given an open subset $U\subseteq C$, we set 
\begin{equation}\label{DefRegularCM}
(\sM_n)_C(U):=\bigcap_{P\in U}\ker\left(\partial_{P}:M_n(\eta_C)\ra M_{n-1}(P)\right).
\end{equation}
This in fact defines a sheaf $(\sM_n)_C$ in the Zariski topology on C and the stalk at $P$ is the kernel of $\partial_P$.

\subsubsection{}\label{sp-Rostcycle}
 Let $\varpi$ be a uniformizer of the local ring $\sO_{C,P}$, and consider as in \cite[p. 329]{Rost}, the map
$$s^\varpi_P:M_n(\eta_C)\ra M_n(P),\quad a\mapsto s^\varpi_P(a):=\partial_P(\{-\varpi\}\cdot a).$$
The restriction to $(\sM_n)_{C,P}$ is independent of the choice of the uniformizer $\varpi$. Indeed if $\varpi'$ is another uniformizer of the local ring $\sO_{C,P}$, then there exists a unit $u\in\sO_{C,P}^\times$ such that $\varpi'=u\varpi $ and using {\bf{R3e}}, we get for any $a\in(\sM_n)_{C,P}$
\begin{equation*}
\begin{split}
s^{\varpi'}_P(a) &=\partial_P(\{-\varpi'\}\cdot a)=\partial_P(\{-u\varpi\}\cdot a)=\partial_P(\{-\varpi\}\cdot a)+\partial_P(\{u\}\cdot a)\\
&=s^\varpi_P(a)-\{\bar u\}\cdot\partial_P(a)=s^\varpi_P(a).
\end{split}
\end{equation*}
Hence we may denote by
$$s_P:(\sM_n)_{C,P}\ra M_n(P)$$
the restriction of the map $s^\varpi_P$.

\begin{proposition}\label{prop-cycle-modules-are-RF}
Let $M_*$ be a cycle module over $S$. Then for any integer $n\in\Z$, the triple $(M_n,\sM_n,s)$ defines a $\Z$-reciprocity functor over $S$ via Proposition \ref{prop-alt-MFsp}, which
lies in $\RF_1$. Furthermore, its associated symbol (see Proposition-Definition \ref{prop-defn-localsymbol}) is given by
\eq{prop-cycle-modules-are-RF1}{M_n(\eta_C)\times \G_m(\eta_C)\to M_n(x_C), \quad (a,f)\mapsto (a,f)_P=\Tr_{P/x_C}\partial_P(\{f\}\cdot a),}
for $C\in (\sC/S)$ and $P\in C$.
\end{proposition}

\begin{proof}
Let us first check that $\sM_n$ has the required functoriality. Let $\pi:V\ra U$ be a morphism between
1-dimensional connected schemes in $\Regone$ with compactifications $D$ and $C$, respectively.
Then the morphism $\pi^*:M_n(\eta_C)\ra M_n(\eta_D)$ induces a morphism 
$\pi^*:(\sM_n)_C(U)\ra(\sM_n)_D(V) $, by {\bf{R3a}}. Similarly, if $\pi$ is a finite morphism, 
{\bf{R3b}} ensures that the morphism $\pi_*:M_n(\eta_D)\ra M_n(\eta_C)$ induces a morphism $\pi_*:(\sM_n)_D(V)\ra(\sM_n)_C (U)$. 
It follows from {\bf{R1a}} and {\bf{R1b}} that the pushforward and the pullback thus defined are functorial.

Thus the triple $(M_n,\sM_n, s)$ is a data (1)-(3) from Proposition \ref{prop-alt-MFsp}.
It remains to check the conditions (R0)-(S3) from Proposition \ref{prop-alt-MFsp} and (MC) from Definition \ref{defn-modulus}.

Condition (RS0) is Lemma \ref{lem-CM-sheaf} below and (RS1) is an immediate consequence of the condition {\bf{(FD)}} that cycle modules are required to fulfill. Let $C\in (\sC/S)$ and $a\in M_n(x_C)$. Then by {\bf{R3c}}, we have $\partial_P(\rho^*a)=0$ for all $P\in C$, which means that (RS2) holds. 
The condition (S3) follows from {\bf{R3d}}. Let $\pi:D\ra C$ be a morphism in $(\sC/S)$, $Q\in D$ and denote by $\pi_Q: Q\to \pi(Q)=:P$ the map induced by $\pi$. Let $\varpi_Q$ (resp. $\varpi_P$) be a uniformizer of the local ring $\sO_{D,Q}$ (resp. $\sO_{C,P}$). We may write 
$$\pi^*\varpi_P=u\cdot\varpi_Q^{e(Q/P)}, $$
where $u\in\sO_{D,Q}^\times$. Relation (S1) follows from \cite[(1.9) Lemma]{Rost}, which gives for any $a\in (\sM_n)_{C,P}$
$$s^{\varpi_Q}_Q(\pi^*a) =\pi_Q^*s^{\varpi_P}_P(a) -\{\bar u\}\cdot\pi_Q^*\partial_P(a)=\pi_Q^*s_P(a).$$
Assume now that $\pi$ is finite. For any $a\in \pi_*((\sM_n)_{D})_P$, we have 
\begin{align}
s_P(\pi_*a)& = \partial_P(\{-\varpi_P\}\cdot \pi_*a)=\partial_P\pi_*\left(\{-\pi^*\varpi_P\}\cdot a\right) &\text{{by \bf{R2b}}}\notag\\
                        &= \sum_{Q\in\pi^{-1}(P)}\pi_{Q*}\partial_Q\left(\{-\pi^*\varpi_P\}\cdot a\right) &\text{{by \bf{R3b}}}\notag
\end{align}
On the other hand by {\bf{R3e}}
\begin{equation*}
\begin{split}
\partial_Q\left(\{-\pi^*\varpi_P\}\cdot a\right)&=e(Q/P)\cdot \partial_Q(\{-\varpi_Q\}\cdot a)+\partial_Q(\{u\}\cdot a)\\
&=e(Q/P)\cdot s_Q(a)-\{\bar u\}\cdot\partial_Q(a)\\
&=e(Q/P)\cdot s_Q(a).
\end{split}
\end{equation*}
This implies (S2). It remains to prove that the modulus condition (MC) is fulfilled. Let $C\in(\sC/S)$ be a curve and $U\subset C$ be a non-empty open subset and take $a\in (\sM_n)_C(U)$. The closed subset $C\setminus U$ consists then of finitely many closed points $P_1,\ldots,P_r$ and let $\fm$ be the effective divisor $\fm:=\sum_{i=1}^rP_i$. Let $f\in\G_m(\eta_C)$ be a rational function such that $f\equiv 1\mod \fm$.
By the reciprocity law for curves {\bf{RC}}, we know that
\begin{equation}\label{RCCycleMod}
\sum_{P\in C}\Tr_{P/x_C}\partial_P(\{f\}\cdot a)=0.
\end{equation}
If $P=P_i$ for some $i\in\{1,\ldots,r\}$, then by definition $f\in U^{(1)}_P=1+\fm_P$. Hence $f$ is a unit and {\bf{R3e}} implies that
\begin{equation}\label{ComputationFilCyclMod}
\partial_P(\{f\}\cdot a)=-\{\bar f\}\cdot \partial_P(a)=-\{1\}\cdot\partial_P(a)=0.
\end{equation}
On the other hand, if $P\in U$ and $\varpi$ is a uniformizer of the local ring $\sO_{C,P}$, then by {\bf{R3f}} we have
\begin{equation*}
\begin{split}
\partial_P(\{f\}\cdot a)&=\partial_P(\{f\})\cdot s^\varpi_P(a)-s^\varpi_P(f)\cdot\partial_P(a)+\{-1\}\cdot\partial_P(\{f\})\cdot\partial_P(a)\\
&=\partial_P(\{f\})\cdot s^\varpi_P(a)=v_P(f)\cdot s_P(a)
\end{split}
\end{equation*}
since $a\in (\sM_n)_C(U)$. Hence the reciprocity law (\ref{RCCycleMod}), may be rewritten as
$$\sum_{P\in C\setminus|\fm|}v_P(f)\cdot \Tr_{P/x_C}(s_P(a))=0.$$
This shows that the modulus condition (MC) is fulfilled. Hence $(M_n,\sM_n, s)$ defines a reciprocity functor.
The formula for the associated local symbol follows from the computation above and the uniqueness statement in 
Proposition-Definition \ref{prop-defn-localsymbol}.
Hence $(M_n,\sM_n, s)\in \RF_1$ by (\ref{ComputationFilCyclMod}).
\end{proof}

\begin{lemma}\label{lem-CM-sheaf}
For all $C\in (\sC/S)$ the functor $\sM_n$ on the small Nisnevich site  $C_\Nis$, 
which sends an \'etale $C$-scheme $U$ to $\oplus_i(\sM_n)_{C_i}(U_i)$, 
where the $U_i$'s are the connected components of $U$ with unique compactification $C_i$ over $C$, is a sheaf.
\end{lemma}

\begin{proof}
Let $X\in\Regone$ and
$$\xymatrix{{W}\ar[r]\ar[d]\ar@{}[rd]|{\square} & {V}\ar[d]^{\pi}\\
{U}\ar[r]^j & {X}} $$
be a distinguished Nisnevich square for the Nisnevich topology \emph{i.e.} $\pi$ is an \'etale morphism, $j$ an open immersion and if $Z:=X\setminus U$ with its reduced scheme structure, then $\pi^{-1}(Z)\ra Z$ is an isomorphism. By \cite[Proposition 2.17]{MR2593671}, it is enough to check that the square
\begin{equation}\label{DistSqNis}
\xymatrix{{\sM_n(X)}\ar[r]\ar[d] & {\sM_n(V)}\ar[d]\\
{\sM_n(U)}\ar[r] & {\sM_n(W)}}
\end{equation}
is cartesian. Let for any scheme $Y\in\Regone$, $\sN_n(Y)$ be the direct sum over the irreducible components $T$ of $Y$ of the $M_n(\eta_T)$'s. Then the square analogous to (\ref{DistSqNis}) with $\sM_n$ replaced by $\sN_n$ is cartesian. Indeed this follows from the following two observations:
\begin{itemize}
\item if $C$ is an irreducible component of $X$, then either $C\cap U\neq\emptyset$ (and thus $C\cap U$ is an irreducible component of $U$), or $C$ is contained in $Z$, and therefore dominated by a unique irreducible component of $V$, which is moreover isomorphic to $C$;
\item since $\pi$ is \'etale, any irreducible component $D$ of $V$ dominates an irreducible component $C$ of $X$, and $C\cap U\neq\emptyset$ if and only if $D\cap W\neq\emptyset$.
\end{itemize}
Now to see that (\ref{DistSqNis}) is cartesian, it is enough to remark that if $P\in X$ is a closed point that does not lie in $U$, then there exists $Q\in V$ such that $\pi$ induces an isomorphism $\kappa(Q)\simeq\kappa(P)$,
then by {\bf{R3a}}  we have
$$\partial_Q\circ\pi^*=e(Q/P)\cdot \partial_P=\partial_P $$
since $\pi$ is unramified.
\end{proof}

\begin{example}\label{ex-MilnorK}
In particular the $n$-th level of Milnor $K$-theory gives a reciprocity functor for all integers $n\geqslant 0$,
\[\mathrm{K}^\mathrm{M}_n=(\KM_n, \KMsheaf_n, s).\]
Notice that for $C\in (\sC/S)$, the sheaf $\KMsheaf_{n,C}$ is defined by \eqref{DefRegularCM}. If $\kappa(x_C)$ is an infinite field
this coincides with the sheaf $\KM_n(\sO_C)$ by \cite[Thm. 7.1.]{Kerz}.

Also notice that we have
\[\mathrm{K}^{\mathrm{M}}_1=\G_m.\]
\end{example}

\subsubsection{Functors of zero cycles }\label{Chow-as-RF}

Let $X$ be a smooth quasi-projective $S$-scheme, of pure dimension $d$, and $n\in\Z$ an integer. Denote by $\CH^{d+n}(X,n)$ Bloch's higher Chow group of zero cycles, see \cite{Bloch-HigherChow}.
As explained in Proposition \ref{prop-cycle-modules-are-RF}, cycle modules are graded reciprocity functors. 
This implies in particular that the functor
\begin{equation}\label{RFChow}
\CH^{d+n}(X,n): ({\rm pt}/S)\to (\Z-{\rm mod}),\quad x\mapsto \CH^{d+n}(X,n)(x):=\CH^{d+n}(X_x,n),
\end{equation}
defines a $\Z$-reciprocity functor over $S$, denoted by
\[\sC H^{d+n}(X,n).\]
Indeed this functor may be obtained from $\KM_*$ via the fibration technics of \cite[\S7]{Rost}. To see this let $M_*$ be a cycle module over $X$ (e.g. Milnor K-theory $\KM_*$) and $\varrho:X\ra S$ be the structural morphism. As shown by \name{M. Rost} in \cite[(7.3) Theorem]{Rost}, for any integer $q\in\Z$, we get a cycle module $A_q[\varrho;M]$ over $S$. This cycle module is such that for any $S$-point $x$ 
$$A_q[\varrho;M]_n(x)=A_q(X_x;M,n), $$
where $A_q(X_x;M,n)$ is the $q$-th homology of the cycle complex $C_*(X_x;M,n)$ defined in \cite[\S5]{Rost}.
Assume now $M_*=\KM_*$. 
We have then two isomorphisms
\begin{equation}\label{Chow-as-RF1}
A_0[\varrho;\KM_*]_n(x) =A_0(X_x;\KM_*,n)\simeq \HZar^d(X_x;\KMsheaf_{d+n})\simeq\CH^{d+n}(X_x,n).
\end{equation}
The first one is due to \name{K. Kato} \cite{Kato1986} (see also \cite[(6.5) Corollary]{Rost}), while the second one may be found in \cite[Theorem 5.5]{Akhtar}. Since $A_0[\varrho,\KM_*]$ is a cycle module, it follows from Proposition \ref{prop-cycle-modules-are-RF} that (\ref{RFChow}) defines a reciprocity functor.
 
\begin{remark} Let $C\in (\sC/S)$ be a curve and $P$ be a closed point in $C$. As part of the cycle module structure, we have a residue homomorphism
$$\partial_P: A_0[\rho;\KM_*]_n(\eta_C)\ra A_0[\rho;\KM_*]_{n-1}(P).$$
For $n=0$ the right hand side vanishes and so does the residue map $\partial_P$. Hence it follows from the definition (\ref{DefRegularCM}) of the regular structure on the reciprocity functor $\sC H^d(X,0):=\sC H_0(X)$ 
that for any open subset $U\subseteq C$  
$$\sC H_0(X)(U)=\CH_0(X_{\eta_C}).$$
In particular, $\sC H_0(X)\in \RF_0$.  Moreover the specialization map is the one defined \emph{e.g.} in  \cite[20.3]{F}.
\end{remark}

\subsubsection{}\label{Chow-as-HI}
By \cite[Thm 2.1]{Bloch-HigherChow} and \cite[Thm 17.21]{MVW} the higher Chow groups
$CH^i(-,n)$ have the structure of a homotopy invariant presheaf with transfers.
Let $X$ be a smooth and quasi-projective $S$-scheme of equidimension $d$ and $n\geqslant 0$. 
Then we also obtain a homotopy invariant presheaf with transfers
\[\CH^{d+n}(X,n): Sm\Cor \to (\Z-{\rm mod}),\quad  Y\mapsto \CH^{d+n}(X,n)(Y):=\CH^{d+n}(X\times Y, n),\]
where a finite correspondence $\gamma\in \Cor(Y,Z)$ acts via $\Delta_X\times\gamma\in \Cor(X\times Y, X\times Z)$,
with $\Delta_X\subset X\times X$ the diagonal.
We denote by $\CH_\Nis^{d+n}(X, n)$ its Nisnevich sheafification. By \cite[Thm 13.8]{MVW} we have
$\CH_\Nis^{d+n}(X, n)\in\HINis$. 
 Applying the functor $\widehat{\phantom{-}}:\HINis\to \RF$ from Proposition \ref{HI=RF0}
we get a canonical isomorphism of reciprocity functors
\eq{Chow-as-HI1}{\sC H^{d+n}(X,n)\cong \CH_\Nis^{d+n}(X, n)\widehat{\phantom{-}}.    }
Indeed by construction we get a canonical isomorphism of Mackey functors (see Definition \ref{defn-PST-MF}) and thus by Proposition \ref{prop-alt-MFsp} we have to check that
the specialization maps are the same. But for $C\in (\sC/S)$, $U\subset C$ open and $P\in U$ the specialization map  $s_P: \sC H^{d+n}(X,n)(U)\to \sC H^{d+n}(X,n)(P)$ is defined using the residue homomorphism $\partial_P$ 
of the underlying Rost's cycle module (see \ref{sp-Rostcycle}), which in turn is defined using the differentials in the  cycle complex $C_*(X\times V;\KM_*,n)$, where $V\in Sm$ is a model of $U$, see \cite[7.]{Rost}.
On the other hand the cycle complex $C_*(X\times V;\KM_*,n)$ is nothing but the Gersten resolution on $X\times V$ of $\KMsheaf_{d+n}$, which is isomorphic to the Gersten resolution of $\CH_\Nis^{d+n}(X, n)$ 
as a Zariski sheaf on $X\times V$ (see e.g. \cite[Lemma 3.2]{GL}). It is this isomorphism which induces the isomorphism \eqref{Chow-as-RF1}. 
But the specialization map for $\CH_\Nis^{d+n}(X, n)$ is again induced by the differentials of the Gersten resolution (see \cite[Lem. 6.3]{Akhtar}), thus we get the compatibility with the specialization maps.

\subsection{K\"ahler Differentials}
\subsubsection{Trace for K\"ahler differentials}\label{KDtrace}
Let $T$ be a scheme, $Y$ a noetherian $T$-scheme and $f :X\to Y$ a finite morphism, which is a complete intersection, i.e. $f$ is flat and for any $x\in X$ there exists an open neighborhood $U$ of $x$, such that
$f_{|U}$  factors as the composition of a regular closed embedding followed by a smooth
morphism, $U\inj P\to Y$. (This is for example satisfied if $X$ and $Y$ are regular and $f$ is finite and flat; also any base change with respect to $Y$ will again be a complete intersection.)
Then in \cite[\S 16]{Ku} there is constructed for all $q$ a trace (pushforward)
\[f_*=\Tr_{X/Y}: f_*\Omega^q_{X/T}\to \Omega^q_{Y/T},\]
which has the following properties:
\begin{enumerate}
 \item[(Tr1)] For $\alpha\in \Omega^i_{Y/T}$ and $\beta\in \Omega^{q-i}_{X/T}$, we have $f_*(f^*(\alpha)\beta)=\alpha f_*(\beta)$.
 \item[(Tr2)] For $q=0$ the map $f_*: f_*\sO_X\to \sO_Y$ is the usual trace for a finite and locally free extension of algebras.
\item[(Tr3)] If $g: Y'\to Y$ is a $T$-morphism of noetherian schemes, then the base change $f': X'=X\times_Y Y'\to Y'$ of $f$ is finite and a complete intersection and the following diagram commutes
             \[\xymatrix{ f_*\Omega^q_{X/T}\ar[r]^{{g'}^*}\ar[d]_{f_*} & h_*\Omega^q_{X'/T}\ar[d]_{f'_*}\\
                               \Omega^q_{Y/T}\ar[r]^{g^*} &  g_*\Omega^q_{Y'/T}, }\]
            where $g': X'\to X$ is the base change of $g$ and $h:=f\circ g'= g\circ f'$.
\item[(Tr4)] If $f_i: X_i\to Y$, $i=1,\ldots,n$, are finite morphisms, which are local complete intersections, then so is $f:=\sqcup f_i: X:=\sqcup_{i=1}^n X_i\to Y$ and 
             for $\beta=(\beta_1,\ldots,\beta_n)\in f_*\Omega^q_{X/T}=\oplus_i f_{i*}\Omega^q_{X_i/T}$ the following equality holds in $\Omega^q_{Y/T}$
    \[f_*(\beta)=\sum_i f_{i*}(\beta_i).\]
\item[(Tr5)] If $g: W\to X$ is finite and a complete intersection, then so is $f\circ g: W\to Y$ and we have 
       \[f_*g_*=(f\circ g)_*: (f\circ g)_*\Omega^q_{W/T}\to \Omega^q_{Y/T}.\]
\item[(Tr6)] $f_*\circ d= d\circ f_*: f_*\Omega^q_{X/T}\to \Omega^{q+1}_{Y/T}$.
\item[(Tr7)] For all $a\in f_*\sO_{X}^\times$, we have
              \[f_*(\frac{da}{a})= \frac{d\Nm_{X/Y}(a)}{\Nm_{X/Y}(a)},\]
          where $\Nm_{X/Y}: f_*\sO_X^\times\to \sO_Y^\times$ is the norm map.
\end{enumerate}
Furthermore using \cite[Thm 16.1]{Ku} and \cite[\S 16, Exercise 5)]{Ku} we obtain:
Let $x$ be a $T$-scheme and $\varphi: Y\to x$ a finite morphism, which is a complete intersection. 
Then the following diagram commutes for all $q\geqslant 0$
\eq{eqtracemultiplicities}{\xymatrix{ \Omega^q_{Y/T}\ar[d]_{\varphi_*}\ar[r] & 
          \bigoplus_{y\in Y}\Omega^q_{y/T}\ar[dl]^{\sum_{y} l_y \varphi_{y*}}\\
                \Omega^q_{x/T}, }
}
where $\varphi_y$ is the composition $y\inj Y\xr{\varphi} x$, $l_y={\rm length}(\sO_{Y,y})$ and the horizontal map is induced by the natural surjections $\Omega^q_{Y/T}\to \Omega^q_{y/T}$.

\subsubsection{Residue map for K\"ahler differentials}\label{KDresidue}
Take $C$ in $(\sC/S)$ and set $K=\kappa(C)$, $k=\kappa(x_C)$ and $x=x_C$. 
For $n\in \N$ set 
\[K_n:=\begin{cases} k(K^{p^n}), & \text{if } {\rm char}(F)=p>0,\\ K, &\text{if } {\rm char}(F)=0.\end{cases}\]
Viewing $K_n$ and $\sO_C$ as subsheaves
of the constant sheaf $K$ on $C$ we define $C_n=\Spec (\sO_C\cap K_n)$. Then (see \cite[p. 88 and Thm 3]{HuKu}) we obtain maps
\[C=C_0\to C_1\to C_2\to\ldots,\]
which are homeomorphisms and each $C_n$ is a regular projective and connected curve over $x$ ($=x_{C_n}$). Further, for $P\in C$ denote its image in $C_n$ by $P_n$. Then, by \cite[Thm1, Thm4]{HuKu}, for almost all $n$ the curve
$C_n$ is smooth over $x$ and  $\kappa(P_n)$ is the separable closure of $k$ in $\kappa(P)$.

In \cite[\S 17, Def. 17.4.]{Ku} the residue map  
\[\Res_P: \Omega^1_{\eta/x}\to k\]
is defined by 
\[\Res_P(\alpha):= \Tr_{P_n/x}(\Res_t(\Tr_{C/C_n}\alpha)), \quad \alpha\in \Omega^1_{\eta/x},\]
where  $\Tr_{C/C_n}$ and $\Tr_{P_n/x}$ are the corresponding traces from \ref{KDtrace}, $n$ is chosen such that $\kappa(P_n)$ 
is the separable closure of $k$ in $\kappa(P)$, which implies that the choice of a local parameter $t$  at $P_n$ determines a unique and continuous isomorphism of $k$-algebras between 
$\kappa(P_n)((t))$ and the completion of $K_n$ at $P_n$, and $\Res_t: \Omega^1_{\kappa(P_n)((t))/\kappa(P_n)}\to \kappa(P_n)$ is the usual residue. 
It is shown in \cite[\S 17]{Ku}, that this definition does not depend on the choices made. Furthermore the following properties are proven:
\begin{enumerate}
 \item $\Res_P$ is $k$-linear and factors through 
              $\Omega^1_{\eta/x}\to H^1_P(\Omega^1_{\eta/x})=\Omega^1_{\eta/x}/\Omega^1_{C/x,P}$.
 \item Let $\pi: D\to C$ be a finite morphism in $(\sC/S)$ and  $P\in C$,  then on $\Omega^1_{\eta_D/x_D}$
           \[\Res_{P}\circ \Tr_{D/C} =\sum_{Q\in \pi^{-1}(P)} \Tr_{x_D/x}\circ \Res_Q\]
 \item For $C\in (\sC/S)$ we have for all $\alpha\in \Omega^1_{\eta/x}$
         \[\sum_{P\in C} \Res_P(\alpha)=0. \] 
\end{enumerate}
\begin{remark}\label{rmk-fineresidue}
Notice that if $\kappa(P)$ is separable over $k$, then $\Res_P: \Omega^1_{\eta/x}\to k$, factors via a {\em fine residue map}
\[\Omega^1_{\eta/x}\xr{\widetilde{\Res}_P} \kappa(P)\xr{\Tr} k,\]
where $\widetilde{\Res}_P$ is defined to be $\Res_t$, with $t$ a local parameter at $P$. This is independent of the choice of $t$ by the same argument as in \cite[II, \S 11]{SerreGACC}.
\end{remark}

\subsubsection{Higher residues.}\label{KDRF}
Let $q\geqslant 0$ be a non-negative integer. For any $X$ in $\Regone$ we can consider the $q$-th absolute K\"ahler differentials of $X$ over $\Spec \Z$, 
\[\Omega^q_X:=\Omega^q_{X/\Spec \Z}.\]
For $C\in (\sC/S)$ and $P\in C$  we define 
      \[\Res^{q+1}_P: \Omega^{q+1}_\eta\to \Omega^q_{x_C}\]
  as follows:
       Take $\pi: C\to C'$ a finite, surjective and purely inseparable morphism in $(\sC/S)$ such that $C'$ is generically smooth over $x=x_C=x_{C'}$.
        (Notice that such a $\pi$ is automatically a homeomorphism.)
         Let $P':=\pi(P)$ be the image of $P$ in $C'$ and $\eta'=\pi(\eta)$ the generic point of $C'$. Then define $\Res^{q+1}_P$ as the composition
        \[\Omega^{q+1}_{\eta}\xr{\pi_*}\Omega^{q+1}_{\eta'}\surj \Omega^1_{\eta'/x}\otimes_{\kappa(x)}\Omega^q_x\xr{\Res_{P'}\otimes\id} \Omega^q_x,\]
 where $\Res_{P'}:\Omega^1_{\eta/x}\to \kappa(x)$ is the residue map from \ref{KDresidue} and the surjection is induced by the short exact sequence
        \[0\to \kappa(\eta')\otimes \Omega^1_{\kappa(x)}\to \Omega^1_{\eta'}\to \Omega^1_{\eta'/x}\to 0.\]

This definition is independent of the chosen finite, surjective and purely inseparable $x_C$-morphism $\pi: C\to C'$ with $C'$ generically smooth over $x_C$. Indeed since such a morphism corresponds (up to isomorphism)
to giving a finite purely inseparable field extension $\kappa(C)/K'$ such that $K'$ is separable over $\kappa(x_C)$ we only have to show, that if $C'\to C''$ is a finite, surjective and purely inseparable $x_C$-morphism
between generically smooth curves, then
the residues constructed with respect to $\pi: C\to C'$ and $\pi': C\to C'\to C''$ coincide. 
But since the kernel of the surjection $\Omega^{q+1}_{\eta'}\surj \Omega^1_{\eta'/x}\otimes_{\kappa(x)}\Omega^q_x$ equals the image of
$\kappa(\eta')\otimes_{\kappa(x)}\Omega^{q+1}_x\to \Omega^{q+1}_{\eta'}$ and the same for $\eta''$ the linearity of the trace yields a commutative diagram
\[\xymatrix{\Omega^{q+1}_{\eta'}\ar@{->>}[r]\ar[d]_{\Tr_{\eta'/\eta''}} & \Omega^1_{\eta'/x}\otimes_{\kappa(x)}\Omega^q_x\ar[d]^{\Tr_{\eta'/\eta''}\otimes \id}\\
            \Omega^{q+1}_{\eta''}\ar@{->>}[r] &   \Omega^1_{\eta''/x}\otimes_{\kappa(x)}\Omega^q_x.  }\]
Now it follows from \ref{KDresidue}, (2) and the transitivity of the trace that the definition of $\Res^{q+1}_{P}$ does not depend on the choice of $\pi$.
In particular (see \ref{KDresidue}) we can choose $\pi:C\to C'$ in such a way that $C'$ is smooth over $x$ and that $\kappa(P')/\kappa(x)$ is separable, which simplifies the definition of $\Res_P$.
It also follows that $\Res^1_P=\Res_P$ and that $\Res^{q+1}_P$ satisfies the analog of the properties \ref{KDresidue}, (1), (2), (3).

\begin{thm}\label{thm-KD}
For all $q\geqslant 0$ the absolute K\"ahler differentials define a reciprocity functor 
\[\Regone\Cor\to (\Z-{\rm mod}),\quad X \mapsto \Omega^q_X,\]
such that $\Omega^q([\Gamma_f])=f^*$ is the usual pullback on $\Omega^q$ (resp. $\Omega^q([\Gamma_f^t])=f_*$ is the trace recalled in \ref{KDtrace}) for $f:X\to Y$ a map in $\Regone$ (resp. a finite and flat map in $\Regone$).
Furthermore the local symbol attached to $\Omega^q$ (see Proposition-Definition \ref{prop-defn-localsymbol}) 
is given by
\[(\alpha,f)_P= \Res_P^{q+1}(\tfrac{df}{f} \alpha),\quad \alpha\in \Omega^q_{\eta_C}, f\in \G_m(\eta_C) \]
for all $C\in (\sC/S)$ and $P\in C$.
\end{thm}

\begin{proof}
It follows from Lemma \ref{CarPSTRegone} and (Tr1)-(Tr5) and \eqref{eqtracemultiplicities} in \ref{KDtrace}, that $\Omega^q$ is a presheaf with transfers on $\Regone$ with pullback and pushforward
as in the statement. Further it clearly is a Nisnevich sheaf and satisfies (FP) from Definition \ref{defn-MFsp}.
Next observe that for  $C\in (\sC/S)$ and $P\in C$ the module $\Omega^q_{C,P}$ is free (of maybe infinite rank).
Indeed if $C/x_C$ is smooth we have the exact sequence 
\[0\to \sO_{C,P}\otimes_{\kappa(x_C)}\Omega^1_{x_C}\to \Omega^1_{C,P}\to \Omega^1_{C/x_C,P}\to 0,\]
which implies that $\Omega^1_{C,P}$ and hence also all its exterior powers are free. In case $\kappa(x_C)$ has positive characteristic and $C$ is only regular this is e.g. \cite[Thm. 7.5]{Ku}.
In particular $\Omega^q_{C,P}$ is a submodule of $\Omega^q_\eta$ and hence property (Inj) from Definition \ref{defn-MFsp} holds. Thus $\Omega^q\in\MFsp$ and it remains to check condition (MC).
For this we claim that for  $C\in(\sC/S)$, $P\in C$, $\alpha\in\Omega^q_{C,P}$ and $t$ a local parameter at $P$,  we have (with $x:=x_C$ and the notation from \ref{not-pb-pf-sp})
\eq{eq-thm-KD1}{\Tr_{P/x}(s_P(\alpha))= \Res^{q+1}_P(\tfrac{dt}{t}\alpha).}
Then $\G_m(\eta_C)\to \Omega^q_{x}$, $f\mapsto\Res^{q+1}_P(\tfrac{df}{f}\alpha)$ clearly is a group homomorphism, which is continuous (by \ref{KDresidue}, 1.) and satisfies the conditions (b) and (c) from Proposition \ref{prop-modulus=symbol}. Hence $\Omega^q$ satisfies (MC) and is a reciprocity functor and the explicit description for the local symbol given in the theorem follows from the uniqueness statement 
in Proposition \ref{prop-modulus=symbol}. Thus it suffices
to prove the claim \eqref{eq-thm-KD1}. In case $C/x$ is smooth and  $P/x$ is separable we have that $\sO_{C,P}$ is \'etale over $\kappa(x)[t]$. Thus 
\[ \Omega^q_{C,P}= (\sO_{C,P}\otimes \Omega^q_{\kappa(x)})\oplus (\sO_{C,P}\otimes \Omega^{q-1}_{\kappa(x)}\cdot dt).\]
Therefore we can write $\alpha= \alpha_0+ \alpha_1 dt$ with $\alpha_i\in \sO_{C,P}\otimes \Omega^{q-i}_{\kappa(x)}$ and we have by definition
\[\Res^{q+1}_P(\tfrac{dt}{t}\alpha)=\Tr_{P/x}(s_P(\alpha_0))=\Tr_{P/x}(s_P(\alpha)).\]
The general case thus follows from the analog of \ref{KDresidue}, 1. and 2. for $\Res^{q+1}_P$ and the following lemma.
\end{proof}

\begin{lemma}\label{lem-KD-inseptrace}
Let $\pi:D\to C$ be a finite, surjective and purely inseparable morphism between regular curves over a field $k$ of characteristic $p>0$.
Let $Q\in D$ be a closed point and $P=\pi(Q)$ its image in $C$ and let $z\in \sO_{D,Q}$ and $t\in \sO_{C,P}$ be local parameters at $Q$ and $P$ respectively.
Then for all $\alpha\in \Omega^{q}_{D,Q}$ there exists a $\beta\in \Omega^q_{C,P}$ such that
\[\Tr_{D/C}(\tfrac{dz}{z}\alpha)\equiv \tfrac{dt}{t}\beta \text{ mod } \Omega^{q+1}_{C,P} \quad \text{and} \quad s_P(\beta)= \Tr_{Q/P}(s_Q(\alpha))\quad \text{in } \Omega^{q}_{P}.\]
\end{lemma}
\begin{proof}
Notice first that $\Spec\sO_{D,Q}\to \Spec\sO_{C,P}$ is a complete intersection morphism and a homeomorphism. Thus  $\Tr_{C/D}$ maps $\Omega^{q+1}_{D,Q}$ to $\Omega^{q+1}_{C,P}$ (by (Tr3)).
Hence it suffices to prove the statement in the case $[D:C]=p=e(Q/P)\cdot f(P/Q)$.

{\em 1. case: $e(Q/P)=1$, $f(Q/P)=p$.} We can write $t=zu$ with $u\in \sO_{D,Q}^\times$. Thus 
 \[\Tr_{D/C}(\tfrac{dz}{z}\alpha)=\Tr_{D/C}(\tfrac{dt}{t}\alpha)-\Tr_{D/C}(\tfrac{du}{u}\alpha)\equiv \tfrac{dt}{t}\Tr_{D/C}(\alpha)\quad \text{mod }\Omega^{q+1}_{C,P}.\]
 Further by (S2) for $\Omega^q$ we have
\[s_P(\Tr_{D/C}(\alpha))= e(Q/P) \Tr_{Q/P}(s_Q(\alpha))=\Tr_{Q/P}(s_Q(\alpha)).\]
 Hence we can take $\beta= \Tr_{D/C}(\alpha)\in \Omega^q_{C,P}$.

{\em 2. case: $e(Q/P)=p$, $f(Q/P)=1$.} In this case we have $t=z^p u$ for some $u\in \sO_{C,P}^\times=\sO_{D,Q}^\times\cap \kappa(C)$ (notice that $z^p\in \kappa(C)$). 
Further, since $\pi^*: \Omega^q_{C,P}\to \Omega^q_{D,Q}$ induces an isomorphism $\Omega^q_{P}\to \Omega^q_Q$, there exists a $\beta\in \Omega^q_{C,P}$ and $\gamma_i\in \Omega^{q-i}_{D,Q}$, $i=0,1$,
such that
\[\alpha= \pi^*\beta+ z\gamma_0+ dz\cdot\gamma_1,\]
in particular $s_Q(\alpha)=s_P(\beta)$.
We obtain
\[\Tr_{D/C}(\tfrac{dz}{z}\alpha)\equiv \Tr_{D/C}(\tfrac{dz}{z})\beta\stackrel{\text{(Tr7)}}{\equiv} \tfrac{d(t/u)}{(t/u)}\beta\equiv \tfrac{dt}{t}\beta\quad \text{mod }\Omega^{q+1}_{C,P}.\]
Hence the lemma.
\end{proof}

\begin{example}
We have  $\Omega^0=\G_a.$
\end{example}

\section{First properties of the category of reciprocity functors}

\subsection{Lax Mackey functors with specialization maps}

\begin{definition}\label{defn-quotients-by-subsets}
 Let $\sM: \Reg\Cor^\op\to R-{\rm mod}$ be a presheaf with transfers on $\Reg$. A {\em subset $\sR$ of $\sM$} is a collection of subsets $\sR(X)\subset \sM(X)$ for $X\in \Reg$.
Given such a subset we define $<\sR>(X)$ to be the $R$-submodule of $\sM(X)$ generated by all $\sM(\gamma)(a)$, with $\gamma\in \Cor(X,Y)$, $Y\in \Reg$, and $a\in \sR(Y)$.
Clearly $\Reg\ni X\mapsto <\sR>(X)$ is a subpresheaf with transfers of 
$\sM$. We define the {\em quotient $\sM/\sR$ of $\sM$ by the subset $\sR$}   to be the presheaf with transfers on $\Reg$ given by
\[X\mapsto \frac{\sM(X)}{<\sR>(X)}=:\sM/\sR(X).\]
\end{definition}
It is immediate that $\sM/\sR$ satisfies the following universal property:
Any morphism $\Phi:\sM\to \sN$ of presheaves with transfers on $\Reg$ with $\Phi(\sR(X))=0$ for all $X\in \Reg$ 
factors uniquely through $\sM/\sR$.
In particular, if $\sR\subset\sM$ already is a subpresheaf with transfers on $\Reg$, then $\sM/\sR$ is the usual quotient in the Abelian category of presheaves with transfers on $\Reg$.

We will need the following auxiliary category.

\begin{definition}\label{defn-LMFsp}
A {\em lax Mackey functor with specialization maps} $\sL$ is a presheaf with transfers on $\Reg$ which satisfies the following condition:
\begin{enumerate}
 \item[(W.F.P.)] For all $C\in (\sC/S)$ with generic point $\eta$ the natural map \[\varinjlim_{U\subset C} \sL(U)\surj \sL(\eta)\] is surjective,
                    where the limit is over all non-empty open subsets $U\subset C$.
\end{enumerate}
We denote by $\LMFsp$ the full subcategory of $\PT$ whose objects are lax Mackey functors with specialization maps.
\end{definition}

Notice that the forgetful functor $\MFsp\to \LMFsp$ is fully faithful.

\begin{remark}\label{remark-quotientLMFsp}
Note that the quotient of a lax Mackey functor with specialization maps $\sL$ by a subset, as defined in Definition \ref{defn-quotients-by-subsets}, is again in the category $\LMFsp$.
\end{remark}

\begin{proposition}\label{prop-adjoint-of-LMF}
There exists a functor 
\[\Sigma: \LMFsp\to \MFsp.\]
such that the composition
\[\MFsp\xr{\rm forget}\LMFsp\xr{\Sigma} \MFsp\]
is the identity functor and for any $\sL\in \LMFsp$ we have a functorial map 
\[\sL\to \Sigma(\sL) \quad \text{in } \LMFsp,\]
which $\Sigma$ maps to the identity on $\Sigma(\sL)$ in $\MFsp$ and which is surjective on Nisnevich stalks.
In particular $\Sigma$ is left adjoint to the natural functor $\MFsp\to\LMFsp$, with 
adjunction maps given by $\sL\to\Sigma(\sL)$, for $\sL\in \LMFsp$ and $\sM\xr{\id}\Sigma(\sM)$, for $\sM\in\MFsp$.
\end{proposition}

\begin{proof}
Let $\sL$ be a lax Mackey functor with specialization maps. For $X\in \Reg$ with generic points $\eta_i$ set
\[\sR(X):=\Ker(\sL(X)\to \oplus_i\sL(\eta_i)).\]
Then $\sR$ is a subset of $\sL$ and we set 
\[\sL':=\sL/\sR,\]
in the sense of Definition \ref{defn-quotients-by-subsets}. Clearly we obtain in this way an endofunctor $\LMFsp\to \LMFsp$, $\sL\mapsto \sL'$, which is the identity on the full subcategory
$\MFsp$ and the natural surjection of presheaves $\sL\surj \sL'$ is functorial.
Now for $\sL\in \LMFsp$ we define recursively 
\[\sL^1:=\sL',\quad \sL^n:=(\sL^{n-1})'\quad, n\geqslant  2,\]
and set
\[\sL^\infty=\varinjlim_n \sL^n.\]
We obtain a functor $\LMFsp\to \LMFsp$, $\sL\mapsto \sL^\infty$, which is the identity on $\MFsp$ together with a functorial surjection $\sL\surj \sL^\infty$. 
Further the restriction map $\sL^\infty(U)\to \sL^\infty(\eta)$ is injective for all integral $U\in \Reg$ with generic point $\eta$.
Indeed, if $a\in \sL^\infty(U)$ maps to zero in $\sL^\infty(\eta)$, there exists a representative $a_n\in\sL^n(U)$ of $a$ such that $a_n$ maps to zero in $\sL^n(\eta)$. 
But then $a_n$ maps to zero in $\sL^{n+1}(U)=(\sL^n)'(U)$ by definition and hence $a=0\in \sL^\infty(U)$. Thus $\sL^\infty$ satisfies the conditions (Inj) and (FP) from Definition \ref{defn-MFsp}.
Hence
\[\Sigma(\sL):=(\sL^\infty)_\Nis\]
is a Mackey functor with specialization maps by Lemma \ref{lem-Nis-keeps-transfer} and we obtain 
a functor $\Sigma$ as in the statement.
\end{proof}

\subsection{Truncated reciprocity functors}

\begin{lemma}\label{hnLFAdj}
For all $n\geqslant 1$, the forgetful functor $\Forget_n:\RF_n\ra\LMFsp$ has a left adjoint 
$$\varrho_n:\LMFsp\ra \RF_n$$
such that $\varrho_n\circ\Forget_n=\id$.
\end{lemma}

\begin{proof}
Let $\sM\in\LMFsp$ be a lax Mackey functor with specialization maps. Consider the quotient as presheaves 
with transfers (in the sense of \ref{defn-quotients-by-subsets})
$$
\mathrm{L}\varrho_n(\sM):=\sM/\eusm R_n $$
 of $\sM$ by the subset $\eusm R_n$ consisting of the elements
$$\sum_{P\in U}v_P(f)\cdot\Tr_{P/x_C}(s^{\sM}_P(a)),$$
where $C\in(\sC/S)$ is a curve, $U\subseteq C$ is a non-empty open subset, $a\in \sM(U)$ and $f\in\G_m(\eta_C)$ with $f\equiv 1$ modulo $\sum_{P\in C\setminus U} n [P]$. Then $\mathrm{L}\varrho_n(\sM)$ is a lax Mackey functor with specialization, and we define a Mackey functor with specialization maps
$$\varrho_n(\sM):=\Sigma(\mathrm{L}\varrho_n(\sM)),$$
with $\Sigma$ as in Proposition \ref{prop-adjoint-of-LMF}.
We claim that $\varrho_n(\sM)$ is a reciprocity functor which lies in $\RF_n$. Indeed take $C\in (\sC/S)$ and
$U\subseteq C$ a non-empty open subset and $a\in \varrho_n(\sM)(U)$. We have to show 
$a\in \Fil^n_P\varrho_n(\sM)(\eta_C)$ for all closed points $P\in C$. By Lemma \ref{lem-alt-descr-modulus-space}
this is equivalent to show that $a\in \varrho_n(\sM)(C,\fm_n)$, where $\fm_n=\sum_{P\in C\setminus U} n[P]$.
By definition of $\varrho_n(\sM)$ we find a covering $\coprod_i V_i\to U$ in the Nisnevich topology 
with $V_i$ connected, such that
$a_{|V_i}\in \varrho_n(\sM)(D_i, \fn_{i})$, where $D_i\to C$ is the compactification of $V_i\to U$ and
$\fn_i=\sum_{Q\in D_i\setminus V_i} n [Q]$. By Theorem \ref{thm-modulus-Nis-local} there exists an
effective divisor $\fm\leqslant \fm_n$ such that $a\in \varrho_n(\sM)(C,\fm)\subset \varrho_n(\sM)(C,\fm_n)$.

Clearly, if $\sM\in \RF_n$, then $\eusm R_n=0$ and thus $\varrho_n\circ\Forget_n(\sM)=\Sigma(\Forget_n(\sM))=\sM$, by Proposition \ref{prop-adjoint-of-LMF}.
Finally by the left adjoint property of $\Sigma$ we have 
\begin{equation*}
\begin{split}
\Hom_{\RF_n}(\varrho_n(\sM),\sN)&\xrightarrow{\sim}\Hom_{\LMFsp}(\mathrm{L}\varrho_n(\sM),\Forget_n(\sN))\\
&\xrightarrow{\sim}\Hom_{\LMFsp}(\sM,\Forget_n(\sN)).
\end{split}
\end{equation*}
Hence $\varrho_n$ is left adjoint to $\Forget_n$.
\end{proof}

\begin{example}
Here is an example to show that the truncation functor can be quite brutal: $\varrho_1(\G_a)=0$.
Indeed, if $x=\Spec k$ is an $S$-point and we take $a\in k\setminus\{0\}$, we consider $\A^1=\Spec k[t]\subset \P^1$
and the function $f=t/(t+a)\in k(t)^\times$, which is congruent to 1 modulo $\{\infty\}$.
Then the image of
\[\sum_{P\in \A^1} v_P(f)\Tr_{P/x}s_P(t+a)=a\]
is zero in $\varrho_1(\G_a)$. 

In particular the left adjoint property of $\varrho_1$ implies that any map of reciprocity functors $\G_a\to \sM$ with $\sM\in \RF_1$ is the zero map.
\end{example}

\section{K-groups of reciprocity functors}

\subsection{Tensor products in {\bf PT} }

\subsubsection{}\label{tensorPT} For $\sM,\sN\in \PT$ and $X\in\RegCon$, we set 
$$(\sM\otimes\sN)(X):=\left(\bigoplus_{Y\xr{\textrm{fin. fl.}}X}\sM(Y)\otimes_R \sN(Y)\right)/\eusm R(X)$$
where the sum is taken over all finite flat morphisms $Y\ra X$ in $\RegCon$ and $\eusm R(X)$ is the submodule of the direct sum generated by the elements
$$(a\otimes g_*b')-(g^*a\otimes b'),\qquad (g_*a'\otimes b)-(a'\otimes g^*b)$$
where $g:Y'\ra Y$ is a finite flat morphism over $X$ and $a\in \sM(Y)$, $a'\in\sM(Y')$, $b\in \sN(Y)$, $b'\in\sN(Y')$.
For $X\in \Regone$ we extend the definition additively.

\begin{lemma}\label{LemmaTensTr}
Let $\sM,\sN\in \PT$ be two presheaves with transfers on $\Regone$. Then $\sM\otimes\sN$ is canonically a presheaf with transfers on $\Regone$.
\end{lemma}

\begin{proof}
It is enough to check that $\sM\otimes\sN$ is equipped with pullback maps and pushforward maps on $\RegCon$  satisfying the conditions of Lemma \ref{CarPSTRegone}.
\par\medskip\noindent
 {\itshape{Pushforward:}}  Let $f:X\ra Y$ be a finite flat morphism in $\RegCon$. Then the natural inclusions, for  finite flat morphisms $X'\ra X$ in $\RegCon$,
       $$\sM(X')\otimes_R\sN(X')\ra\bigoplus_{Y'\xrightarrow{\textrm{fin. fl.}} Y}\sM(Y')\otimes_R\sN(Y')$$
        induce a morphism of $R$-modules
        $$f_*:(\sM\otimes\sN)(X)\ra (\sM\otimes\sN)(Y),$$
      which clearly is functorial.
\par\medskip\noindent
 {\itshape{Pullback:}}  Let $g:Y\ra X$ be a morphism in $\RegCon$. Given a finite flat morphism $X'\ra X$ in $\RegCon$, consider the $R$-linear morphism
 \begin{equation}\label{PullbackTens}
  g^*_{X'}:\sM(X')\otimes_R\sN(X')\ra (\sM\otimes\sN)(Y),\quad a\otimes b\mapsto \sum_{T\subseteq Y'}\length(\sO_{Y',\eta_T})\cdot (g_T^*a\otimes g_T^*b),
  \end{equation}
where the sum is taken over the irreducible components of $Y'=Y\times_XX'$ and $g_T:\tilde{T}\ra X'$ is the canonical morphism with $\tilde{T}$ the normalization of $T$.
Note that the definition makes sense since the canonical morphism $f_T:\tilde{T}\ra Y$ is finite and flat. Let us show that these morphisms induce a $R$-linear morphism
 $$g^*:(\sM\otimes\sN)(X)\ra (\sM\otimes\sN)(Y).$$
 Let $f:X''\ra X'$ be a finite flat $X$-morphism between two finite flat connected $X$-schemes in $\Regone$.  Let us use the following notation
$$\xymatrix{
{\tilde{T}'}\ar@/_3em/[ddd]_{g_{T'}}\ar[r]^{f'}\ar[d]_{\nu_{T'}} &{\tilde{T}}\ar[d]^{\nu_T}\ar@{.>}@/^3em/[ddd]^(.3){g_T}  & {}\\
{T'}\ar[r]\ar[d] & {T}\ar[d]& {}\\
{Y''}\ar[r]\ar[d]\ar@{}[rd]|{\square} & {Y'}\ar[r]\ar[d]\ar@{}[rd]|{\square} & {Y}\ar[d]^{g} \\
{X''}\ar[r]^{f} & {X'}\ar[r] & {X} } $$
and 
$$l_{T'}=\length(\sO_{Y'',\eta_{T'}}),\quad l_{T}=\length(\sO_{Y',\eta_T}),\quad l^{T'}_{T}= \length(\sO_{X''\times_{X'}T,\eta_{T'}})$$
where $T$ is an irreducible component of $Y'$ and $T'$ is an irreducible component of $Y''$ that dominates $T$.
By \cite[Lemma A.4.1]{F}, applied to the flat local homomorphism $\sO_{Y',\eta_T}\ra\sO_{Y'',\eta_{T'}}$, we have
\eq{prop-MFtensor1}{l_{T'}=l_T\cdot l^{T'}_{T}.}
Note that since $f$ is universally equidimensional (being finite and flat), any irreducible component $T$ of $Y'$ is dominated by an irreducible component $T'$ of $Y''$.
Now for $a'\in \sM(X'')$ and  $b\in \sN(X')$ we have in $(\sM\otimes\sN)(Y)$
\begin{align}
g^*_{X'}((f_*a')\otimes b) & = \sum_{T\subseteq Y'} l_T \cdot(g_T^*f_*a'\otimes g_T^*b), & \text{by definition}\notag\\
                                     & = \sum_{T\subseteq Y'}\sum_{T'\subseteq X''\times_{X'}T} l_T \cdot l^{T'}_T\cdot (f'_*g_{T'}^*a'\otimes g_T^*b), & \text{by Lemma \ref{lem-basic-corr-formulas}}\notag\\
                                     & = \sum_{T'\subseteq Y''} l_{T'} \cdot (f'_*g_{T'}^*a'\otimes g_T^*b), & \text{by \eqref{prop-MFtensor1}}\notag\\
                                     & = \sum_{T'\subseteq Y''} l_{T'} \cdot (g_{T'}^*a'\otimes g_{T'}^*f^*b), & \text{by definition of }\eusm R(Y)\notag\\
                                     & = g_{X''}^*(a'\otimes f^*b), &\text{by definition.}\notag
\end{align}
By symmetry we obtain that $g_{X'}^*$ maps $\eusm R(X)$ to zero. This shows that the map (\ref{PullbackTens}) is well defined. Functoriality is proven by a similar computation.
\par\medskip\noindent 
{\itshape{Degree formula and cartesian square formula:}} Consider a finite and flat morphism $g:Y\ra X$ in $\RegCon$. Let $f:X'\ra X$ be a finite flat morphism in $\RegCon$, and $a\in \sM(X')$, $b\in \sN(X')$. 
Since $g_T$ is a finite flat morphism (with notation as above), the projection formula gives in $(\sM\otimes \sN)(X)$
\begin{align*}
g_*g^*(a\otimes b)&=g^*_{X'}(a\otimes b)=\sum_{T\subseteq Y'}\length(\sO_{Y',\eta_T})\cdot g^*_Ta\otimes g^*_Tb\\
&=\sum_{T\subseteq Y'}\length(\sO_{Y',\eta_T})\cdot a\otimes g_{T*}g^*_Tb=\left(\sum_{T\subseteq Y'}\length(\sO_{Y',\eta_T})\deg(g_T)\right)\cdot a\otimes b\\
 &= \deg(g)\cdot a\otimes b.
\end{align*}
It remains to check the cartesian square formula. Let $g:Y\ra X$ be a morphism in $\RegCon$ and $f:X'\ra X$ be a finite flat morphism in $\RegCon$ and take $a\in \sM(X')$ and  $b\in \sN(X')$. 
Using the above notation and (\ref{prop-MFtensor1}) we obtain
\begin{align*}
\sum_{T\subseteq Y'} l_T\cdot f_{T*}g_T^*(a\otimes b) &= \sum_{T\subseteq Y'} l_T\cdot\left(\sum_{T'\subseteq X''\times_{X'}T}l^{T'}_{T}\left(g^*_{T'}a\otimes g^*_{T'}b\right)\right) \\
 & = \sum_{T'\subseteq Y''} l_{T'} \cdot(g_{T'}^*a\otimes g_{T'}^*b)=g^*f_*(a\otimes b),
\end{align*}
as desired.
\end{proof}

\begin{remark}\label{defn-MFtensor}
Let $M$ and $N$ be two Mackey functors (see Definition \ref{defn-PST-MF}). Recall that the tensor product as Mackey functor of $M$ and $N$ is defined as follows (see e.g. \cite{Kahn}). Let $x$ be an $S$-point 
$$(M\otimesM N)(x):=\left[\bigoplus_{\xi\xrightarrow{\textrm{fin.}} x}M(\xi)\otimes_RN(\xi)\right]/\eusm R(x)$$
where the sum is taken over all finite morphisms of $S$-points and $\eusm R(x)$ is the submodule of the direct sum generated by the elements of the shape
$$(\delta_*a')\otimes b- a'\otimes \delta^*b \quad \text{and} \quad a\otimes\delta_*b'- \delta^*a\otimes b',$$
where $a\in M(\xi)$, $a'\in M(\xi')$, $b\in N(\xi)$, $b'\in M(\xi')$ and $\delta: \xi'\to \xi$ is a morphism between finite $x$-points. It follows from the definition that the underlying Mackey functor of the tensor product in $\PT$ is simply the tensor product of the underlying Mackey functors as defined above.
\end{remark}

\begin{lemma}\label{LemmaLMFsp}
If $\sM\in\LMFsp$ and $\sN\in\LMFsp$ then $\sM\otimes\sN$ is also a lax Mackey functor with specialization maps. 
\end{lemma}

\begin{proof}
Let $C$ be a curve in $(\sC/S)$. We have to show that the condition (W.F.P.) holds \emph{i.e.} that the map
\begin{equation}\label{MorLMFsp}
\colim_{U\subset C}(\sM\otimes\sN)(U)\ra(\sM\otimes\sN)(\eta_C)
\end{equation}
is surjective. Let $\xi\ra \eta_C$ be a \emph{finite} morphism, and $a\in\sM(\xi)$, $b\in\sN(\xi)$. 
There exists an open subset $V$ in the  normalization $C_\xi$ of $C$ in $\xi\ra \eta_C$ such that 
$a$ lifts to $\sM(V)$ and $b$ to $\sN(V)$. We may then find an open subset $U$ in $C$ such that $V$ contains 
the inverse image $U_\xi$ of $U$ in $C_\xi$. Since the morphism $U_\xi\ra U$ is finite and flat, 
we see that $a\otimes b$ lifts to $(\sM\otimes\sN)(U)$, and therefore (\ref{MorLMFsp}) is surjective.
\end{proof}

\begin{definition}\label{defn-bilinearMF}
Let $\sM_i$, $i=1,\ldots,n$, and $\sN$ be objects in the category $\PT$ \emph{i.e.} presheaves with transfers on $\Regone$. Then an {\em $n$-linear map of presheaves with transfers} 
\[\Phi: \prod_{i=1}^n \sM_i\to \sN\]
 is a collection of $n$-linear maps of $R$-modules $\Phi_X: \prod_{i=1}^n \sM_i(X)\to \sN(X)$, where $X$ runs through all schemes $X\in\RegCon$, satisfying the following properties (we simply write $\Phi$ instead of $\Phi_X$):
\begin{enumerate}
 \item[(L1)] For $f: X\to Y$ a morphism of schemes in $\RegCon$ and $a_i\in\sM_i(Y)$, $i=1,\ldots,n$ we have 
          $$f^*\Phi(a_1,\ldots,a_n)=\Phi(f^*a_1,\ldots, f^*a_n).$$
 \item[(L2)] For $f:X\to Y$ a finite flat morphism of schemes in $\RegCon$ and $a_i\in \sM_i(Y)$, $i\neq i_0$, and $b\in \sM_{i_0}(X)$ we have
      \mlnl{\Phi(a_1,\ldots, a_{i_0-1},f_*b, a_{i_0+1},\ldots, a_n)\\ =f_*\Phi(f^*a_1,\ldots f^*a_{i_0-1},b, f^*a_{i_0+1},\ldots, f^*a_n).}
\end{enumerate}
We denote by $\Lin{n}(\sM_1,\ldots, \sM_n; \sN)$ the $R$-module of $n$-linear maps $\prod_i \sM_i\to\sN$ and set $\Bil(\sM_1,\sM_2;\sN):=\Lin{2}(\sM_1,\sM_2;\sN)$.
\end{definition}

\begin{proposition}\label{prop-tensor-PT}
For $\sM,\sN\in\PT$ the map $\sM(X)\times \sN(X)\to (\sM\otimes\sN)(X)$, $(a,b)\mapsto a\otimes b$, $X\in \RegCon$, is a bilinear map and it is universal, i.e. it induces an isomorphism
of functors on $\PT$
\[\Hom_{\PT}(\sM\otimes\sN,-)\simeq\Bil(\sM,\sN;-).\]
Furthermore, the functor $\otimes$ makes $\PT$ a monoidal category with unit object the constant presheaf $R$.
\end{proposition}
\begin{proof}
Straightforward.
\end{proof}

\begin{example}\label{ex-tensor-constant}
Let $\sM$ be a presheaf with transfers on $\Regone$ and $N$ an $R$-module viewed as a constant 
presheaf with transfers (cf. Example \ref{ex-ModisRF}).
Then it is easy to check that $\sM\otimes N$ is the presheaf with transfers given by $X\mapsto \sM(X)\otimes_R N$, $X\in \RegCon$, and a correspondence $\alpha\in \Cor(X,Y)$
between two schemes in $\RegCon$ acts via $\sM(\alpha)\otimes \id_N$.
\end{example}

\subsection{K-groups of reciprocity functors}

\begin{definition}\label{defn-MLmapRF}
Let $\sM_i$, $i=1,\ldots,n$, and $\sN$ be reciprocity functors. Then an {\em $n$-linear map of reciprocity functors} 
\[\Phi: \prod_{i=1}^n \sM_i\to \sN\]
is a $n$-linear map in $\PT$ (see Definition \ref{defn-bilinearMF}) satisfying the following additional property:
\begin{enumerate}
 \item[(L3)] For any positive integer $r\geqslant  1$, any $C\in(\sC/S)$ and $P\in C$ we have
             \[\Phi(\Fil^r_P\sM_1(\eta_C)\times\cdots\times \Fil^r_P\sM_n(\eta_C))\subset \Fil^r_P\sN(\eta_C),\]
              where $\Fil_P$ is the filtration defined in Definition \ref{defn-filtration}.
 \end{enumerate}
We denote by $\Lin{n}_\RF(\sM_1,\ldots,\sM_n; \sN)$ the $R$-module of $n$-linear maps $\prod_i \sM_i\to \sN$ and set $\Bil_\RF(\sM_1, \sM_2;\sN):=\Lin{2}_\RF(\sM_1,\sM_2;\sN)$.
\end{definition}
Notice that if $\Psi_i: \sM_i'\to \sM_i$, $i=1,\ldots, n$, and $\Psi: \sN\to \sN'$ are morphisms of reciprocity functors 
and $\Phi$ is an $n$-linear map as above, then $\Phi\circ (\prod_i\Psi_i)$  
and $\Psi\circ \Phi$ are also $n$-linear maps. Thus $\Lin{n}_\RF(\sM_1,\ldots, \sM_n; \sN)$ has the expected functorial properties.

\begin{remark}\label{rem-altL4}
It follows from Lemma \ref{lem-alt-descr-modulus-space} that condition (L3) is equivalent to the following:
For all $C\in (\sC/S)$ and effective divisors $\fm_1,\ldots,\fm_n$ on $C$ we have
              \[\Phi(\sM_1(C,\fm_1)\times\ldots\times \sM_n(C,\fm_n))\subset \sN(C,\max\{\fm_1,\ldots,\fm_n\}).\]
\end{remark}

\begin{definition}\label{defn-LMFtensor}
Let $\sM_i$, $i=1,\ldots,n$ be reciprocity functors.
\begin{enumerate}
 \item Let
\[\LT(\sM_1,\ldots, \sM_n):=\sM_1\otimes \cdots \otimes\sM_n/\eusm R,\]
be the quotient (in the sense of \ref{defn-quotients-by-subsets}) of the presheaf with transfers $\sM_1\otimes \cdots \otimes\sM_n\in\PT$ by the subset $\eusm R$, 
consisting of the elements
\eq{defn-LMFtensor1}{\sum_{P\in C\setminus |\max_{i}\{\fm_i\}|} v_P(f)\cdot  s^{\sM_1}_P(a_1)\otimes\cdots\otimes s^{\sM_n}_P(a_n)\quad \text{in } (\sM_1\otimes \cdots \otimes\sM_n)(x_C),}
where $C\in (\sC/S)$,   $\fm_i$, $i=1,\ldots, n$, are effective divisors on $C$, $a_i\in \sM_i(C,\fm_i)$  and $f\in \G_m(\eta_C)$ with $f\equiv 1$ mod $\max_i\{\fm_i\}$.

\item By Remark \ref{remark-quotientLMFsp} and Lemma \ref{LemmaLMFsp}, $\LT(\sM_1,\ldots, \sM_n)$ lies in $\LMFsp$, and the reciprocity $K$-functor associated with $\sM_1,\ldots,\sM_n$ is defined by
$$\T(\sM_1,\ldots, \sM_n):=\Sigma(\LT(\sM_1,\ldots, \sM_n)), $$
where $\Sigma$ is the functor from Proposition \ref{prop-adjoint-of-LMF}.
\end{enumerate}
\end{definition}

\begin{thm}\label{thm-tensorRF}
Let $\sM_i$, $i=1,\ldots, n$, be reciprocity functors. Then the functor $\T(\sM_1,\ldots, \sM_n)$ defined above is a reciprocity functor and the natural $n$-linear morphism of presheaves with transfers (in the sense of Definition \ref{defn-bilinearMF})
\eq{thm-tensorRF0}{\sM_1\times\cdots\times\sM_n\xr{\otimes} \sM_1\otimes\ldots\otimes\sM_n \surj \LT(\sM_1,\ldots,\sM_n)\to\T(\sM_1,\ldots,\sM_n)}
is an $n$-linear morphism of reciprocity functors denoted by
\eq{thm-tensorRF1}{\tau: \sM_1\times\cdots\times\sM_n\to \T(\sM_1,\ldots,\sM_n),\quad (a_1,\ldots, a_n)\mapsto \tau(a_1,\ldots, a_n).}
Furthermore, this gives a functor
\[\prod_{i=1}^n \RF \to \RF,\quad (\sM_1,\ldots,\sM_n)\mapsto \T(\sM_1,\ldots,\sM_n), \]
which represents the functor $\RF\to (R-{\rm mod})$, $\sN\mapsto \Lin{n}_\RF(\sM_1,\ldots,\sM_n; \sN)$ with $\tau$ as the universal $n$-linear map; in particular
\[\Lin{n}_\RF(\sM_1,\ldots,\sM_n; \sN)=\Hom_\RF(T(\sM_1,\ldots,\sM_n), \sN).\]
\end{thm}

\begin{proof}
First we prove that $\T(\sM_1,\ldots,\sM_n)$ is a reciprocity functor, i.e. we have to show that 
for $C\in(\sC/S)$ a curve and $U\subset C$ a non-empty open subset any section $\alpha\in \T(\sM_1,\ldots,\sM_n)(U)$ admits a modulus whose support is equal to $C\setminus U$.
First assume that $\alpha=a_1\otimes \cdots\otimes a_n$ with $a_i\in\sM_i(U_\xi)$, where $\xi\to \eta_C$ is a finite morphism and
we denote by $C_\xi$ the normalization of $C$ in $\kappa(\xi)$ and by $U_\xi$ the pullback of $U$ to $C_\xi$.
Since the $\sM_i$'s are reciprocity functors we find effective divisors $\fm_i$ on $C$ with $|\fm_i|=C\setminus U$, such that $a_i\in \sM_i(C_\xi,\fm_{i,\xi})$,
where $\fm_{i,\xi}$ denotes the pullback of $\fm_i$ to $C_\xi$. Then $|\max_i\{\fm_i\}|=C\setminus U$ and we have to show 
\[\sum_{P\in C\setminus |\max_i\{\fm_i\}|} v_P(f) \Tr_{P/x_C} s_P(a_1\otimes\cdots\otimes a_n)=0\quad \text{in } \T(\sM_1,\ldots ,\sM_n)(x_C),\]
for all $f\in \G_m(\eta_C)$ with $f\equiv 1$ mod $\max_i\{\fm_i\}$.
But by definition  of $\Tr_{P/x_C}$ and $s_P$ (see the construction of the pushforward and pullback in  Lemma \ref{LemmaTensTr}), we have that the sum on the left-hand side equals 
\mlnl{\sum_{P\in C\setminus|\max_i\{\fm_i\}|}v_P(f)\sum_{P'\in C_\xi\times_C P}e(P'/P)(s^{\sM_1}_{P'}(a_1)\otimes\cdots\otimes s^{\sM_n}_{P'}(a_n)) \\
            = \sum_{P'\in C_\xi\setminus|\max_i\{\fm_{i,\xi}\}|}v_{P'}(f)(s^{\sM_1}_{P'}(a_1)\otimes\cdots\otimes s^{\sM_n}_{P'}(a_n)),}
which is zero in $\LT(\sM_1,\ldots,\sM_n)(x_C)$ by the relation \eqref{defn-LMFtensor1} we divide out; a fortiori it is zero in $\T(\sM_1,\ldots,\sM_n)(x_C)$.
Now a general section $\alpha\in \T(\sM_1,\ldots,\sM_n)(U)$ is Nisnevich locally a sum of elements $a_1\otimes\ldots\otimes a_n$ as above and hence it admits a modulus
with support equal to $C\setminus U$ by Theorem \ref{thm-modulus-Nis-local}. Thus $\T(\sM_1,\ldots,\sM_n)$ is a reciprocity functor.

Next we claim that  the $n$-linear morphism in $\PT$ \eqref{thm-tensorRF0} indeed induces a $n$-linear morphism of reciprocity functors $\tau$. We have to check that (L3) holds.
For this take effective divisors $\fm_i$ on $C$ and elements $a_i\in \sM_i(C,\fm_i)$.
As we saw above this gives in $\T(\sM_1,\ldots,\sM_n)(x_C)$
\mlnl{\sum_{Q\in|\max_i\{\fm_i\}|} (\tau(a_1,\ldots, a_n),f)_Q=\\
     -\sum_{Q\in C\setminus|\max_i\{\fm_i\}|} v_Q(f)\Tr_{Q/x_C}s_Q(\tau(a_1,\ldots, a_n))=0}
for all $f\equiv 1$ mod $\max_i\{\fm_i\}$. Hence 
\[\tau(a_1,\ldots, a_n)\in\T(\sM_1,\ldots,\sM_n)(C,\max_i\{\fm_i\}).\]
By Remark \ref{rem-altL4} this gives (L3).

The $n$-linear map \eqref{thm-tensorRF1} induces a natural transformation of functors $\RF\to (R-{\rm mod})$
\[\Hom_{\RF}(\T(\sM_1,\ldots \sM_n), -)\to \Lin{n}_{\RF}(\sM_1,\ldots,\sM_n;-)\]
and we claim, that it is in fact an isomorphism of functors.
For this it suffices to show that it is an isomorphism of $R$-modules when evaluated at any $\sN\in\RF$.
The injectivity follows immediately from the fact that for a regular connected scheme $X\in\Regone$ any element in $\T(\sM_1,\ldots,\sM_n)(X)$ can Nisnevich locally  be written as a sum of elements of the form 
$\Tr_{Y/X}(\tau(a_1,\ldots, a_n))$, where $Y\to X$ is a finite flat morphism in $\RegCon$ and $a_i\in\sM_i(Y)$.
For the surjectivity let $\Phi: \sM_1\times\cdots\times \sM_n\to \sP$ be a $n$-linear morphism between reciprocity functors.
Then the corresponding $n$-linear map on the underlying presheaves with transfers factors over $\sM_1\otimes \cdots\otimes\sM_n$. Furthermore since $\Phi$ satisfies (L3) 
we get
\[\Phi(\eqref{defn-LMFtensor1})=\sum_{Q\in C\setminus|\max_i\{\fm_i\}|} v_Q(f) \Tr_{Q/x_C}s_Q^{\sP}(\Phi(a_1,\ldots,a_n))=0.\]
Thus $\Phi$ factors over $\LT(\sM_1,\ldots,\sM_n)$ and we obtain a morphism of lax Mackey functors with \
specialization maps $\LT(\sM_1,\ldots,\sM_n)\to \sP$. The left adjoint property of $\Sigma$ hence gives 
a morphism in $\RF$
\[\T(\sM_1,\ldots,\sM_n)=\Sigma(\LT(\sM_1,\ldots,\sM_n))\to \sP.\]
This finishes the proof.
\end{proof}

\begin{corollary}\label{cor-RFtensor}
Let $\sM_i$, $\sM_i'$, $i=1,\ldots, n$, be reciprocity functors.
\begin{enumerate}
\item For all $1\leqslant i< j\leqslant n$ we have a functorial isomorphism
      \[\T(\sM_1,\ldots, \sM_i,\ldots, \sM_j,\ldots, \sM_n) \cong\T(\sM_1,\ldots, \sM_j,\ldots, \sM_i,\ldots, \sM_n).\]
\item For all $1\leqslant i\leqslant n$ we have a functorial isomorphism
      \mlnl{\T(\sM_1,\ldots, \sM_i\oplus \sM'_i,\ldots, \sM_n)\\ \cong \T(\sM_1,\ldots, \sM_i\ldots, \sM_n)\oplus \T(\sM_1,\ldots, \sM_i'\ldots, \sM_n).}
\item There is a functorial map
          \[\T(\sM_1,\sM_2,\sM_3)\to \T(\T(\sM_1,\sM_2),\sM_3),\]
           which is surjective as a map of Nisnevich sheaves. 
\end{enumerate}
\end{corollary}

\begin{proof}
(1) and (2) follow immediately from the universal property.  The existence of the map in (3) follows from the universal property and the natural map 
\[\Bil_\RF(\T(\sM_1,\sM_2),\sM_3;\sN)\to \Lin{3}_\RF(\sM_1,\sM_2,\sM_3;\sN).\]
The surjectivity statement follows from the fact that both sides are quotients of $\sM_1\otimes\sM_2\otimes \sM_3$.
\end{proof}

\begin{remark}\label{rem-tensor-equiv}
 We don't know whether the map in (3) above is an isomomorphism (maybe one has to impose further conditions on the $\sM_i$'s). Thus we don't know whether $\T$ is associative,
and we cannot call it a tensor product.
\end{remark}

\begin{corollary}\label{cor-tensor-constant}
Let $\sM_1,\ldots,\sM_n$ be reciprocity functors and let $N$  be an $R$-module viewed as a constant reciprocity functor.
Then we have a functorial isomorphism of reciprocity functors
\[\T(\sM_1,\ldots,\sM_n, N)\cong \Sigma(\T(\sM_1,\ldots, \sM_n)\otimes N),\]
where $\otimes$ on the right hand side is the tensor product in $\PT$ (see Example \ref{ex-tensor-constant}).
In particular, $\T(\sM_1,\ldots,\sM_n, R)=\T(\sM_1,\ldots,\sM_n)$ and for $R$-modules $N_i$ we have 
$\T(N_1,\ldots, N_n)= N_1\otimes_R\ldots \otimes_R N_n$ .
\end{corollary}

\begin{proof}
It suffices to show that (in the notation of Definition \ref{defn-LMFtensor}) we have 
\[\LT(\sM_1,\ldots,\sM_n, N)= \T(\sM_1,\ldots, \sM_n)\otimes N, \]
which is clear since \eqref{defn-LMFtensor1} is already zero on the right hand side.
\end{proof}

\begin{corollary}\label{cor-tensor-filtration}
For  $\sM_1,\ldots, \sM_n\in \RF_n$ (see Definitions \ref{defn-RFn} for the notation), we have $\T(\sM_1,\ldots, \sM_n)\in \RF_n$, i.e.
$\T$ restricts to a functor
\[\T: \prod \RF_n\to \RF_n.\]
\end{corollary}

\begin{proof}
Let $C\in (\sC/S)$ be a curve and $\fm_1,\ldots, \fm_n$ effective divisors on it. Then by (L3) the $n$-linear map $\tau$ induces a map
\[\sM_1(C,\fm_1)\times\cdots\times\sM_n(C,\fm_n)\to \T(\sM_1,\ldots, \sM_n)(C,\max\{\fm_1,\ldots,\fm_n\}).\]
Now the statement follows from Proposition \ref{prop-modulus}, (2) and the fact that any element in $\T(\sM_1,\ldots, \sM_n)(\eta_C)$ can be written as a sum of elements of
the form $\Tr_{\eta_D/\eta_C}(\tau(a_1,\ldots, a_n))$, with $D\to C$ a finite and flat map in $(\sC/S)$ and $a_i\in \sM_i(\eta_D)$.
\end{proof}

%

\begin{remark}\label{rmk-T-is-not-right-exact}
One can show that the category $\RF$ is quasi-Abelian in the sense of \cite[1.1]{Schneiders}. (The kernel of a morphism
in $\RF$ equals the kernel of the corresponding morphism in $\PT$ and the cokernel of a morphism in $\RF$ is obtained by
applying the functor $\Sigma$ from Proposition \ref{prop-adjoint-of-LMF} to the cokernel of the corresponding morphism in $\PT$.)
In particular one has a notion of short exact sequences. But it seems that for given reciprocity functors $\ul{\sM}=\sM_1,\ldots,\sM_{n-1}$
the functor 
\[\T(\ul{\sM},-):\RF\to \RF\]
is {\em not} right exact.
More precisely if $0\to \sN'\xr{\varphi} \sN\xr{\psi} \sN''\to 0$ is a short exact sequence in $\RF$ 
(i.e. $\sN'\cong {\rm Ker}_{\RF}(\psi)$ and $\sN''\cong {\rm Coker}_{\RF}(\varphi)$), then the sequence
\[\T(\ul{\sM},\sN')\to \T(\ul{\sM},\sN)\to \T(\ul{\sM},\sN'')\to 0\]
is in general only exact on the right. It follows from the universal property of $\T$ that 
a sufficient condition for the exactness of the above sequence
is the surjectivity of
\eq{rmk-T-is-not-right-exact1}{{\rm Fil}^r_P\sN(\eta_C)\surj {\rm Fil}^r_P\sN''(\eta_C),}
for all $C\in(\sC/S)$, all $P\in C$ and all $r\ge 1$. This was pointed out to us by \name{B. Kahn}. But condition \eqref{rmk-T-is-not-right-exact1} is in general not satisfied.
As an example you can assume, that the base field
$F$ is algebraically closed and has positive characteristic $p$, then we have a surjection of reciprocity functors
$W_2\surj \G_a$ but for example ${\rm Fil}^p_0 W_2(\eta_{\P^1_F})$ will {\em not} surject onto 
${\rm Fil}^p_0\G_a(\eta_{\P^1_F})$ as the computation of the filtration groups in \cite[Prop. 6.4, (3)]{Kato-Russell} shows.
Note however that condition \eqref{rmk-T-is-not-right-exact1} automatically holds on $\RF_1$.
\end{remark}

\begin{remark}\label{rmk-base-change}
Let $F\subset F'$ be a finite extension of perfect fields, giving rise to a map $S'=\Spec F'\to S=\Spec F$.
In the following we indicate by a label $S$ or $S'$ with respect to which base we are working.
We get a natural functor $\Regone\Cor_{S'}\to \Regone\Cor_S$, which induces a functor
$\RF_S\to \RF_{S'}$. Thus given $\sM_1,\ldots, \sM_n\in \RF_S$ we can view the $n$-linear map
$\tau: \sM_1\times\cdots\times\sM_n\to \T_S(\sM_1,\ldots,\sM_n)$ as a $n$-linear map of reciprocity functors
over $S'$ and hence we get a map
\[\T_{S'}(\sM_1,\ldots,\sM_n)\surj \T_S(\sM_1,\ldots, \sM_n),\quad \tau_{S'}(\ul{m})\mapsto \tau_S(\ul{m})\quad \text{in } \RF_{S'},\]
which is automatically surjective since both sides are quotients of $\sM_1\otimes\ldots \otimes\sM_n$.
\end{remark}

\begin{notation}\label{not-iteratedTensor} 
  Let $\sM$ and $\sN$ be two reciprocity functors. Then we will use the following notation (and variants of it)
       \[\T(\sM,\sN^{\times n}):=\T(\sM,\underbrace{\sN,\ldots,\sN}_{\text{$n$-times}}).\]
\end{notation}

\section{Computations}

\subsection{Relation with homotopy invariant Nisnevich sheaves with transfers}

\subsubsection{Tensor product for presheaves with transfers}\label{PST-tensor}
By \cite[3.2]{VoDM}  the Abelian category $\PST$ (see \ref{PST}) has a monoidal structure extending the one on $\SmCor$ via the Yoneda embedding. By \cite[\S 2]{SuVoBK} it is given by the following formula:
Let $\sF$ and $\sG$ be two presheaves with transfers, then $\sF\otimes_\PST \sG$ is the presheaf with transfers which on $X\in\SmCor$ is given by
\eq{PST-tensor1}{(\sF\otimes_\PST \sG)(X)= \bigoplus_{Y,Z\in\Sm} \sF(Y)\otimes_\Z \sG(Z)\otimes_\Z\Cor(X,Y\times Z)/\Lambda,}
where $\Sm$ is the category of smooth $S$-schemes and $\Lambda$ is the subgroup generated by elements of the following form
\eq{PST-tensor2}{\phi\otimes\psi\otimes (f\times\id_Z)\circ h- f^*(\phi)\otimes\psi\otimes h,}
where $\phi\in \sF(Y)$, $\psi\in \sG(Z)$, $f\in\Cor(Y',Y)$, $h\in\Cor(X, Y'\times Z)$, and
\eq{PST-tensor3}{\phi\otimes\psi\otimes (\id_Y\times g)\circ h-\phi\otimes g^*(\psi)\otimes h,}
where $\phi\in \sF(Y)$, $\psi\in \sG(Z)$, $g\in\Cor(Z',Z)$, $h\in\Cor(X, Y\times Z')$.

\subsubsection{} The categories $\HI$, $\NST$ and $\HINis$ inherit symmetric monoidal structures defined respectively by:
$$F\otimes_\NST G:=(F\otimes_\PST G)_\Nis,\qquad F\otimes_\HI G:=h_0(F\otimes_\PST G),$$
$$ F\otimes_\HINis G:=h^\Nis_0(F\otimes_\PST G),$$ 
where $h_0: \PST\to \HI$ is the left adjoint of the forgetful functor $\HI\to\PST$ (see \ref{PST}) and  $h_0^\Nis$ is the composition of $h_0$ with the sheafification functor.

\begin{proposition}\label{prop-tensor-PST-vs-MF}
Let $\sF,\sG\in\PST$ be two presheaves with transfers.
There exists a canonical and functorial isomorphism in $\PT$ 
$$(\sF\otimes_\PST\sG)\widehat{}\,\cong\, \hat{\sF}\otimes \hat{\sG},$$
where $(-)\widehat{}: \PST\to \PT$ is the functor from Proposition \ref{prop-PSTtoPT} and the tensor product on the right hand side is the one defined in \ref{tensorPT}.
\end{proposition}

\begin{proof}
Let us first construct a morphism 
$$\theta: (\sF\otimes_\PST\sG)\widehat{}\ra \hat{\sF}\otimes \hat{\sG}\quad \text{in }\PT.$$ 
Let $X$ be a scheme in $\Regone$. Given a model $U\in\fM_X$ of $X$, we define a map
\eq{prop-tensor-PST-vs-MF1}{\theta_{X,U}: \bigoplus_{Y,Z\in \Sm}\sF(Y)\otimes_\Z\sG(Z)\otimes_\Z \Cor(U,Y\times Z)\to (\hat{\sF}\otimes  \hat{\sG})(X)}
as follows: For  $\alpha\in\sF(Y)\otimes\sG(Z)$ with $Y,Z\in\Sm$ and $[W]\in\Cor(U,Y\times Z)$ an elementary correspondence, take an open subset $U'\subset U$ which is a model of $X$  and 
such that the normalization $\tilde{W}'$ of the pullback of $W$ along $U'$ is smooth.
(This is possible since the image of $X$ in $U$ is contained in the locus of at most 1-codimensional points.)
Then we define
$$\theta_{X,U}(\alpha\otimes [W]):= \text{class of } (p_{\tilde{W}', Y}^*\otimes p_{\tilde{W}',Z}^*)(\alpha) 
          \text{ in } (\hat{\sF}\otimes  \hat{\sG})(X),$$
where $p_{\tilde{W}', Y}: \tilde{W}'\to Y$ and $p_{\tilde{W}', Z}: \tilde{W}'\to Z$ are induced by projection.
Clearly this element is independent of the choice of $U'$ and we thus obtain a well-defined map \eqref{prop-tensor-PST-vs-MF1}.
We claim that $\theta_{X,U}$ factors to give a map (again denoted by $\theta_{X,U}$)
\[\theta_{X,U}: (\sF\otimes_\PST\sG)(U)\to (\hat{\sF}\otimes \hat{\sG})(X).\]
For this we have to show that $\theta_{X,U}$ sends the elements \eqref{PST-tensor2} and  \eqref{PST-tensor3} to zero. Notice that we can assume that the correspondences $f,g,h$ appearing in 
these formulas are elementary. So take $Y,Y',Z\in\Sm$ and $\phi\in \sF(Y)$, $\psi\in \sG(Z)$ and elementary correspondences $f\in\Cor(Y',Y)$, $h\in\Cor(U, Y'\times Z)$ as in \eqref{PST-tensor2}.
After replacing $U$ with a smaller model of $X$, we can assume that $\tilde{h}$ - the normalization of $h$ - is smooth and further that the normalizations of all irreducible components
of the scheme-theoretic intersection of 
\[I:= (h\times Y\times Z)\cap (U\times (f\times \id_Z))\]
inside $U\times Y'\times Z\times Y\times Z$ are smooth as well as the normalization of their (reduced) images in $U\times Y\times Z$. (We can do this since all this schemes are finite and surjective over $U$.)
For an irreducible component $T$ of $I$ set
\[S_T=p_{UYZ}(T)_{\rm red}\subset U\times Y\times Z\]
and denote by $d_T:=[T:S_T]$ the degree of $T$ over $S_T$ and by $n_T$ the intersection number of $T$ inside $I$.
Thus
\[ [h\times Y\times Z]\cdot [U\times (f\times\id_Z)]=\sum_{T\subseteq I} n_T [T]\]
 and by definition
\[(f\times\id_Z)\circ h=\sum_{T\subseteq I} n_T \cdot d_T\cdot [S_T]\in \Cor(U, Y\times Z).\]
We introduce some more notations using the following commutative diagram in which all maps are the natural once (induced by composition of embeddings, projections and normalization):
\eq{prop-tensor-PST-vs-MF2}{\xymatrix{   & \tilde{S}_T\ar[r]^{p_{S,Y}}\ar@/^4em/[drr]^{p_{S,Z}}                &  Y                                &    \\ 
                                         & \tilde{T}\ar[d]_{p_{T,h}}\ar[r]\ar[u]^{\pi_T}\ar[ur]_{p_{T,Y}}      &  f\times\id_Z\ar[u]\ar[d]\ar[r]   &   Z\\
                                      Y' & \tilde{h}\ar[r]\ar@/_4em/[rru]_{p_{h, Z}}\ar[l]^{p_{h,Y'}}          &  Y'\times Z.\ar[ur]                & 
                                      }  }
Thus we get in $(\hat{\sF}\otimes \hat{\sG})(X)$ (where 'p.f.' refers to the projection formula coming from the definition of $\eusm R$ in \ref{tensorPT})
\begin{align}
\theta_{X,U}(\phi\otimes \psi\otimes (f\times\id_Y)\circ h) &= \sum_{T\subseteq I} n_T\cdot d_T\cdot (p_{S,Y}^*\phi\otimes p_{S,Z}^*\psi), &\text{by def. of } \theta_{X,U},\notag\\
                                                             &=\sum_{T\subseteq I} n_T \cdot (\pi_{T*}\pi_T^*p_{S,Y}^*\phi)\otimes p_{S,Z}^*\psi\notag\\
                                                            &=\sum_{T\subseteq I} n_T \cdot p_{T,Y}^*\phi\otimes p_{T,h}^*p_{h,Z}^*\psi, & \text{by p.f. and \eqref{prop-tensor-PST-vs-MF2}}\notag\\
                                                           &= \sum_{T\subseteq I} n_T \cdot p_{T,h*}p_{T,Y}^*\phi\otimes p_{h,Z}^*\psi, & \text{by p.f. }\notag
 \end{align}
Now observe, that we have a canonical isomorphism
\[I\cong (h\times Y\times Z)\times_{U\times Y'\times Z\times Y\times Z} (U\times f\times \id_Z)\cong h\times_{Y'} f.\]
Furthermore, the map $\tilde{h}\times_{Y'}f\to h\times_{Y'} f$ is an isomorphism over an open and dense subset of $h$ and thus induces finite birational maps $T'\to T$ between the irreducible components; hence their normalizations are equal, $\tilde{T}'=\tilde{T}$.
For $T'$ an irreducible component of $\tilde{h}\times_{Y'} f$ denote by $S_{T'}$ its reduced image (via projection) in $\tilde{h}\times Y$ and by $\tilde{S}_{T'}$ its normalization.
We obtain a commutative diagram
\[\xymatrix{     &     &   Y  \\
             \tilde{T}=\tilde{T}'\ar[r]^{\rho_T}\ar@/^2em/[urr]^{p_{T,Y}}\ar@/_2em/[drr]_{p_{T,h}} & \tilde{S}_{T'}\ar@{^(->}[r]\ar@/^1em/[ur]^{p_{S_{T'},Y}}\ar@/_1em/[dr]_{p_{S_{T'},h}} & \tilde{h}\times Y\ar[u]\ar[d]\\
                &      & \tilde{h}.
           }\]
Further if $T$ is an irreducible component in $I$ with intersection multiplicity $n_T$ and $T'$ is the corresponding irreducible component in $\tilde{h}\times_{Y'} f$, then 
(we denote by $\eta_T$ the generic point of $T$ and by abuse of notation also its image in the various other schemes appearing)
\begin{align}
n_T &= \sum_{i\geqslant 0} (-1)^i \length\left({\rm Tor}_i^{\sO_{Y'\times  Z,\eta_T}}(\sO_{h,\eta_T},\sO_{f\times\id_Z,\eta_T})\right)\notag\\
    &= \sum_{i\geqslant 0} (-1)^i \length\left({\rm Tor}_i^{\sO_{Y',\eta_T}}(\sO_{\tilde{h},\eta_T},\sO_{f,\eta_T})\right)\notag,
\end{align}
which is the intersection number of $T'$ in $\tilde{h}\times_{Y'} f$. Thus
\[f\circ[\Gamma_{p_{h,Y'}}]= \sum_{T\subset I} n_T\cdot\deg\rho_T\cdot [\tilde{S}_{T'}].\]
We get 
\begin{align}
\theta_{X,U}(\phi\otimes \psi\otimes (f\times\id_Y)\circ h) &= \sum_{T\subseteq I} n_T \cdot p_{T,h*}p_{T,Y}^*\phi\otimes p_{h,Z}^*\psi \notag\\
                                                            &= \sum_{T\subseteq I} n_T \cdot \deg \rho_T \cdot p_{S_{T'},h*}p_{S_{T'},Y}^*\phi\otimes p_{h,Z}^*\psi \notag\\
                                                            &=  (f\circ[\Gamma_{p_{h,Y'}}])^*\phi\otimes p_{h,Z}^*\psi \notag\\
                                                            &= p_{h,Y'}^*f^*\phi\otimes p_{h,Z}^*\psi\notag\\
                                                            &= \theta_{X,U}(f^*\phi\otimes\psi\otimes h).\notag
\end{align}
Hence $\theta_{X,U}(\eqref{PST-tensor2})=0$ and similarly one checks $\theta_{X,U}(\eqref{PST-tensor3})=0$. We thus obtain a well-defined map
\eq{prop-tensor-PST-vs-MF5}{\theta_{X}: (\sF\otimes_\PST\sG)\widehat{}\,(X)\to (\hat{\sF}\otimes\hat{\sG})(X).}
It is straightforward to check that the $\theta_X$, $X\in \Regone$, induce a morphism of presheaves with transfers on $\Regone$
\[\theta :(\sF\otimes_\PST \sG)\widehat{}\to (\hat{\sF}\otimes\hat{\sG}),\]
which clearly is functorial in $\sF$ and $\sG$.  Thus it remains to show that \eqref{prop-tensor-PST-vs-MF5} is an isomorphism for all $X$. We will give the inverse map. Let $X$ be in $\RegCon$ and define a map
\[\nu_X: \hat{\sF}(X)\times \hat{\sG}(X)\to (F\otimes_\PST \sG)\widehat{}\,(X)\]
as follows. For $(\hat{a},\hat{b})$ in $\hat{\sF}(X)\times \hat{\sG}(X)$ take a model $U\in\fM_X$ so that there are elements $a\in\sF(U)$ and $b\in \sG(U)$ representing $\hat{a}$ and $\hat{b}$, respectively. Then define
\[\nu_X(\hat{a},\hat{b}):=\text{class of } a\otimes b\otimes [\Gamma_{\delta_U}] \text{ in } (\sF\otimes_\PST\sG)\widehat{}\,(X),\]
where $\delta_U: U\to U\times U$ is the diagonal morphism and $[\Gamma_{\delta_U}]\in\Cor(U,U\times U)$ is its graph.
Clearly $\nu_X$ does not depend on the choice of $U$ or of the representatives $a$ and $b$.
Varying $X$ we claim that the $\nu_X$'s define a bilinear map of presheaves with transfers on $\Regone$ (see Definition \ref{defn-bilinearMF})
\begin{equation}\label{IsoLMFsp}
\nu : \hat{\sF}\times \hat{\sG}\to (\sF\otimes_\PST\sG)\widehat{}.
\end{equation}
Indeed the property (L1) follows immediately from 
\[\delta_U\circ \varphi= (\varphi\times\varphi)\circ \delta_V=(\varphi\times\id_U)\circ(\id_V\times\varphi)\circ\delta_V\]
for each map $\varphi: V\to U$ in $\Sm$. The property (L2) follows from the following equality in $(\sF\otimes_\PST\sG)(U)$ which holds for all 
finite maps $\varphi :V\to U$ in $Sm$ and all $a\in \sF(V)$ and $b\in \sG(U)$ (we denote by $\Gamma_\varphi$ the graph of $\varphi$ and by $\Gamma_\varphi^t$ its transpose):
\begin{align}
(\varphi_*a)\otimes b\otimes [\Gamma_{\delta_U}] &= ([\Gamma_{\varphi}^t]^*a)\otimes b\otimes [\Gamma_{\delta_U}]\notag\\
                                                 &= a\otimes b\otimes (([\Gamma^t_\varphi]\times\id_U)\circ [\Gamma_{\delta_U}]), &\text{by \eqref{PST-tensor2}},\notag\\
                                                 &=a\otimes b\otimes((\id_V\times[\Gamma_\varphi])\circ[\Gamma_{\delta_V}]\circ[\Gamma^t_\varphi])\notag\\
                                                 &= a\otimes \varphi^*b\otimes ([\Gamma_{\delta_V}]\circ [\Gamma^t_\varphi]), &\text{by \eqref{PST-tensor3}.}\notag
\end{align}
Thus $\nu$ is a bilinear map of presheaves with transfers on $\Regone$ and hence induces a morphism presheaves with transfers on $\Regone$  (also denoted by $\nu$)
\[\nu: \hat{\sF}\otimes \hat{\sG}\to (\sF\otimes_\PST\sG)\widehat{}.\]
It is immediate that $\theta\circ\nu=\id$. The equality $\nu\circ\theta=\id$ follows from the following equality in $(\sF\otimes_\PST\sG)(U)$, which holds for all $Y,Z\in\Sm$ and
all $a\in \sF(Y)$, $b\in \sG(Z)$ and all elementary correspondences $[W]\in\Cor(U, X\times Y)$ such that $\tilde{W}$ is smooth (where we denote
by $p_U$, $p_X$ and $p_Y$ the maps induced by the respective projections from $\tilde{W}$):
\begin{align}
 p_Y^*a\otimes p_Z^*b\otimes ([\Gamma_{\delta_{\tilde{W}}}]\circ [\Gamma_{p_U}^t]) &= a\otimes p_Z^*b\otimes (([\Gamma_{p_Y}]\times\id_{\tilde{W}})\circ [\Gamma_{\delta_{\tilde{W}}}]\circ [\Gamma_{p_U}^t] )\notag\\
                                    &= a\otimes b\otimes ((\id_Y\times \Gamma_{p_Z})\circ ([\Gamma_{p_Y}]\times\id_{\tilde{W}})\circ [\Gamma_{\delta_{\tilde{W}}}]\circ [\Gamma_{p_U}^t] )\notag\\
                                    &= a\otimes b\otimes [W].\notag
\end{align}
This finishes the proof.
\end{proof}

\begin{definition}\label{defn-Kgeo} Let $\sM\in\LMFsp$ be a lax Mackey functor with specialization maps and  $x$ 
an $S$-point. One defines the $R$-module 
$\Kgeo(x;\sM)$ as the quotient in the category of $R$-modules of $\sM(x)$ by the submodule generated by the elements 
$$\sum_{P\in U}v_P(f)\cdot\Tr_{P/x}(s^{\sM}_P(a)) $$
where $C\in(\sC/S)$ is a curve of finite type over $x$, $U\subseteq C$ is an open subset, $f\in\G_m(\eta_C)$ with $f\in U^{(1)}_P$, for all $P\in C\setminus U$, and $a\in \sM(U)$.
\end{definition}
Note that this kind of quotient was already considered in \cite[6.1 Definition]{KY}. By construction we have canonical surjections
$$\sM(x)\surj\Kgeo(x;\sM)\surj \mathrm{L}\varrho_1(\sM)(x)\surj\varrho_1(\sM)(x).$$

\subsubsection{}
Let $\sF\in\PST$ be a presheaf with transfers. Applying the functor $\widehat{\phantom{-}}$ from Proposition \ref{prop-PSTtoPT} to the canonical morphism of presheaves with transfers $a:\sF\ra h^\Nis_0(\sF)$ provides a morphism 
$$\hat{a}:\widehat{\sF}\ra\widehat{h^\Nis_0(\sF)}  $$
in $\LMFsp$. By Proposition \ref{HI=RF0}, the right hand side is a reciprocity functor that lies in $\RF_1$, therefore using the adjunction from Lemma \ref{hnLFAdj}, we obtain a morphism of reciprocity functors
\begin{equation}\label{adjFh0}
\hat{a}^\sharp:\varrho_1(\widehat{\sF})\ra\widehat{h^\Nis_0(\sF)}.
\end{equation}
The following proposition may be seen as a generalization of \cite[6.5 Proposition]{KY}.
\begin{proposition}\label{Prophath0}
Let $\sF\in\PST$ be a presheaf with transfers. Then the morphism (\ref{adjFh0}) is an isomorphism of reciprocity functors. Moreover for any $S$-point $x$, the morphisms
$$\Kgeo(x;\hat{\sF})\surj \varrho_1(\hat{\sF})(x)\xrightarrow{\hat{a}^\sharp}\widehat{h^\Nis_0(\sF)}(x), $$
where $\Kgeo(x;\hat{\sF})$ is the $\Z$-module defined in \ref{defn-Kgeo}, are isomorphisms.
\end{proposition}

\begin{proof}
Since $a:\sF\to h^\Nis_0(\sF)$ is a surjection on Nisnevich stalks, so is $\hat{a}$. Further by the construction of $\varrho_1$ in Lemma \ref{hnLFAdj}, 
there is a natural map $\hat{\sF}\to \varrho_1(\hat{\sF})$ whose composition with $\hat{a}^\sharp$ is the map $\hat{a}$. Thus $\hat{a}^\sharp$ is a surjection of Nisnevich sheaves.
By the conditions (Inj) and (FP) which a reciprocity functor satisfies it suffices to check the injectivity on $S$-points. Thus it suffices
to show that the surjective composition from the statement 
\[b:\Kgeo(x;\hat{\sF})\surj \widehat{h^\Nis_0(\sF)}(x)\] 
is injective.
Notice that for an $S$-point $x$ the $\Z$-module $\widehat{h^\Nis_0(\sF)}(x)$
is just the generic stalk of $h^\Nis_0(\sF)$ viewed as a Nisnevich sheaf on some model of $x$; in particular
\begin{equation*}
\widehat{h^\Nis_0(\sF)}(x) = \widehat{h_0(\sF)}(x) =\Coker\left[\hat{\sF}_{\P^1_x}(\A^1_x)\xrightarrow{i_0^*-i_1^*}\hat{\sF}(x)\right].
\end{equation*}
But now consider on $\P^1_x$ the modulus $\fm=\{\infty\}$ and the rational function $f=\tfrac{t}{t-1}\in \kappa(x)(t)^\times$ which is congruent to one modulo $\fm$. Given $\alpha\in\hat{\sF}_{\P^1_x}(\A^1_x)$, we have
$$i^*_0(\alpha)-i^*_1(\alpha)=\sum_{P\in\A^1_x}v_P(f)\Tr_{P/x}(s^{\hat{\sF}}_P(\alpha))$$
in $\hat{\sF}(x)$ and by definition this element vanishes in $\Kgeo(x;\hat{\sF})$. Thus the natural surjection $\hat{\sF}(x)\to \Kgeo(x;\hat{\sF})$ factors via $\widehat{h^\Nis_0(\sF)}(x)$
and gives an inverse of $b$. This proves the statement.
\end{proof}

\begin{lemma}\label{lem-CompTh0}
Let $\sF_1,\ldots,\sF_n\in\HINis$. Then
$$\varrho_1(\hat{\sF}_1\otimes\cdots\otimes\hat{\sF}_n)= \T(\hat{\sF}_1,\cdots,\hat{\sF}_n) \quad\text{in }\RF_1.$$
\end{lemma}

\begin{proof}
By Corollary \ref{cor-tensor-filtration} $\T(\hat{\sF}_1,\ldots,\hat{\sF}_n)\in \RF_1$. Thus applying $\varrho_1$ to the natural map
$\hat{\sF}_1\otimes\ldots\otimes\hat{\sF}_n\to \T(\hat{\sF}_1,\ldots,\hat{\sF}_n)$ gives a canonical map
\eq{lem-CompTh01}{\varrho_1(\hat{\sF}_1\otimes\cdots\otimes\hat{\sF}_n)\to \T(\hat{\sF}_1,\cdots,\hat{\sF}_n).}
On the other hand we have a natural $n$-linear map of presheaves with transfers on $\Regone$
\[\hat{\sF}_1\times\ldots\times \hat{\sF}_n\to \varrho_1(\hat{\sF}_1\otimes\cdots\otimes\hat{\sF}_n),\]
which automatically satisfies (L3) (see Definition \ref{defn-MLmapRF}), since the right hand side is in $\RF_1$.
Thus it is a $n$-linear map of reciprocity functors and hence induces an inverse to \eqref{lem-CompTh01}.
\end{proof}

\begin{thm}\label{CompHIA}
Let $\sF_1,\ldots,\sF_n\in\HINis$ be homotopy invariant Nisnevich sheaves with transfers. There exists a canonical and functorial isomorphism of reciprocity functors
\begin{equation}\label{EquivFond}
\T(\hat{\sF}_1,\ldots,\hat{\sF}_n)\xrightarrow{\sim} (\sF_1\otimesHtr\cdots\otimesHtr \sF_n)\widehat{\phantom{-}}.
\end{equation}
Moreover for any $S$-point x the canonical morphism 
$$\Kgeo(x,\hat{\sF}_1\otimes\ldots\otimes \hat{\sF}_n)\ra\T(\hat{\sF}_1,\ldots,\hat{\sF}_n)(x)$$  is an isomorphism.
\end{thm}

\begin{proof}
Let $\sF:=\sF_1\otimes_\PST\cdots\otimes_\PST\sF_n$ and $a:\sF\ra h_0^\Nis(\sF)$ be the canonical morphism of presheaves with transfers.
By  Proposition \ref{prop-tensor-PST-vs-MF}, we have a canonical and functorial isomorphism in $\LMFsp$
$$\hat{\sF}_1\otimes\cdots\otimes\hat{\sF}_n\xrightarrow{\nu}\hat{\sF}.$$
Thus by Proposition \ref{Prophath0} we get an isomorphism of reciprocity functors
$$\varrho_1\left(\hat{\sF}_1\otimes\cdots\otimes\hat{\sF}_n\right)\xrightarrow{\varrho_1(\nu)} \varrho_1(\hat{\sF})\xrightarrow{\hat{a}^\sharp}\left[h_0^\Nis(\sF)\right]\hat{}. $$
Since  $h_0^\Nis(\sF)=\sF_1\otimes_\HINis\cdots\otimes_\HINis\sF_n$ by definition the first statement follows from Lemma \ref{lem-CompTh0}. 
The second part of the statement is an immediate consequence of  Proposition \ref{Prophath0} and Lemma \ref{lem-CompTh0}.
\end{proof}

Using the main result of \cite{KY} we get:

\begin{corollary}
Let $\sF_1,\ldots, \sF_n$ be homotopy invariant Nisnevich sheaves with transfers on $\Sm_S$ with $S=\Spec F$ the spectrum of a perfect field.
Denote by 
\[\mathrm{K}(F,\sF_1,\ldots, \sF_n)\]
the $K$-group defined in \cite[Def. 5.1]{KY} (in particular if the $\sF_i$'s are semi-Abelian varieties, it coincides
with Somekawa's $K$-group defined in \cite[1.]{Somekawa}). Then there is an isomorphism
\[ \T( \hat{\sF}_1,\ldots,\hat{\sF}_n)(S)\cong\mathrm{K}(F, \sF_1,\ldots, \sF_n).\]
\end{corollary}

\begin{proof}
This follows from Theorem \ref{CompHIA} together with the fact that $\Kgeo(S,\hat{\sF}_1\otimes\ldots\otimes \hat{\sF}_n)$ coincides by its
very definition with the K-group of geometric type $\mathrm{K}'(F, \sF_1,\ldots, \sF_n)$ defined in \cite[Def. 6.1]{KY} and which 
by \cite[Thm 6.2 ]{KY} and \cite[Thm. 11.12]{KY} is isomorphic to $\mathrm{K}(F,\sF_1,\ldots, \sF_n)$.
\end{proof}

\subsection{Applications}

\subsubsection{} We now relate, as in \cite{KY}, the K-groups of some reciprocity functors associated with homotopy invariant Nisnevich sheaves with transfers to Hom groups in \name{V. Voevodsky}'s triangulated category of effective motivic complexes $\DMeffm$. The main result is Theorem \ref{ComputationMC}. Recall that $\DMeffm$ is the full triangulated subcategory of the derived category $\mathrm{D}^{-}(\NST)$ formed by the complexes whose cohomology sheaves are homotopy invariant. Let $X$ be a smooth scheme of finite type over $S$. 
In the sequel we use the following notations
\[M(X):= C_*(L(X)), \qquad M^c(X):=C_*(L^c(X))\]
and
$$h^\Nis_0(X):=\tH^0(C_*(L(X))),\qquad h^{\Nis,c}_0(X):=\tH^0(C_*(L^c(X))),$$ 
where $L(X)$ is the sheaf with transfers represented  by $X$, $L^c(X)$ is the presheaf with transfers such that $L^c(X)(Y)$, for $Y\in\Sm$, is the free Abelian group generated by the closed integral subschemes of $Y\times X$ that are quasi-finite and dominate an 
irreducible component of $Y$ (see \cite[\S 4.1]{VoDM}), and $C_*$ is the Suslin complex \emph{i.e.} the $\A^1$-localization functor.


\subsubsection{} 
Let us denote by $\otimes_\DM$ the tensor product in the category $\DMeffm$ defined by
$$\sC\otimes_\DM\sD:=C_*(\sC\otimes^{\mathrm{L}}\sD) $$
where $-\otimes^{\mathrm{L}}-$ is the derived tensor product on $\mathrm{D}^-(\NST)$ defined using "free resolutions" as in \cite[\S2]{SuVoBK}. 
In particular if $\sF_1,\ldots,\sF_n$ are Nisnevich sheaves with transfers, then unfolding the definitions we get
\begin{equation}\label{IsoMotC}
\tH^0(C_*(\sF_1)\otimes_\DM\ldots \otimes_\DM C_*(\sF_n)) =h^\Nis_0(\sF_1)\otimes_\HINis\ldots\otimes_\HINis h^\Nis_0(\sF_n).
\end{equation}

\begin{thm}\label{ComputationMC}
\begin{enumerate}
\item
Let $X_1,\ldots,X_n$ be smooth schemes of finite type over $S$ and $r\geqslant 0$ be an integer. Set $X:=X_1\times\cdots\times X_n$.  For $*\in\{\emptyset,c\}$, we have isomorphisms of reciprocity functors
\begin{equation*}
 \T(\widehat{h}^{\Nis,*}_0(X_1),\ldots,\widehat{h}^{\Nis,*}_0(X_n),\G_m^{\times r}) \cong \tH^0(M^*(X)(r)[r])\widehat{\phantom{-}}.
\end{equation*}
 \end{enumerate}
\end{thm}
\begin{proof}
Since $\G_m\simeq\Z(1)[1]$, this follows from \cite[Proposition 4.1.7]{VoDM} and 
the isomorphisms \eqref{IsoMotC} and \eqref{EquivFond}.
\end{proof}

\begin{lemma}\label{lem-comp-Chow-HI}
\begin{enumerate}
 \item Assume $X$ is a smooth projective $S$-scheme of pure dimension $d$. Then for all $r\geqslant 0$ we have an isomorphism of reciprocity functors
         \[\tH^0(M(X)(r)[r])\widehat{\phantom{-}}\cong \sC H^{d+r}(X,r),\]
         where $\sC H_0(X,r)$ is the reciprocity functor defined in \ref{Chow-as-RF}.
\item Assume $X$ is a smooth, quasi-projective $S$-scheme of pure dimension $d$ and that $S$ admits resolution of singularities. Then for all $r\geqslant 0$ we have an isomorphism of reciprocity functors
         \[\tH^0(M^c(X)(r)[r])\widehat{\phantom{-}}\cong \sC H^{d+r}(X,r).\]
\end{enumerate}
\end{lemma}

\begin{proof}
 Let $\CH_\Nis^{d+r}(X,r)\in \HINis$ be the homotopy invariant Nisnevich with transfers defined in \ref{Chow-as-HI}.
Then by \eqref{Chow-as-HI1} it suffices to show that $\tH^0(M(X)(r)[r])\cong \CH^{d+r}_\Nis(X,r)$ in the first case and $\tH^0(M^c(X)(r)[r])\cong\CH^{d+r}_\Nis(X,r)$
in the second case. This follows from a classical computation, using duality, the cancellation theorem \cite{MR2804268},  \cite[Corollary B.2]{MR2249535}, \cite[Proof of Thm 4.3.7]{VoDM}
 and the comparison with higher Chow groups  \cite[Cor. 2]{MR1883180}.
\end{proof}

Thus we can rewrite Theorem \ref{ComputationMC} as follows:

\begin{corollary}\label{cor-tensor-chow}
Let $X_1,\ldots,X_n$ be smooth, quasi-projective $S$ schemes of pure dimension $d_1,\ldots,d_n$ and $r\geqslant 0$ an integer.
Assume that either the $X_i$'s are projective or $S$ admits resolution of singularities. Then we have an isomorphism of reciprocity functors
       \[\T(\sC H_0(X_1),\ldots, \sC H_0(X_n),\G_m^{\times r})\cong \sC H^{d+r}(X_1\times\cdots\times X_n, r),\]
       where $d=d_1+\cdots+d_n$.
In particular for all $S$-points $x$ we have
\[\T(\sC H_0(X_1),\ldots, \sC H_0(X_n),\G_m^{\times r})(x)\cong \CH^{d+r}(X_{1,x}\times_x\cdots\times_x X_{n,x}, r).\]
\end{corollary}

\begin{remark}
 \begin{enumerate}
  \item The above corollary should be compared to the results \cite[Thm 2.2]{RS} and \cite[Thm. 6.1]{Akhtar} (see also \cite[12.3 - 12.5]{KY} for the case in which
         $x$ is the spectrum of a perfect field). In particular in the projective case we get for all $S$-points $x$
         \[\T(\sC H_0(X_1),\ldots, \sC H_0(X_n),\G_m^{\times r})(x)= \mathrm{K}(\kappa(x); \sC H_0(X_1),\ldots, \sC H_0(X_n),\G_m^{\times r}),\]
       where the K-group on the right is the Somekawa-type K-group defined by \name{Raskind}-\name{Spie\ss} and \name{Akhtar} in \cite[Def. 2.1.1]{RS} and \cite[3.1]{Akhtar}.
  \item One can proof Corollary \ref{cor-tensor-chow} also in a direct way without using the motivic machinery of
          Voevodsky. Indeed, identifying $\G_m$ with $\sC H^1(S,1)$ 
  one easily checks that the product structure on higher Chow groups induces a multi-linear map of reciprocity functors 
   in the sense of Definition \ref{defn-MLmapRF}
      \[\sC H_0(X_1)\times\cdots\times \sC H_0(X_n)\times\G_m^{\times r}
                                                            \to \sC H^{d+r}(X_1\times\cdots\times X_n, r).\]
     (Notice that the condition (L3) is automatic since we are in $\RF_1$.)
   Hence we get a map from $\T({\rm LHS})$ to the RHS and one can prove by hand that this is an isomorphism 
in a similar way as in \cite{RS} and \cite{Akhtar}.
 \end{enumerate}
\end{remark}

\subsection{Relation with Milnor K-theory}
In the following we denote by $\G_m$ and $\KM_n$ the $\Z$-reciprocity functors over $S$ defined in Example \ref{ex-MilnorK}.

\begin{proposition}\label{prop-Milnor-to-tensor}
Let $\sM$ be a $\Z$-reciprocity functor over $S$ and $x$ an $S$-point. For all $n\geqslant  1$ 
there exists a homomorphism of Abelian groups
\[\sM(x)\otimes_{\Z} \KM_n(x)\to \T(\sM,\G_m^{\times n})(x), \quad m\otimes \{a_1,\ldots,a_n\}\mapsto \tau(m,a_1,\ldots, a_n).\] 
\end{proposition}

\begin{proof}
The proof is similar to the proof of \cite[Prop. 5.9]{MVW}. Write $x=\Spec k$. Clearly there is a natural morphism of Abelian groups
\[\sM(x)\otimes (k^\times)^{\otimes_\Z n}\to \T(\sM,\G_m^{\times n})(x), \quad m\otimes a_1\otimes \ldots\otimes a_n \mapsto \tau(m,a_1,\ldots, a_n).\]
Hence (using Corollary \ref{cor-RFtensor}, (1)) it suffices to show that 
\eq{prop-Milnor-to-tensor1}{\tau(m,a,1-a,\ul{b})=0\quad \text{in } \T(\sM,\G_m^{\times n})(x),}
for all  $a, b_i\in k\setminus\{0,1\}$, $m\in \sM(x)$ and $\ul{b}=(b_3,\ldots, b_n)$.
For this take $c\in \bar{k}$ with $c^3=a$, set $L=k(c)$ and $y=\Spec L$. Further, let $\mu\in \bar{k}$ be a third primitive root of 1, 
set $E=L(\mu)$ and $z=\Spec E$. Denote by $\varphi: z\to x$ the induced map. Notice that $[E:k]$ divides 6. Define a rational function $f$ on $\P_x^1\supset\A_x^1=\Spec k[t]$
\[f:=\frac{t^3-a}{t^3-(1+a)t^2+(1+a)t-a}\in k(t).\]
Then in $E(t)$ we have
\[f=\frac{(t-c)(t-\mu c)(t-\mu^2 c)}{(t-a)(t+\mu)(t+\mu^2)}.\]
Further, let $\pi: \P^1_z\to z$ be the structure map, then by Remark \ref{rem-altL4} 
\[\tau(\pi^*\varphi^*m,t,1-t, \pi^*\varphi^*\ul{b}) \in\T(\sM,\G_m^{\times n})(\P^1_z, \{0\}+ \{1\}+\{\infty\}).\]
Since $f\equiv 1$ mod $\{0\}+\{1\}+\{\infty\}$ reciprocity yields in $\T(\sM,\G_m^{\times n})(z)$
\begin{align}
0 & = \sum_{P\in \P^1_z}(\tau(\pi^*\varphi^*m,t,1-t,\pi^*\varphi^*\ul{b}), f)_P\notag \\ 
  & =\tau(\varphi^*m,c,1-c,\varphi^*\ul{b})+ \tau(\varphi^*m,\mu c, 1-\mu c, \varphi^*\ul{b})+ \tau(\varphi^*m,\mu^2 c, 1-\mu^2 c, \varphi^*\ul{b})\notag\\
  & \phantom{=} -\tau(\varphi^*m,c^3, 1-c^3,\varphi^*\ul{b})- \tau(\varphi^*m,-\mu, 1+\mu,\varphi^*\ul{b})- \tau(\varphi^*m,-\mu^2, 1+\mu^2, \varphi^*\ul{b}).\notag
\end{align}
Multiplying this by 3 we obtain in $\T(\sM,\G_m^{\times n})(z)$
\mlnl{0= \tau(\varphi^*m,c^3, (1-c)(1-\mu c)(1-\mu^2 c), \varphi^*\ul{b})- 3\cdot \tau(\varphi^*m,c^3, 1-c^3, \varphi^*\ul{b})\\ - \tau(\varphi^*m,-1, (1+\mu)(1+\mu^2), \varphi^*\ul{b}). }
Since $(1-c)(1-\mu c)(1-\mu^2 c)=1-c^3$ and $(1+\mu)(1+\mu^2)=1$  we obtain
\[0=2\cdot\tau(\varphi^*m,c^3,1-c^3, \varphi^*\ul{b})=2\cdot\varphi^*\tau(m,a, 1-a, \ul{b}) \quad \text{in } \T(\sM,\G_m^{\times n})(z).\]
Applying $\varphi_*$ we get 
\[12\cdot \tau(m,a, 1-a, \ul{b})=0 \quad \text{in } \T(\sM,\G_m^{\times n})(x).\]
This holds  for all $S$-points $x=\Spec k$ and all $a, b_i\in k\setminus\{0,1\}$, $m\in M(x)$. Thus the vanishing \eqref{prop-Milnor-to-tensor1} follows by exactly the same argument
as in the proof of \cite[Lemma 5.8]{MVW}. (Notice that there is a misprint and the first $1-x$ in the displayed formula on page 32, line 13 should be replaced by $1-y$.)
\end{proof}

\begin{proposition}\label{prop-Milnor-surj-tensor}
The maps from Proposition \ref{prop-Milnor-to-tensor} (with $\sM=\Z$) for $x$ running through all $S$-points yield a morphism of Mackey functors
\[\sigma:\KM_n\to \T(\G_m^{\times n}) \quad \text{in } \MF.\]
In particular $\sigma(x)$ is a surjective homomorphism of $\Z$-modules for all $S$-points $x$.
\end{proposition}

\begin{proof}
The compatibility with pullback is clear. It remains to show that for a finite morphism $\varphi: y=\Spec L\to x=\Spec k$ the pushforward $\varphi_*$ is compatible with $\sigma$.
This is really the same argument as in \cite[Lemma 4.7]{NS}:
Since both sides satisfy (MF1), (MF2) (see Remark \ref{rmk-alt-defn-MF}), the arguments from \cite[I, (5.9)]{BT} reduce us to the case in which every finite extension of $k$ has degree a power of a fixed prime $\ell$
and hence can be written as a successive sequence of extensions of degree $\ell$. Thus by functoriality we can assume $[L:k]=\ell$; hence $\KM_{n+1}(L)=\KM_n(k)\cdot \KM_1(L)$ (by \cite[I, Cor. 5.3]{BT})
and the statement follows from the projection formula on both sides. 
 The surjectivity statement follows immediately, since by definition (see Definition \ref{defn-LMFtensor} and Theorem \ref{thm-tensorRF}) any element
in $\T(\G_m^{\times n})(x)$ is a finite sum of elements of the form $\Tr_{y/x}(\tau(a_1,\ldots, a_n))$ for $y/x$ finite and $a_i\in \G_m(y)$.
\end{proof}

\begin{thm}\label{thm-Milnor=tensor}
For all $n\geqslant  1$ the natural map
\[\Phi:\G_m^{\times n}\to \KM_n,\quad (a_1,\ldots, a_n)\mapsto \{a_1,\ldots, a_n\}\]
is a $n$-linear map of reciprocity functors (in the sense of Definition \ref{defn-MLmapRF}) and the induced morphism
\[\T(\G_m^{\times n}) \xr{\simeq} \KM_n \quad \text{in }\MF\]
is an isomorphism of {\em Mackey functors} whose inverse is given by the map $\sigma$ from Proposition \ref{prop-Milnor-surj-tensor} above.
Moreover if the base field $F$ is {\em infinite}, then the above is an isomorphism of {\em reciprocity functors}.
\end{thm}

\begin{proof}
Clearly $\Phi$ is a $n$-linear morphism of Mackey functors (Definition \ref{defn-bilinearMF}). Since $\KM_n\in \RF_1$ the condition (L3) in Definition \ref{defn-MLmapRF} is automatically satisfied.
Thus we obtain a morphism $\psi: \T(\G_m^{\times n})\to \KM_n$ in $\RF$. Now $\sigma$ is surjective and obviously we have $\psi\circ\sigma=\id$, hence it is an isomorphism in $\MF$.
Notice that by (Inj) $\psi$ is an isomorphism in $\RF$ if it is surjective on Nisnevich stalks. But by \cite[Prop. 4.3]{EMS} (see also \cite{Gabber}) $\KM_n(\sO_{C,P})\to \sK_{n,C,P}^M$
is surjective if $F$ is infinite, where $\sK_{n,C,P}$ is defined in Example \ref{ex-MilnorK}. 
Since this holds for all $C, P$ we also have a surjection $\KM_n(\sO_{C,P}^h)\to (\sK_{n,C,P}^M)^h$.
Thus the Nisnevich stalk at some point $P\in C$ of $\sK_n^M$ is generated by symbols $\{a_1,\ldots,a_n\}$
with $a_i\in \sO(V)^\times$, where $V\to C$ is some Nisnevich neighborhood of $P$. 
Since $\psi(\tau(a_1,\ldots, a_n))=\{a_1,\ldots,a_n\}$ we obtain the second statement. 
\end{proof}

\begin{remark}\label{rmk-KM-Gm-tensor-only-equiv}
\begin{enumerate}
 \item In the same way one can prove 
            \[\KM_n\cong\T(\KM_{n_1},\ldots,\KM_{n_r})\quad \text{for all } n=n_1+\ldots+n_r,\, r\geqslant  1.\]
\item  Combining Theorem \ref{thm-Milnor=tensor} and Corollary \ref{cor-tensor-chow} we get the Nesterenko-Suslin  isomorphism \cite{NS}, i.e. 
         \[\KM_n(x)\cong \CH^n(x,n)\quad \text{for all  $S$-points $x$}.\]
\end{enumerate}
\end{remark}

\subsection{Relation with K\"ahler differentials}
In the following we denote simply by $\G_a$, $\G_m$ and $\Omega^n$ the $\Z$-reciprocity functors over $S$ defined in section 2.
Recall $S=\Spec F$ with $F$ a perfect field.

\begin{lemma}\label{lem-module-structure}
Let $\sM_i$, $i=1,\ldots, n$, be $\Z$-reciprocity functors. Then for any $S$-point $x=\Spec k$ the natural $F$-vector space structure on $k\otimes_\Z \sM_1(x)\otimes_\Z\ldots\otimes_\Z \sM_n(x)$ extends naturally to 
$\T(\G_a, \sM_1,\ldots, \sM_n)(x)$.
\end{lemma}

\begin{proof}
Multiplication with an element $\lambda\in F$ induces a morphism $\lambda:\G_a\to\G_a$ of reciprocity functors over $S$. We obtain a morphism
$\lambda\otimes\id: \T(\G_a, \sM_1,\ldots, \sM_n) \to \T(\G_a,\sM_1,\ldots, \sM_n)$ of reciprocity functors over $S$ and it is straightforward to check
that this induces the looked for module structure $\T(\G_a, \sM_1,\ldots, \sM_n)(x)$.
\end{proof}

\begin{remark}\label{rmk-module-structure-base-change}
It is not clear that in the situation above, the natural $k$-vector space structure on the $\Z$-module $k\otimes_\Z \sM_1(x)\otimes_\Z\ldots\otimes_\Z \sM_n(x)$ extends to the $\Z$-module
$\T(\G_a, \sM_1,\ldots, \sM_n)(x)$.
\end{remark}

\begin{lemma}\label{lem-Kaehler-symbol-calculation}
Let $x=\Spec k$ be an $S$-point and let $0$ be the zero point in the affine line $\A^1=\Spec k[t]\subset \P^1$.
Then for any $f\in k(t)^\times$ and $a,b_1,\ldots, b_{n-1}\in k^\times$ we have
\[(\tau(at,t, b_1,\ldots, b_{n-1}), f)_0=0 \quad \text{in } \T(\G_a, \G_m^{\times n})(x).\]
\end{lemma}

\begin{proof}
We have 
\[at\in \G_a(\P^1,2\cdot \{\infty\}),\quad t\in \G_m(\P^1, \{0\}+\{\infty\})\]
and
\[\ul{b}=(b_1,\ldots, b_{n-1})\in \G_m^{\times n}(\P^1).\]
Hence
$\tau(at,t,\ul{b})\in \T(\G_a,\G_m^{\times n})(\P^1, \{0\}+2\cdot\{\infty\})$ and 
\mlnl{(\tau(at,t,\ul{b}), f)_0=(\tau(at,t,\ul{b}), f_0)_0+(\tau(at,t,\ul{b}), f_0)_\infty  \\
       =  -\sum_{P\in \P^1\setminus\{0,\infty\}} v_P(f_0) \Tr_{P/x}s_P(\tau(at, t,\ul{b})),}
for all $f_0\in k(t)^\times$ with 
\eq{lem-Kaehler-symbol-calculation1}{f/f_0\in U^{(1)}_0 \text{ and } f_0\in U^{(2)}_\infty.}
We distinguish 3 cases. (In the calculation we use Lemma \ref{lem-module-structure} without mentioning it.)

{\em 1. case: $f\in \sO_{\A^1,0}^\times$.} In this case  $f_0:=\frac{t^2-f(0)}{t^2-1}$ satisfies \eqref{lem-Kaehler-symbol-calculation1}.
        Let $\alpha\in \bar{k}$ be a root of $t^2-f(0)$ and $\varphi: y=\Spec k(\alpha)\to x$ the induced map.
        First assume ${\rm char}(k)\neq 2$. Then we have in $\T(\G_a,\G_m^{\times n})(y)$
       \begin{align}
        \varphi^*(\tau(at, t, \ul{b}), f)_0 & = -\tau(a\alpha, \alpha,\ul{b}) - \tau(-a\alpha, -\alpha,\ul{b}) \notag\\
                                              &\phantom{ = } +\tau(a, 1, \ul{b}) + \tau(-a, -1,\ul{b}) \notag\\
                                     & = \tau(-a\alpha, \alpha,\ul{b})+ \tau(-a\alpha,\frac{-1}{\alpha}, \ul{b})\notag\\
                                     & = \tau(-a\alpha, -1,\ul{b})=0.\notag
       \end{align}
       Since $\varphi$ has degree 1 or 2, property (MF2) implies that $(\tau(at, t, \ul{b}), f)_0$ is zero in $\T(\G_a,\G_m^{\times n})(x)$.
       If ${\rm char}(k)=2$ we obtain
        \[(\tau(at, t,\ul{b}), f)_0=-2\cdot\Tr_{y/x}(\tau(a\alpha, \alpha, \ul{b}))+2\cdot \tau(a, 1, \ul{b})=0.\]

{\em 2. case: $f=t$.} In this case $f_0:=\frac{t(t+1)}{t^2+t+1}$ satisfies \eqref{lem-Kaehler-symbol-calculation1}. 
                      Let $\alpha\in \bar{k}$ be a third primitive root of 1 and $y=\Spec k(\alpha)\to x$ the induced map.
                      Assume first ${\rm char}(k)\neq 3$ and $y$ has degree 2 over $x$. Then we have in $\T(\G_a,\G_m^{\times n})(x)$
                     \[(\tau(at,t, \ul{b}),f)_0= - \tau(-a, -1,\ul{b})+\Tr_{y/x}(\tau(a\alpha, \alpha,\ul{b})) = \Tr_{y/x}(\tfrac{1}{3}\tau(a\alpha,1,\ul{b}))=0.\]
                       If the degree of $y$ over $x$ is 1, then 
                     \[(\tau(at,t,\ul{b}),f)_0= - \tau(-a, -1,\ul{b})+ \tau(a\alpha,\alpha,\ul{b}) +\tau(a\alpha^2,\alpha^2,\ul{b})=0.\]
                     If ${\rm char}(k)=3$, then 
                      \[(\tau(at,t,\ul{b}),f)_0=- \tau(-a,-1,\ul{b})+ 2\cdot  \tau(a, 1, \ul{b})=0.\]

{\em 3. case: $f\in k(t)^\times$ arbitrary.} Write $f=t^nu$ with $u\in\sO_{\A^1,0}^\times$ and conclude with the first two cases.
\end{proof}

\begin{proposition}\label{prop-Kaehler-map-to-RFtensor}
Let $x=\Spec k$ be an $S$ -point. For all $n\geqslant  1$ there exists a homomorphism of $F$-vector spaces (see Lemma \ref{lem-module-structure}) 
\[\theta(x): \Omega^n_x\to \T(\G_a,\G_m^{\times n})(x), \quad a\frac{db_1}{b_1}\cdots \frac{db_n}{b_n}\mapsto \tau(a,b_1,\ldots, b_n).\]
\end{proposition}

\begin{proof}
 By \cite[Lemma 4.1]{BE03} the kernel of the surjective map
\[k\otimes_\Z (k^\times)^{\otimes n}\to \Omega^n_x, \quad a\otimes(b_1\otimes\ldots\otimes b_n)\mapsto a\frac{db_1}{b_1}\cdots \frac{db_n}{b_n} \]
is the subgroup of $k\otimes_\Z (k^\times)^{\otimes n}$ generated by elements of the following form
\eq{prop-Kaehler-map-to-RFtensor1}{a\otimes b_1\otimes\ldots\otimes b_n,\quad \text{with }b_i=b_j \text{ for some } 1\leqslant i<j\leqslant n, }
with $a\in k$, $b_i\in k^\times$, and
\eq{prop-Kaehler-map-to-RFtensor2}{\lambda a\otimes a\otimes b_1\otimes\ldots\otimes b_{n-1} +
      \lambda (1-a)\otimes(1-a) \otimes b_1\otimes\ldots\otimes b_{n-1},}
with $\lambda, b_i\in k^\times$ and $a\in k\setminus\{0,1\}$.
Thus it suffices to check, that the natural map
\[k\otimes_\Z (k^\times)^{\otimes n}\to \T(\G_a,\G_m^{\times n})(x), \quad a\otimes b_1\otimes\ldots\otimes b_n\mapsto \tau(a, b_1,\ldots, b_n) \]
maps the elements \eqref{prop-Kaehler-map-to-RFtensor1} and \eqref{prop-Kaehler-map-to-RFtensor2} to 0. 
By Proposition \ref{prop-Milnor-to-tensor} (with $\sM=\G_a$) this map factors over $k\otimes_\Z\KM_n(k)$, which shows (using also Lemma \ref{lem-module-structure}) that \eqref{prop-Kaehler-map-to-RFtensor1}
is mapped to zero.
Now take $\lambda,a, b_i\in k^\times$ as in \eqref{prop-Kaehler-map-to-RFtensor2} and write $0$ for the zero point in the affine line  $\A^1=\Spec k[t]\subset \P^1$.
We get
\eq{prop-Kaehler-map-to-RFtensor3}{\tau(\lambda t, t,\ul{b})\in \T(\G_a,\G_m^{\times n})(\P^1, 0+2\cdot \infty).}
Set 
\[f:=\frac{(t-a)(t-(1-a)) }{t(t-1) }.\]
We have $f\equiv 1$ mod $2\cdot \infty$. Thus $(\tau(\lambda t,t,\ul{b}),f)_\infty=0$.
Further $(\tau(\lambda t,t,\ul{b}),f)_0=0$ by Lemma \ref{lem-Kaehler-symbol-calculation}. Hence
\begin{align}
0 & = \sum_{P\in \P^1}(\tau(\lambda t,t,\ul{b}),f)_P
      = \sum_{P\in \P^1\setminus\{0,\infty\}}(\tau(\lambda t,t,\ul{b}),f)_P\notag\\
  & =\tau(\lambda a, a,\ul{b})+ \tau(\lambda (1-a), (1-a),\ul{b}).\notag
\end{align}
This yields the statement.
\end{proof}

\begin{corollary}\label{cor-Kaehler-map-to-RFtensor-on-RS}
Let $X$ be in $\RegCon$ with generic point $\eta$. Then the morphism $\theta(\eta)$ from above induces a map
\[\theta(X): \Omega^n(X)\to \T(\G_a,\G_m^{\times n})(X).\]
\end{corollary}
\begin{proof}
Since any reciprocity functor $\sM$ is in particular a Zariski sheaf on $X$ and we have inclusions $\sM(X)\subset \sM(\eta)$, it suffices to check
that $\theta(\eta)$ sends $\Omega^n_{X,P}$ to $\T(\G_a,\G_m^{\times n})_{X,P}$ for all closed points $P\in X$.  But this follows immediately from the fact that any element in 
$\Omega^n_{X,P}$ is a sum of elements of the form $a \frac{db_1}{b_1}\cdots \frac{db_n}{b_n}$ with $a\in \sO_{X,P}$ and $b_i\in \sO_{X,P}^\times$.
\end{proof}

\begin{proposition}\label{prop-Kaehler-map-to-RFtensor-is-MFhom}
The maps $\theta(X)$ from Corollary \ref{cor-Kaehler-map-to-RFtensor-on-RS} for $X$ running through $\RegCon$ yield a morphism of reciprocity functors
\[\theta:\Omega^n\to \T(\G_a,\G_m^{\times n}) \quad \text{in } \RF.\]
In particular $\theta$ is a surjective homomorphism of Nisnevich sheaves.
\end{proposition}

\begin{proof}
The compatibility of $\theta$ with pullbacks is immediate. It suffices to check the compatibility with pushforward on $S$-points and 
to consider the two cases in which $y=\Spec L\to x=\Spec k$ is either separable or purely inseparable of degree $p$.
In the separable case, we can write any element in $\Omega^n_L=L\otimes\Omega^n_k$ as a finite sum of elements of the form
\[a\frac{db_1}{b_1}\cdots \frac{db_n}{b_n},\quad \text{with } a\in L \text{ and } b_i\in k.\]
Thus the compatibility with the trace follows since both sides satisfy the projection formula.
In the purely inseparable case we write $L= k[c]$ with $c^p=a\in k$, where $p={\rm char}(k)>0$. We can write any element in $\Omega^n_L$ as a sum of elements of the following form
\eq{prop-Kaehler-map-to-RFtensor-is-MFhom1}{\lambda c^j\frac{db_1}{b_1}\cdots \frac{db_n}{b_n},\quad \text{with } 0\leqslant j\leqslant p-1, \,\lambda, b_i\in k}
and 
\eq{prop-Kaehler-map-to-RFtensor-is-MFhom2}{\lambda c^j\frac{dc}{c}\frac{db_1}{b_1}\cdots \frac{db_{n-1}}{b_{n-1}},\quad \text{with } 1\leqslant j\leqslant p, \,\lambda, b_i\in k.}
The compatibility with the trace for the elements \eqref{prop-Kaehler-map-to-RFtensor-is-MFhom1} again follows from the projection formula (in fact the trace is zero in this case).
The compatibility for the elements \eqref{prop-Kaehler-map-to-RFtensor-is-MFhom2} in the case $j=p$ follows from the projection formula on both sides together with (Tr7) in \ref{KDtrace} (since $c^p\in k$).
Further, by \cite[Ex. 16.6]{Ku} we have 
\[\Tr_{L/k}(c^j\frac{dc}{c}\frac{b_1}{b_1}\cdots \frac{db_{n-1}}{b_{n-1}})=0,\quad \text{for all }1\leqslant j\leqslant p-1.\]
Thus it remains to show that $\Tr_{L/k}(\tau(\lambda c^j,c, \ul{b}))=0$ in $\T(\G_a,\G_m^{\times n})(x)$ for
 $1\leqslant j\leqslant p-1$. 
For this we view $y=\Spec k[t]/(t^p-a)\subset \A^1=\Spec k[t]\subset \P^1$ as a closed point.
Then  for $1\leqslant j\leqslant p-1$ reciprocity yields
\begin{align}
\Tr_{L/k}(\tau(\lambda c^j,c, \ul{b})) & = -(\tau(\lambda t^j, t,\ul{b}), t^p-a)_0  -(\tau(\lambda t^j, t,\ul{b}), t^p-a)_\infty\\
                                       & = -\tfrac{1}{j}(\tau(\lambda t^j, t^j,\ul{b}), t^p-a)_0-\tfrac{1}{j}(\tau(\lambda t^j, t^j,\ul{b}), t^p-a)_\infty.
\end{align}
Now let $\pi: \P^1\to\P^1$ be the $x$-morphism induced by $k[t]\to k[t]$, $t\mapsto t^j$. 
Then by Proposition \ref{prop-pfpb-localsymbol}, 2. and Lemma \ref{lem-Kaehler-symbol-calculation} we have 
\[(\tau(\lambda t^j, t^j,\ul{b}), t^p-a)_0=(\tau(\lambda t, t,\ul{b}), \pi_*(t^p-a))_0=0.\]
For the same reason also
\[(\tau(\lambda t^j, t^j,\ul{b}), t^p)_0=(\tau(\lambda t, t, \ul{b}), \pi_*(t^p))_0=0.\]
Hence by reciprocity also
\[(\tau(\lambda t^j, t^j, \ul{b}), t^p)_\infty=0.\]
We obtain
\[\Tr_{L/k}(\tau(\lambda c^j, c, \ul{b}))=-\tfrac{1}{j}(\tau(\lambda t^j, t^j,\ul{b}), \frac{t^p-a}{t^p})_\infty\]
which is zero since $\tau(\lambda t^j, t^j,\ul{b})\in \T(\G_a,\G_m^{\times n})(\P^1, 0+(j+1)\cdot\{\infty\})$ and $\frac{t^p-a}{t^p}\in U^{(p)}_\infty$.

Hence $\theta$ is a morphism of reciprocity functors. Finally, for $X\in \RegCon$ and $P\in X$ closed any element in $\T(\G_a,\G_m^{\times n})_{X,P}^h$ (the Nisnevich stalk on $X$ in $P$) is a finite sum 
of elements of the form  $\Tr_{Y/X}(\tau(a,b_1,\ldots, b_n))$, with $\pi: Y\to X$ in $\RegCon_*$,  $a\in\pi_*(\sO_Y)_P^h $ and $b_i\in \pi_*(\sO_Y^\times)_P^h$ 
by definition (see Definition \ref{defn-LMFtensor} and Theorem \ref{thm-tensorRF}). Thus the surjectivity statement follows immediately.
\end{proof}

\begin{thm}\label{thm-tensorRF=KD}
Assume $F$ has characteristic zero. Then the map 
\[\theta: \Omega^n\xr{\simeq} \T(\G_a,\G_m^{\times n})\]
from Proposition \ref{prop-Kaehler-map-to-RFtensor-is-MFhom} is an isomorphism of reciprocity functors.
\end{thm}

Before we prove the theorem we need the following Lemma.
 
\begin{lemma}\label{lem-GaFil}
Assume $F$ has characteristic zero. Then for all $C\in (\sC/S)$, $P\in C$ and $r\geqslant 1$
\[\Fil^r_P\G_a(\eta_C)=\{a\in \kappa(\eta_C) \,|\, v_P(a)\geqslant -r+1 \}.\]
In particular
\[\Fil^0_P \G_a(\eta_C)=\Fil^1_P\G_a(\eta_C)=\sO_{C,P}.\]
\end{lemma}

\begin{proof}
Recall that the local symbol of $\G_a$ at a point $P\in C$ is given by (see Theorem \ref{thm-KD})
\[(a,f)_P=\Res_P(a\frac{df}{f}).\]
The inclusion $\supset$ is thus straightforward to check. For the other inclusion take $a\in \Fil^r_P \G_a(\eta_C)$ and assume we can write $a= u/t^s$ with $s\geqslant r$, $u\in \sO_{C,P}^\times$
and $t\in \sO_{C,P}$ a local parameter. Then for $b\in \sO_{C,P}$ we have by definition of $\Fil$
\mlnl{0=(a, 1+bt^s)_P=\Res_P\left(\frac{u}{t^s} \frac{d(1+b t^s)}{1+b t^s}\right)\\ 
       = \Res_P\left(\frac{sub}{1+b t^s} \frac{dt}{t}\right)+\Res_P\left(\frac{ub}{1+b t^s} db\right)=s\Tr_{P/x_c}(u(P)b(P)).}
Since $\Tr: \kappa(P)\times \kappa(P)\to \kappa(x_C)$, $(\lambda,\mu)\mapsto \Tr(\lambda\mu)$ is non-degenerate, we get $u(P)=0$. A contradiction.
\end{proof}

\begin{proof}[Proof of Theorem \ref{thm-tensorRF=KD}.]
For $X\in \RegCon$ define 
\[\Phi_X: \G_a(X)\times (\G_m(X))^{\times n}\to \Omega^n(X),\quad (a, b_1,\ldots, b_n)\mapsto a\frac{db_1}{b_1}\cdots \frac{db_n}{b_n}.\]
Clearly the collection $\Phi=\{\Phi_X\}_{X\in \RegCon}$ satisfies condition (L1) of Definition \ref{defn-bilinearMF} and by (Tr1), (Tr2) and (Tr7) in \ref{KDtrace} also condition (L2).
Now let $C\in (\sC/S)$ be a curve, $P\in C$ a point and $r\geqslant 1$ a positive integer. 
Take $a\in \Fil^{r}_P\G_a(\eta_C)$ and $b_i\in \Fil^r_P\G_m(\eta_C)=\G_m(\eta_C)$, $i=1,\ldots, n$.
Then we can write $a= \frac{a_0}{t^{r-1}}$ with $a_0\in \sO_{C,P}$ (by Lemma \ref{lem-GaFil}) and 
$b_i=t^{s_i}u_i$ with $s_i\in\Z$ and $u_i\in\sO_{C,P}^\times$. 
For $u= 1+ t^{r}c\in U^{(r)}_P$ we get
\mlnl{\Res_P(\Phi(a, b_1\ldots, b_n)\frac{du}{u})=
 \Res_P\left(\frac{a_0}{t^{r-1}}\frac{du_1}{u_1}\cdots \frac{du_n}{u_n} \frac{t^{r-1}(rcdt+tdc)}{1+t^r c}\right) \\
     +\sum_{i=1}^n s_i\Res_P\left(\frac{a_0}{t^{r-1}}\frac{du_1}{u_1}\cdots
      \underbrace{\frac{dt}{t}}_{i\text{-th place}}\cdots \frac{du_n}{u_n} \frac{t^{r-1}(rcdt+tdc)}{1+t^{r}c}\right), }
which clearly is zero. Hence $\Phi(a, b_1\ldots, b_n)\in \Fil^r_P\Omega^n(\eta_C)$ and thus satisfies (L3).
Therefore $\Phi$ is an (n+1)-linear map of reciprocity functors and the universal property 
of $\T(\G_a,\G_m^{\times n})$ yields a map of reciprocity functors (also denoted by $\Phi$)
\[\Phi: \T(\G_a,\G_m^{\times n})\to \Omega^n.\]
By the very definition of $\Phi$ and $\theta$ we have $\Phi\circ \theta=\id_{\Omega^n}$ 
(it suffices to check this on $S$-points). In particular $\theta$ is injective. 
Since it is also surjective by Proposition \ref{prop-Kaehler-map-to-RFtensor-is-MFhom}, the theorem follows.
\end{proof}

\begin{remark}
 The above theorem should be compared to \cite[Thm. 3.6]{Hiranouchi}. 
        In particular 
         \[\T(\G_a,\G_m^{\times n})(S)= \mathrm{K}(F; \G_a,\G_m^{\times n}),\]
       where the K-group on the right is the Somekawa-type K-group defined by \name{Hiranouchi} in \cite[Def. 3.3]{Hiranouchi}.
\end{remark}

\subsubsection{}\label{higher-Bs}
Assume $F$ has characteristic $p>0$. Notice that $d: \Omega^{n-1}\to \Omega^n$ is a map of $\Z$-reciprocity functors 
and thus 
\[X\mapsto B_1\Omega^n(X):= (d\Omega^{n-1}_X)(X)\]
defines a Nisnevich sheaf with transfers on $\Regone$, $B_1\Omega^n\in\NT$. We denote by $\Omega^n/B_1\Omega^n\in \NT$ the quotient in $\NT$.
Furthermore for any $X\in \Regone$ we have the inverse Cartier operator 
\[C^{-1}: \Omega^n_X\to \Omega^n_X/B_1\Omega^n_X,\]
which on stalks is given by
\[C^{-1}: \Omega^n_{X,P}\to \Omega^n_{X,P}/B_1\Omega^n_{X,P},\quad a\frac{db_1}{b_1}\cdots\frac{db_n}{b_n} \mapsto a^p\frac{db_1}{b_1}\cdots\frac{db_n}{b_n}.\]
The Cartier operator is clearly compatible with pullbacks and it is well-known that it is also compatible with the trace (see e.g. \cite[Thm. 2.6, (i), (v)]{Kay} and use that the Frobenius on the de Rham-Witt complex lifts
the inverse Cartier operator). Thus we get a morphism of Nisnevich sheaves with transfers on $\Regone$
\[C^{-1}: \Omega^n\to \Omega^n/B_1\Omega^n \quad \text{in }\NT.\]
Recursively we define $B_r\Omega^n\in\NT$ for $r\geqslant 2$ as the Nisnevich sheaf with transfers which as a Zariski sheaf on $X\in \Regone$ is defined 
as the preimage of $C^{-1}(B_{r-1}\Omega^n_X)$ under the natural surjection $\Omega^n_X\to \Omega^n_X/B_1$ (see e.g. \cite[(2.2)]{IlDRW}). We obtain a chain in $\NT$
\[0:=B_0\subset B_1\subset \ldots\subset B_r\subset\ldots\subset \Omega^n\]
and we set 
\[B_\infty\Omega^n:=\bigcup_r B_r\Omega^n\subset \Omega^n\quad \text{in } \NT.\]
The the inverse Cartier operator induces thus an endomorphism in $\NT$
\[C^{-1}: \Omega^n/B_\infty\to \Omega^n/B_\infty.\]

\begin{lemma}\label{cor-B-RF}
The Nisnevich sheaves with transfers on $\Regone$, $B_r\Omega^n$ and $\Omega^n/B_r$, $r\in [1,\infty]$,
defined above are reciprocity functors. 
\end{lemma}

\begin{proof}
Since the $B_r(\Omega^n)$ are sub-Nisnevich-sheaves with transfers of $\Omega^n$, they are clearly reciprocity functors.
If we know that $\Omega^n/B_r$, $r\in [1,\infty]$, satisfies the property (Inj) from Definition \ref{defn-MFsp}, then 
it is also a Mackey functor with specialization maps.
In this case it is also a reciprocity functor. Indeed a section $a\in(\Omega^n/B_r)(U)$, 
$U\subset C$, $C\in (\sC/S)$, can be lifted Nisnevich (or even Zariski) locally to a section in $\Omega^n$.
Since the latter is a reciprocity functor it follows that $a$ satisfies the assumptions of 
Theorem \ref{thm-modulus-Nis-local} (with $\sM=\Omega^n/B_r$) and thus $a$ has a modulus. 

Thus it suffices to prove (Inj), which follows from the following claim:

{\em Let $A$ be a discrete valuation ring, which is essentially of finite type over $F$ and denote by $K$ its fraction field.
Then for all  $n\ge 0$ and $r\in [1,\infty]$ we have
\[ \Omega^n_A\cap B_r(\Omega^n_K)= B_r(\Omega^n_A),\]
where the intersection is taken inside $\Omega^n_K$.}

Indeed, since $A$ is essentially smooth over $F$, the $A$-module $\Omega^n_{A/k}=\Omega^n_A$ is
free of finite rank and a submodule of $\Omega^n_K$. Furthermore the Frobenius iterate $F^r: A\to A$ is a finite and
flat (hence free) homomorphism. In particular $F^r_*\Omega^n_A$ is a free $A$-module as well as its sub-$A$-module
$B_r(\Omega^n_A)$ (see \cite[0, Prop. 2.2.8, (a)]{IlDRW}). But it is straightforward to check by induction on $r$ that we
have $B_r(\Omega^n_K)= B_r(\Omega^n_A)\otimes_A K$. This implies the claim and finishes the prove of the lemma.
\end{proof}

\begin{corollary}\label{cor-Kaehler-posChar}
Assume $F$ has characteristic $p>0$. Then the surjective morphism of reciprocity functors 
$\theta:\Omega^n \to \T(\G_a,\G_m^{\times n})$ from Proposition \ref{prop-Kaehler-map-to-RFtensor} factors via
\[\Omega^n/B_\infty\surj \T(\G_a,\G_m^{\times n})\quad \text{in } \RF\]
and the following diagram commutes (in $\RF$)
\[\xymatrix{\Omega^n/B_\infty\ar[r]^{C^{-1}}\ar@{->>}[d] & \Omega^n/B_\infty\ar@{->>}[d]\\
            \T(\G_a,\G_m^{\times n})\ar[r]^{F\otimes\id} & \T(\G_a,\G_m^{\times n}),  }\]
where $F:\G_a\to \G_a$ is the absolute Frobenius.
\end{corollary}
\begin{proof}
The commutativity of the diagram is immediate once we know the vertical maps exist. Thus we have to show that for any $X\in \RegCon$ and any $r\geqslant 1$ the map 
$\theta$ sends $B_r(X)$ to zero. Since $\T(\G_a,\G_m^{\times n})$ is a reciprocity functor it suffices to prove this for all $S$-points $x$.
First we show that $B_1(x)$ is mapped to zero in $\T(\G_a,\G_m^{\times n})(x)$, i.e. we have to show that for $a\in \kappa(x)^\times$ and $\ul{b}\in \G_m(x)^{\times {n-1}}$ we have 
\[\tau(a,a,\ul{b})=0 \quad \text{in } \T(\G_a,\G_m^{\times n})(x).\]
We can assume that $\kappa(x)$ contains a $(p-1)$-th primitive root of unity $\zeta$,
else we consider $\varphi: \Spec k(x)(\zeta)\to x$, which is of degree prime to $p$ and get 
\[\deg \varphi\cdot \tau(a,a,\ul{b})= \varphi_*\tau(\varphi^*a, \varphi^*a,\varphi^*\ul{b})=0.\]
Further by Remark \ref{rmk-base-change} we can also assume $\zeta$ lies in our ground field $F$.
Proposition \ref{prop-Kaehler-map-to-RFtensor} implies
\[\tau(a,a,\ul{b})=\tau(a+1,a+1,\ul{b})\quad \text{in } \T(\G_a,\G_m^{\times n})(x).\]
Applying $F\otimes \id$ we obtain
\[\tau(a^p, a,\ul{b})=\tau((a+1)^p, a+1,\ul{b}).\]
On the other hand, Proposition \ref{prop-Kaehler-map-to-RFtensor} yields
\[\tau((a+1)^p, a+1,\ul{b})= \tau((a+1)^{p-1}a, a,\ul{b}) \quad \text{in } \T(\G_a,\G_m^{\times n})(x).\]
So that all in all we obtain
\eq{cor-Kaehler-posChar1}{ \tau\left((a^{p-1}-(a+1)^{p-1})a, a,\ul{b}\right) = 0 \quad \text{in } \T(\G_a, \G_m^{\times n})(x).}
Consider $\A^1=\Spec \kappa(x)[t]\subset \P^1$. Since $(t^{p-1}-(t+1)^{p-1})a\in \G_a(\P^1, (p-1)\cdot\{\infty\})$ we have
\[\tau((t^{p-1}-(t+1)^{p-1})a, a, \ul{b})\in \T(\G_a,\G_m^{\times n})(\P^1,(p-1)\cdot\{\infty\}).\]
Further $f:= \frac{t^{p-1}-a^{p-1}}{t^{p-1}}\in \G_m(\eta_{\P^1})$ is congruent to 1 modulo $(p-1)\cdot\{\infty\}$. Thus reciprocity yields
\begin{align}
 0 & =\sum_{P\in \A^1} v_P(f) \Tr_{P/x}s_P(\tau((t^{p-1}-(t+1)^{p-1})a, a, \ul{b}))\notag\\
   & =    \sum_{i=0}^{p-1} \tau\left(((a\zeta^i)^{p-1}-(a\zeta^i+1)^{p-1})a, a, \ul{b})+(p-1)\tau(a,a,\ul{b}\right).\notag
\end{align}
Using $\zeta\in F$, Lemma \ref{lem-module-structure}, $\tau(c, \zeta^i,\ul{b})=0$ for all $c\in\G_a(x)$, all $i$, and \eqref{cor-Kaehler-posChar1} for $a\zeta^i$ we obtain
\begin{align}
0 & =\sum_{i=0}^{p-1} \zeta^{-i}\cdot \tau\left(((a\zeta^i)^{p-1}-(a\zeta^i+1)^{p-1}) \zeta^i a, \zeta^ia, \ul{b}\right)+(p-1)\tau(a,a,\ul{b})\notag\\
   & = (p-1)\tau(a,a,\ul{b}).\notag 
\end{align}
Hence $\theta$ maps $B_1$ to zero.
By definition of $B_n$, the image of $B_n$ in $\T(\G_a,\G_m^{\times n})(x)$ equals $F\otimes\id(\text{image of } B_{n-1}\text{ in } \T(\G_a,\G_m^{\times n})(x))$ and hence vanishes by induction.
This gives the statement.
\end{proof}

\begin{remark}
One would like to construct a map from $\T(\G_a,\G_m^{\times n})$ to $\Omega^n/B_\infty$ using the universal property as in the proof of Theorem \ref{thm-tensorRF=KD}.
The problem is that we don't have a description of $\Fil_P^r\G_a(\eta_C)$ in positive characteristic. In case $P$ is separable over $x_C$ this can be done (cf. \cite[Prop.6.4]{Kato-Russell}).
But the points which are inseparable over $x_C$ make trouble.
\end{remark}

\subsection{A vanishing result}

The following vanishing result was suggested by \name{B. Kahn}.

\begin{thm}\label{thm-van-unipotent-gps}
Assume ${\rm char}(\kappa(S))\neq 2$. Let $\sM_1,\ldots, \sM_n$ be $\Z$-reciprocity functors.
\[\T(\G_a,\G_a,\sM_1,\ldots,\sM_n)=0.\]
\end{thm}

\begin{proof}
It suffices to show that for any $S$-point $x$
and any elements $a,b\in\G_a(x)$, $m_i\in\sM_{i}(x)$ we have (with $\ul{m}=(m_2,\ldots,m_n)$)
\eq{thm-van-unipotent-gps1}{\tau(a,b,\ul{m})=0 \quad \text{in } \T(\G_a,\G_a,\sM_2,\ldots,\sM_n)(x).}
To show this consider $\pi: \P^1_x\to x$ and write $\P^1_x\setminus\{\infty\}=\Spec \kappa(x)[t]$. Then $\pi^*(a)\cdot t, \pi^*(b)\cdot t\in \G_a(\P^1_x, 2\{\infty\})$ and $\pi^*(m_i)\in \sM_i(\P^1)$.
Thus
\[\tau_0:=\tau(\pi^*(a)\cdot t, \pi^*(b)\cdot t, \pi^*(\ul{m}))\in \T(\G_a,\G_a,\ul{\sM})(\P^1_x, 2\{\infty\}).\]
Define $f:= t^2/(t^2-1)\in \kappa(x)(t)$, which is congruent to 1 modulo $2\{\infty\}$. Then the reciprocity law yields
\[0=\sum_{P\in \P^1_x} (\tau_0, f)_P= -\tau(a,b,\ul{m})- \tau(-a,-b,\ul{m})=-2 \tau(a,b,\ul{m}),\]
which proves \eqref{thm-van-unipotent-gps1}.
\end{proof}

\begin{remark}
In case the characteristic of $F$ is zero the above theorem gives the vanishing $\T(\sM_1,\ldots,\sM_n)=0$ whenever two of the $\sM_i$'s are smooth, commutative, connected and unipotent algebraic group, since these groups are isomorphic
to a direct sum of $\G_a$'s. We don't know whether this vanishing also holds in positive characteristic, even the case 
in which two of the $\sM_i$'s are Witt vector groups of finite length is not clear to us.
\end{remark}

\bibliographystyle{amsplain}
\bibliography{ReciprocityFunctor}
\end{document}